\documentclass{article}
\usepackage[utf8]{luainputenc}
\usepackage{cite}
\usepackage{color}
\usepackage{amscd}
\usepackage{amsmath}
\usepackage{amssymb}
\usepackage{esint}
\usepackage[unicode=true,
bookmarks=false,
breaklinks=false,pdfborder={0 0 1},backref=section,colorlinks=false]
{hyperref}

\makeatletter
\usepackage{amsfonts}
\usepackage{color}
\usepackage{graphicx}
\usepackage{cite}
\providecommand{\U}[1]{\protect\rule{.1in}{.1in}}
\newtheorem{theorem}{Theorem}
\newtheorem{acknowledgement}[theorem]{Acknowledgement}

\newtheorem{corollary}[theorem]{Corollary}

\newtheorem{definition}[theorem]{Definition}
\newtheorem{example}[theorem]{Example}

\newtheorem{fact}[theorem]{Fact}
\newtheorem{lemma}[theorem]{Lemma}

\newtheorem{notation}[theorem]{Notation}

\newtheorem{proposition}[theorem]{Proposition}
\newtheorem{remark}[theorem]{Remark}

\newtheorem{convention}[theorem]{Convention}

\newenvironment{proof}[1][Proof]{\noindent\textbf{#1.} }{\ \rule{0.5em}{0.5em}}

\newcommand{\Ep}{E_1}

\makeatother

\begin{document}
 
	\title{A finite dimensional approximation to pinned Wiener measure on some symmetric spaces}

	\author{Zhehua Li}

	\date{\today }
	
	\maketitle
	
	\numberwithin{theorem}{section} \numberwithin{equation}{section}
	
	\begin{abstract}
		Let $M$ be a Riemannian manifold, $o\in M$ be a fixed base point, $W_o\left(
		M\right)  $ be the space of continuous paths from $\left[  0,1\right]  $ to
		$M$ starting at $o\in M,$ and let $\nu_{x}$ denote Wiener measure on $W_o\left(
		M\right)  $ conditioned to end at $x\in M.$ The goal of this paper is to give
		a rigorous interpretation of the informal path integral expression for
		$\nu_{x};$
		\[
		d\nu_{x}\left(  \sigma\right)  \text{\textquotedblleft}%
		=\text{\textquotedblright}\delta_{x}\left(  \sigma\left(  1\right)  \right)
		\frac{1}{Z}e^{-\frac{1}{2}E\left(  \sigma\right)  }\mathcal{D}\sigma\text{ ,
		}\sigma\in W_o\left(  M\right)  .
		\]
		In this expression $E\left(  \sigma\right)  $ is the \textquotedblleft
		energy\textquotedblright\ of the path $\sigma,$ $\delta_{x}$ is the $\delta$
		-- function based at $x,$ $\mathcal{D}\sigma$ is interpreted as an infinite
		dimensional volume \textquotedblleft measure\textquotedblright\, and $Z$ is a
		certain \textquotedblleft normalization\textquotedblright\ constant. We will
		interpret the above path integral expression as a limit of measures,
		$\nu_{\mathcal{P},x}^{1},$ indexed by partitions, $\mathcal{P}$ of $\left[0,1\right]$. The measures
		$\nu_{\mathcal{P},x}^{1}$ are constructed by restricting the above path
		integral expression to the finite dimensional manifolds, $H_{\mathcal{P}%
			,x}\left(  M\right)  ,$ of piecewise geodesics in $W_o\left(  M\right)  $ which
		are allowed to have jumps in their derivatives at the partition points and end at $x$. The
		informal volume measure, $\mathcal{D}\sigma,$ is then taken to be a certain
		Riemannian volume measure on $H_{\mathcal{P},x}\left(  M\right)  .$ When $M$
		is a symmetric space of non--compact type, we show how to
		naturally interpret the pinning condition, i.e. the $\delta$ -- function term,
		in such a way that $\nu_{\mathcal{P},x}^{1},$ are in fact well defined finite
		measures on $H_{\mathcal{P},x}\left(  M\right)  .$ The main theorem of this paper then asserts that $\nu_{\mathcal{P},x}^{1}\rightarrow\nu_{x}$ (in
		a weak sense) as the mesh size of $\mathcal{P}$ tends to zero.
		
	\end{abstract}
	\tableofcontents{}
	
	{}

\section{Introduction\label{cha.1}}
Let $\left(\nabla\right)$ be the Levi-Civita covariant derivative on $M$ which we add to the default set up $\left(M^{d},g,\nabla ,o\right)$.
The path space
\[
W_o\left(M\right):=\left\{ \sigma\in C\left(\left[0,1\right]\mapsto M\right)\mid\sigma\left(0\right)=o\right\} 
\]
is known as the \textbf{Wiener space} on $M$ and let $\nu$ be the \textbf{Wiener measure} on
$W_o\left(M\right)$---i.e. the law of $M$--valued Brownian motion
which starts at $o\in M.$

Consider the heat equation of the following form:
\begin{equation}
\frac{\partial}{\partial t}\phi=-H\phi\text{ , }\phi\left(x,0\right)=f\left(x\right),\label{heat eqn}
\end{equation}
where $H=-\frac{1}{2}\Delta_{g}+V$ is the Schr\"odinger operator, $\Delta_{g}$ is the Laplace-Beltrami operator on $\left(M,g,o\right)$ and $V:M\to \mathbb{R}$ is an external potential.
Let $e^{-tH}$ be the solution operator to $\left(\ref{heat eqn}\right)$. Under modest regularity conditions, this operator admits an integral kernel $p_t^H\left(\cdot,\cdot\right)$. In the physics literature one frequently finds Feynman type informal identities of the form,
\begin{equation}
p_1^{H}\left(o,x\right)=\text{\textquotedblleft}\frac{1}{Z}\int_{W_o\left(M\right)}\delta_{x}\left(\sigma\left(1\right)\right)e^{-\int_{0}^{1}\left[\frac{1}{2}\left\vert \dot{\sigma}\left(\tau\right)\right\vert ^{2}+V\left(\sigma\left(\tau\right)\right)\right]d\tau}\mathcal{D}\sigma\text{\textquotedblright}\label{eq:-54}
\end{equation}
and
\begin{equation}
\left(e^{-H}f\right)\left(o\right)=\text{\textquotedblleft}\frac{1}{Z}\int_{W_o\left(M\right)}f\left(\sigma\left(1\right)\right)e^{-\int_{0}^{1}\left[\frac{1}{2}\left\vert \dot{\sigma}\left(\tau\right)\right\vert ^{2}+V\left(\sigma\left(\tau\right)\right)\right]d\tau}\mathcal{D}\sigma\text{\textquotedblright}.\label{eq:-20-1}
\end{equation}
Variants
of these informal path integrals are often used as the basis for \textquotedblleft defining\textquotedblright and making computations in quantum-field theories. From a mathematical perspective,
making sense of such path integrals is thought to be a necessary step
to developing a rigorous definition of interacting quantum field theories,
(see for example; Glimm and Jaffe \cite{Glimm1987}, Barry Simon \cite{Simon2005},
the Clay Mathematics Institute's Millennium problem involving Yang-Mills
and Mass Gap). In this paper we give an interpretation of the formal identity using a finite dimensional approximation scheme when the manifold is a symmetric space of non---compact type, for example, hyperbolic space $\mathbb{H}^d$.
\begin{equation}
\text{\textquotedblleft}\int_{W_o\left(M\right)}\delta_x\left(\sigma\left(1\right)\right)\frac{1}{Z}e^{-\frac{1}{2}\int_{0}^{1}\left\vert \dot{\sigma}\left(\tau\right)\right\vert ^{2}d\tau}\mathcal{D}\sigma\text{\textquotedblright}:= p_1\left(o,x\right)\label{eq.hk}
\end{equation}
where $p_t\left(x,y\right)$ is the heat kernel associated to $\frac{1}{2}\Delta_g$ on $M$.

\subsection{Finite Dimensional Approximation Scheme for Path Integrals\label{sec.1.2}}
The central idea behind finite dimensional approximation scheme is to define a path integral as a limit of the same integrands restricted to \textquotedblleft natural\textquotedblright\ approximate path spaces, for example, piecewise linear paths, broken lines, polygonal paths and so on. The ill--defined expression under these finite dimensional approximations usually becomes well--defined or has better interpretations, see (\cite{Fuj79}, \cite{Ichinose97}). For example, when $M=\mathbb{R}^d$, it is known that Wiener measure on $W_0\left(\mathbb{R}^{d}\right)$
may be approximated by Gaussian measures on piecewise linear path
spaces. More specifically, Eq. (\ref{eq:-20-1}) with $V=0$ and restricted
to a finite dimensional subspace of piecewise linear paths based on
a partition of $\left[0,1\right]$ has a natural interpretation as
Gaussian probability measure. The interpretation  results from the canonical isometry between the piecewise linear path space and $\mathbb{R}^{dn}$, where $n$ is the number of partition points. By combining Wiener's theorem on the
existence of Wiener measure with the dominated convergence theorem,
one can see that these Gaussian measures converge weakly to $\nu$
as the mesh of partition tends to zero, (see for example \cite[Proposition 6.17]{Driver2003} for details).
An analogous theory on general manifolds was also developed, see
for example Pinsky \cite{Pinsky1976}, Atiyah \cite{Atiyah1983}, Bismut
\cite{Bismut1985}, Andersson and Driver \cite{Andersson1999}
and references therein. In \cite{Andersson1999}, followed by \cite{Lim2007} and \cite{Laetsch2013}, the finite dimensional approximation problem is viewed in its full
geometric form by restricting the expression in Eq. (\ref{eq:-20-1})
to finite dimensional sub-manifolds of piecewise geodesic paths on
$M.$  Unlike the flat case $(M=\mathbb{R}^d)$ where the choice of translation invariant Riemannian metric on path spaces is irrelevant, various Riemannian metrics on approximate path spaces are explored. Based on these metrics, different approximate measures are constructed which lead to different limiting measures on $W_o\left(M\right)$, see \cite{Andersson1999}, \cite{Laetsch2013}, and \cite{Lim2007}. In this paper we adopt a so--called $G_{\mathcal{P}}^1$ metric on the piecewise geodesic space. 
%In \cite{Andersson1999}, the finite dimensional approximation result based on this metric is shown to agree with the classical result in Euclidean space.

In the remainder of this section, we establish some necessary notations.

\begin{definition}[Cameron-Martin space on $\left(M,g,o\right)$]\label{CM}
Let
\begin{equation*}
H\left(M\right):=\left\{ \sigma\in C\left(\left[0,1\right]\mapsto M\right):\sigma\left(0\right)=o\text{ , }\sigma\text{ is a.c. and }\int_{0}^{1}\left\vert \sigma^{\prime}\left(s\right)\right\vert _g^{2}ds<\infty\right\}
\end{equation*}be the\textbf{ Cameron-Martin space} on $\left(M,g,o\right)$. (Here a.c. means absolutely continuous.)
\end{definition}
\begin{notation}
Let $\Gamma\left(TM\right)$ be differentiable sections of $TM$ and $\Gamma_{\sigma}\left(TM\right)$ be differentiable sections of $TM$ along $\sigma\in H\left(M\right)$.
\end{notation}
The space, $H\left(M\right)$, is an infinite dimensional Hilbert manifold which is a central object in problems related to calculus of variations on $W_o(M)$. Klingenberg  \cite{Klingenberg77} contains a good exposition of the manifold of paths. In particular, Theorem 1.2.9 in \cite{Klingenberg77} presents its differentiable structure in terms of atlases. We will be interested in certain Riemannian metrics on $H\left(M\right)$ and on certain finite dimensional submanifolds where the formal path integrals make sense.
\begin{definition}
For any $\sigma\in H\left(M\right)$ and $X,Y\in \Gamma_\sigma^{a.c. }\left(TM\right)$, We define a metric $G^1$ as follows:
\[\left<X,Y\right>_{G^{1}}=\int_{0}^{1}\left\langle \frac{\nabla X}{ds}\left(s\right),\frac{\nabla Y}{ds}\left(s\right)\right\rangle _{g}ds
\]
where $\Gamma_\sigma^{a.c.}\left(TM\right)$ is the set of absolutely continuous vector fields along $\sigma$ with finite
energy, i.e. $\int_{0}^{1}\left\langle \frac{\nabla X}{ds}\left(s\right),\frac{\nabla X}{ds}\left(s\right)\right\rangle _{g}ds<\infty$.
\end{definition}
\begin{remark}
To see that $G^1$ is a metric on $H\left(M\right)$, we identify the tangent space $T_{\sigma}H\left(M\right)$ with $\Gamma_\sigma^{a.c.}\left(TM\right)$. To motivate this identification, consider a differentiable one-parameter family of curves $\sigma_t$ in $H\left(M\right)$ such that $\sigma_0=\sigma$. By definition of tangent vector, $\frac{d}{dt}\mid_0\sigma_t\left(s\right)$ should be viewed as a tangent vector at $\sigma$. This is actually the case, for detailed proof, see Theorem 1.3.1 in \cite{Klingenberg77}.  
\end{remark}
\begin{definition}[Piecewise geodesic space] \label{def.1}Given a partition \[\mathcal{P}:=\left\{ 0=s_{0}<\cdots<s_{n}=1\right\} \text{ of } \left[0,1\right],\]define: 
\begin{equation}
H_{\mathcal{P}}\left(M\right):=\left\{ \sigma\in H\left(M\right)\cap C^{2}\left(\left[0,1\right]\setminus\mathcal{P}\right):\nabla\sigma^{\prime}\left(s\right)/ds=0\text{ for }s\notin\mathcal{P}\right\}. \label{eq:-1-1-1}
\end{equation}
\end{definition}
The piecewise geodesic space $H_\mathcal{P}\left(M\right)$ is a finite dimensional embedded submanifold of $H\left(M\right)$. As for its tangent space, following the argument of Theorem 1.3.1 in \cite{Klingenberg77}, for any $\sigma\in H_{\mathcal{P}}\left(M\right)$, the tangent space
$T_{\sigma}H_{\mathcal{P}}\left(M\right)$ may be identified with
vector-fields along $\sigma$ of the form $X\left(s\right)\in T_{\sigma\left(s\right)}M$
where $s\rightarrow X\left(s\right)$ is piecewise $C^{2}$ and satisfies
Jacobi equation for $s\notin\mathcal{P}$, i.e. \[\frac{\nabla^{2}X}{ds^{2}}\left(s\right)=R\left(\dot{\sigma}\left(s\right),X\left(s\right)\right)\dot{\sigma}\left(s\right),\]where $R$ is the curvature tensor. (See Theorem \ref{def.6.2} below for a more detailed description of $TH_\mathcal{P}\left(M\right)$). After specifying the tangent space of $H_\mathcal{P}\left(M\right)$, we can define the $G_\mathcal{P}^1$ metric as follows.
\begin{definition}For any $\sigma\in H_{\mathcal{P}}\left(M\right)$
and $X,Y\in T_{\sigma}H_{\mathcal{P}}\left(M\right),$ let 
\begin{equation}
\left\langle X,Y\right\rangle_{G_{\mathcal{P}}^{1}} :=\sum_{j=1}^{n}\left\langle \frac{\nabla X}{ds}\left(s_{j-1}+\right),\frac{\nabla Y}{ds}\left(s_{j-1}+\right)\right\rangle_g \Delta_{j}\label{eq:-4-1-1}
\end{equation}
where $\Delta_{j}=s_{j}-s_{j-1}$ and $\frac{\nabla Y}{ds}\left(s_{j-1}+\right)=\lim_{s\downarrow s_{j-1}}\frac{\nabla Y}{ds}\left(s\right)$. \end{definition}

Endowed with the Riemannian metric $G_{\mathcal{P}}^{1}$, $H_\mathcal{P}\left(M\right)$ becomes a finite dimensional Riemannian manifold and the right hand side of $\left(\ref{eq:-20-1}\right)$ is now well--defined on $H_\mathcal{P}\left(M\right)$ if $\mathcal{D}\sigma$ is interpreted as the volume measure induced from this Riemannian metric. This motivates the following approximate measure definition.
\begin{definition}[Approximate measure on $H_{\mathcal{P}}\left(M\right)$]\label{def.2}Let
$\nu_{\mathcal{P}}^{1}$ be the probability measure on $H_{\mathcal{P}}\left(M\right)$
defined by; 
\begin{equation}
d\nu_{\mathcal{P}}^{1}\left(\sigma\right)=\frac{1}{Z_{\mathcal{P}}^{1}}e^{-\frac{1}{2}\int_{0}^{1}\left\langle \sigma^{\prime}\left(s\right),\sigma^{\prime}\left(s\right)\right\rangle ds}dvol_{G_{\mathcal{P}}^{1}}\left(\sigma\right),\label{eq:-6-1-1}
\end{equation}
where $dvol_{G_{\mathcal{P}}^{1}}$ is the volume measure on $H_{\mathcal{P}}\left(M\right)$
induced from the metric $G_{\mathcal{P}}^{1}$ and $Z_{\mathcal{P}}^{1}$
is the normalization constant. \end{definition}

\subsection{Main Theorems\label{sec.1.4}}
In this section we state the main results of this paper while avoiding many technical details.

\begin{definition}[Pinned piecewise geodesic space]\label{not1}For any $x\in M$,
\[
H_{\mathcal{P},x}\left(M\right):=\left\{ \sigma\in H_{\mathcal{P}}\left(M\right):\sigma\left(1\right)=x\right\}. 
\]
\end{definition}
We prove below in Proposition \ref{pro.3.2.1} that when $M$ has non--positive sectional curvature, $H_{\mathcal{P},x}\left(M\right)$ is an embedded submanifold of $H_\mathcal{P}\left(M\right)$. 
\begin{theorem}\label{thm1} If $M$ is a Hadamard manifold with
bounded sectional curvature and $\mathcal{P}=\left\{ k/n\right\} _{k=0}^{n}$
are equally-spaced partitions, then there exists a finite measure $\nu_{\mathcal{P},x}^{1}$ supported
on $H_{\mathcal{P},x}\left(M\right),$ such that for any bounded continuous function $f$ on $H_\mathcal{P}\left(M\right)$, 
\[
\lim_{m\rightarrow\infty}\int_{H_{\mathcal{P}}\left(M\right)}\delta_{x}^{\left(m\right)}\left(\sigma\left(1\right)\right)f\left(\sigma\right)d\nu_{\mathcal{P}}^{1}\left(\sigma\right)=\int_{H_{\mathcal{P}}\left(M\right)}f\left(\sigma\right)d\nu_{\mathcal{P},x}^{1}\left(\sigma\right).
\]
where $\delta_{x}^{\left(m\right)}$ is an approximate sequence
of $\delta_{x}$ in $C_{0}^{\infty}\left(M\right)$. \end{theorem}
\begin{remark}
The formula for $d\nu_{\mathcal{P},x}^1$ is explicitly given, see Definition \ref{def.ab}.
\end{remark}

%Theorem \ref{thm1} can be viewed as a finite dimensional version of $\left(\ref{eq.hk}\right)$. A rigorous theory explaning $\left(\ref{eq.hk}\right)$ is Watanabe's theory of generalized Wiener functionals. In \cite{Watanabe1988}, Watanabe considers the following expression
%\[
%\mathbb{E}\left[\delta_{x}\circ \Ep\cdot \Phi\right]
%\]
%where $\Ep$ is the \textbf{end
%point evaluation map}, i.e. for any $\sigma\in W_o\left(M\right)$, 
%$\Ep\left(\sigma\right)=\sigma\left(1\right)$ and $\Phi$ are some \textquotedblleft nice \textquotedblright Wiener functionals (test functions). As was shown by Airault--Malliavin \cite{Airault88} and Sugita \cite{Sugita88}, if $M=\mathbb{R}^d$ is a Euclidean space, there exists a modification of $\Phi$, called quasi--continuous modification (denoted by $\tilde{\Phi}$), such that the following identity holds:
%\[\mathbb{E}_{\nu_x}\left[\tilde{\Phi}\right]=\mathbb{E}_{\nu}\left[\delta_{x}\circ \Ep \cdot \Phi\label{1}\right].
%\]
%The point of this theorem is that it represents a generalized Wiener functional $\delta_{x}\circ \Ep$ as a measure $\nu_x$ supported on a "hypersurface" $\mathcal{S}_x:=\left\{\sigma\in W\left(\mathbb{R}^d\right):\Ep\left(\sigma\right)=x \right\}$. Theorem \ref{thm1}  represents $\delta_x\circ \Ep$ as a measure $\nu_{\mathcal{P},x}^1$ (See Definition \ref{def.ab}) in the "hypersurface" $H_{\mathcal{P},x}\left(M\right)$, which can be viewed as a finite dimensional analog of Identity \ref{1}. 

The next theorem asserts, under additional geometric restrictions, that the
measure $\nu_{\mathcal{P},x}^{1}$ we obtained from Theorem \ref{thm1}
serves as a good approximation to pinned Wiener measure $\nu_x$.

\begin{theorem}\label{thm2}If $M$ is a symmetric space of non--compact type, i.e. it is a Hadamard manifold with parallel curvature tensor, then for any cylinder function $f\in \mathcal{FC}_b^1$, see Definition \ref{def.rcf},
\[
\lim_{\left\vert \mathcal{P}\right\vert \rightarrow0}\int_{H_{\mathcal{P}}\left(M\right)}f\left(\sigma\right)d\nu_{\mathcal{P},x}^{1}\left(\sigma\right)=\int_{W_o\left(M\right)}f\left(\sigma\right)d\nu_{x}\left(\sigma\right)
\]
where $\nu_x$ is pinned Wiener measure, see Theorem \ref{thm Brown} below. \end{theorem}

\subsection{Structure of the Paper\label{sec.1.5}}
For the guidance to the reader, we give a brief summary of the contents of this paper. 

In Section \ref{cha.2} we set up some notations and preliminaries
in probability and geometry. In particular we present the Eells-Elworthy-Malliavin construction of Brownian motion on manifolds. 

In  Section \ref{cha.3} we define explicitly the pinned approximate measure $\nu_{\mathcal{P},x}^{1}$ and study its properties. In Theorem \ref{thm3.2.1}, we prove that $\nu_{\mathcal{P},x}^{1}$ is a finite measure and that $x\to \int_{H_{\mathcal{P},x}\left(M\right)} fd\nu_{\mathcal{P},x}^{1}$ is a continuous function on $M$ provided $f$ is bounded and continuous. This property is the key ingredient in proving Theorem \ref{thm1}. The proof of Theorem \ref{thm1} is also given in this section.

In Section \ref{cha.4} we develop the so--called orthogonal lift of a vector field $X$ on $M$ to vector fields $\tilde{X}_{\mathcal{P}}$ on $H_\mathcal{P}(M)$ and $\tilde{X}$ on $W_o(M)$. Integration by parts formulae for these two operators are presented which will serve as an important tool in the proof of Theorem \ref{thm2}.

\iffalse
In Section \ref{cha.4} we develop the so--called orthogonal lift of a vector field $X$ on $M$ to a vector field $\tilde{X}\left(\cdot\right)$ on $W_o\left(M\right)$. We define $\tilde{X}\left(\cdot\right)$ first on $H\left(M\right)$ by minimizing a norm of $\tilde{X}\left(\cdot\right)$ which is induced from a \textquotedblleft damped\textquotedblright  metric related to the Ricci curvature of $M$ (see Definition \ref{def.4.1}). This lift is then \textquotedblleft stochastically\textquotedblright  extended to $W_o\left(M\right)$. Some tools from Malliavin calculus are reviewed as needed in order to define $\tilde{X}\left(\cdot\right)$ as an anticipating differential opearator on $W_o\left(M\right)$. We then establish integration--by--parts formula for $\tilde{X}\left(\cdot\right)$.

In Section \ref{cha.5} we focus on the finite dimensional manifold $H_\mathcal{P}\left(M\right)$. In Section \ref{sec.5.1} a parametrization of the tangent space of $H_\mathcal{P}\left(M\right)$ is given. Using this parametrization and some linear algebra we obtain a formula for the orthogonal lift $\tilde{X}_\mathcal{P}$ of $X\in \Gamma\left(TM\right)$ relative to the norm induced from the $G_\mathcal{P}^1$ metric on $H_\mathcal{P}\left(M\right)$.
\fi
In Section \ref{cha.6}, (using the development maps introduced in Section \ref{cha.2}), we view $\tilde{X}_\mathcal{P}$ as defined on all of $W_o\left(M\right)$ and show that for any bounded cylinder function $f$ (also introduced in Section \ref{cha.2}), $\left\Vert\tilde{X}_\mathcal{P}f-\tilde{X}f\right\Vert_{L^q(W_o(M))}\to 0$ as $\left\vert\mathcal{P}\right\vert\to 0$ for any $q\geq 1$ and more challengingly, we show the same result for  $\tilde{X}^{tr,\nu}f-\tilde{X}_\mathcal{P}^{tr,\nu_\mathcal{P}^1}f$, where  $\tilde{X}^{tr,\nu}$ is the adjoint of $\tilde{X}$ with respect to $\nu$ and $\tilde{X}_\mathcal{P}^{tr,\nu_\mathcal{P}^1}$ is the adjoint of $\tilde{X}_\mathcal{P}$ with respect to $\nu_\mathcal{P}^1$.

In Section \ref{cha.7}, we combine all the tools that are developed from previous sections to prove the main Theorem \ref{thm2} of this paper.

\begin{acknowledgement}
	I want to thank my advisor Bruce Driver for many meaningful discussions and careful reading of my dissertation whose main part gives rise to this paper.  
\end{acknowledgement}

\section{Preliminaries in Geometry and Probability\label{cha.2}}
For the remainder of this paper, let $u_0:\mathbb{R}^d\to T_oM$ be a fixed linear isometry which we add to the standard setup $\left(M,g,o,u_0, \nabla\right)$. We use $u_0$ to identify $T_oM$ with $\mathbb{R}^d$. Suggested references for this section are Chapter 2 of \cite{Hsu01} and Sections 2, 3 of \cite{Driver1992}. Some other references are \cite{Andersson1999},  \cite{Elworthy82}, \cite{Cruzeiro1996} and \cite{Driver1995b} to name just a few.

\begin{definition}[Orthonormal Frame Bundle $\left(\mathcal{O}\left(M\right),\pi\right)$]\label{example2}

For any $x\in M$, denote by $\mathcal{O}\left(M\right)_{x}$ the
space of orthonormal frames on $T_{x}M$, i.e. the space
of linear isometries from $\mathbb{R}^{d}$ to $T_{x}M$. Denote $\mathcal{O}\left(M\right):=\cup_{x\in M}\mathcal{O}\left(M\right)_{x}$
and let $\pi:\mathcal{O}\left(M\right)\to M$ be the $\left(\text{fiber}\right)$ projection
map, i.e. for each $u\in\mathcal{O}\left(M\right)_{x}$, $\pi\left(u\right)=x$.
The pair $\left(\mathcal{O}\left(M\right),\pi\right)$ is the orthonormal
frame bundle over $M$.
\end{definition}
In this paper we use the connection on $\mathcal{O}(M)$ that are specified by the following connection form.
\begin{definition}[Connection Form on $\mathcal{O}\left(M\right)$]\label{con}We define a $\mathfrak{so}\left(d\right)$--valued connection form $\omega^\nabla$ on $\mathcal{O}\left(M\right)$ in the following way; for any $u\in \mathcal{O}\left(M\right)$ and $X\in T_u\mathcal{O}\left(M\right)$,
\[\omega_u^\nabla\left(X\right):=u^{-1}\frac{\nabla u\left(s\right)}{ds}\mid_{s=0}\]
where $u\left(\cdot\right)$ is a differentiable curve on $\mathcal{O}\left(M\right)$ such that $u\left(0\right)=u$ and $\frac{du\left(s\right)}{ds}\mid_{s=0}=X$. For any $\xi\in \mathbb{R}^d$, $\frac{\nabla u\left(s\right)}{ds}\mid_{s=0}\xi:=\frac{\nabla u\left(s\right)\xi}{ds}\mid_{s=0}$ is the covariant derivative of $u\left(\cdot\right)\xi$ along $\pi\left(u\left(\cdot\right)\right)$ at $\pi\left(u\right)$.   
\end{definition}
\begin{definition}[Horizontal Bundle $\mathcal{H}$]
	Given a connection form $\omega^\nabla$, the horizontal bundle $\mathcal{H}\subset T\mathcal{O}(M)$ is defined to be the kernel of $\omega^\nabla$.
\end{definition}
\begin{definition}\label{def.hv}
	For any $a\in \mathbb{R}^d$, define the horizontal lift $B_a\in \Gamma\left(\mathcal{H}\right)$ in the following way: for any $u\in \mathcal{O}\left(M\right)$, $B_a(u)\in \mathcal{H}_u\subset T_u\mathcal{O}(M)$ is uniquely determined by
	\[\omega_u^\nabla\left(B_a\left(u\right)\right)=0\text{  and  }\pi_*\left(B_a\left(u\right)\right)=ua.\]
\end{definition}
\begin{definition}[Horizontal Lift of a Path]
	For any $\sigma\in H\left(M\right)$, a curve
	$u:\left[0,1\right]\to\mathcal{O}\left(M\right)$ is said to be a
	horizontal lift of $\sigma$ if $\pi\circ u=\sigma$ and $u^\prime\left(s\right)\in \mathcal{H}_{u(s)}\text{ }\forall s\in [0,1]$.
\end{definition}
\begin{remark}
	In this paper we only consider horizontal lift with fixed start point $u_0\in \pi^{-1}\left(\sigma\left(0\right)\right)$. Under this assumption, given $\sigma\in H(M)$, its horizontal lift $u(\sigma,\cdot)$ is unique.
\end{remark}
We denote $u$ by $\psi\left(\sigma\right)$ and call $\psi$ the \textbf{horizontal lift map}.
\begin{definition}[Development Map]\label{def-d}Given
	$w\in H\left(\mathbb{R}^{d}\right)$, the solution to the ordinary
	differential equation 
	\[
	du\left(s\right)=\sum_{i=1}^{d}B_{e_i}\left(u\left(s\right)\right)dw^{i}\left(s\right),u\left(0\right)=u_{0}
	\]
	is defined to be the \textbf{development} of $w$ and we will denote this map $w\to u$ by $\eta$, i.e. $\eta\left(w\right)=u$. Here $\{e_i\}_{i=1}^d$ is the standard basis of $\mathbb{R}^d$.
	
\end{definition}
\begin{definition}[Rolling Map]
	
	$\phi=\pi\circ\eta:H\left(\mathbb{R}^{d}\right)\to H\left(M\right)$
	is said to be the rolling map to $H\left(M\right)$.
\end{definition}
\begin{definition}[Anti-rolling Map]
	
	Given $\sigma\in H\left(M\right)$ with $u=\psi\left(\sigma\right).$
	The anti-rolling of $\sigma$ is a curve $w\in H\left(\mathbb{R}^{d}\right)$
	defined by:
	
	\[
	w_{t}=\int_{0}^{t}u_{s}^{-1}\sigma_{s}^{\prime}ds
	\]
\end{definition}
\begin{remark}
	It is not hard to see $w=\phi^{-1}\left(\sigma\right)$ and $u(\sigma,s)u_0^{-1}$ is the parallel translation along $\sigma\in H(M)$.	
\end{remark}
The Eells-Elworthy-Malliavin construction of Brownian motion depends in essence on a stochastic version of the maps defined above. Since the
development maps on the smooth category are defined through ordinary
differential equations, a natural way to introduce probability is to replace
ODEs by (Stratonovich) stochastic differential equations.

First we set up some measure theoretic notations and conventions. Suppose
that $\left(\Omega,\left\{ \mathcal{G}_{s}\right\} ,\mathcal{G},P\right)$
is a filtered measurable space with a finite measure $P$. For any
$\mathcal{G}$---measurable function $f$, we use $P\left(f\right)$
and $\mathbb{E}_{P}\left[f\right]$ (if $P$ is a probability measure)
to denote the integral $\int_{\Omega}fdP$. Given two filtered measurable spaces
$\left(\Omega,\left\{ \mathcal{G}_{s}\right\} ,\mathcal{G},P\right)$
and $\left(\Omega^{\prime},\left\{ \mathcal{G}_{s}^{\prime}\right\} ,\mathcal{G}^{\prime},P^{\prime}\right)$
and a $\mathcal{G}/\mathcal{G}^{\prime}$ measurable map $f:\Omega\to\Omega^{\prime}$,
the law of $f$ under $P$ is the push-forward measure $f_{*}P\left(\cdot\right):=P\left(f^{-1}\left(\cdot\right)\right)$.
We are mostly interested in the path spaces $W_{o}\left(M\right)$,
$W_{0}\left(\mathbb{R}^{d}\right)$ and $W_{u_{0}}\left(\mathcal{O}\left(M\right)\right)$,
where the following notation is being used.

\begin{notation} If $\left(Y,y\right)$ is a pointed manifold, let $W\left(Y\right):=C\left(\left[0,1\right],Y\right)$ be the space
	of all continuous paths in $Y$ equipped with the uniform topology, $W_{y}\left(Y\right):=\left\{ w\in W\left(Y\right)\mid w\left(0\right)=y\right\} $
	be the subset of continuous paths that start at $y.$
	
\end{notation}
\begin{definition}
	For any $s\in\left[0,1\right]$ let $\Sigma_{s}:W_{y}\left(Y\right)\to Y$
	be the \textbf{coordinate functions} given by $\Sigma_{s}\left(\sigma\right)=\sigma\left(s\right)$.
\end{definition}
We will often view $\Sigma$ as a map from $W_{y}\left(Y\right)\text{ to }W_{y}\left(Y\right)$
in the following way: for any $\sigma\in W_{y}\left(Y\right)$ and
$s\in\left[0,1\right]$, $\Sigma\left(\sigma\right)\left(s\right)=\Sigma_{s}\left(\sigma\right)$.
Let $\mathcal{F}_{s}^{o}$ be the $\sigma-$algebra generated by $\left\{ \Sigma_{\tau}:\tau\leq s\right\} $.
We use $\mathcal{F}_{1}^{o}$ as the raw $\sigma-$algebra and $\left\{ \mathcal{F}_{s}^{o}\right\} _{0\leq s\leq1}$
as the filtration on $W_{y}\left(Y\right).$ The next theorem defines
the Wiener measure $\nu$ and pinned Wiener measure $\nu_{x}$ on
$\left(W_{y}\left(Y\right),\mathcal{F}_{1}^{o}\right).$

\begin{theorem}\label{thm Brown}
	Assume $Y$ is a stochastically complete Riemannian manifold, then there exist two finite measures $\nu$ and $\nu_{x}$ on $\left(W_{y}\left(Y\right),\mathcal{F}_{1}^{o}\right)$
	which are uniquely determined by their finite dimensional distributions
	as follows. For any partition $0=s_{0}<s_{1}<\cdots<s_{n-1}<s_{n}=1$
	of $\left[0,1\right]$ and bounded functions $f:Y^{n}\to\mathbb{R};$
	
	\begin{equation}
	\nu\left(f\left(\Sigma_{s_{1}},\dots,\Sigma_{s_{n}}\right)\right)=\int_{Y^{n}}f\left(x_{1},\dots,x_{n}\right)\Pi_{i=1}^{n}p_{\Delta s_{i}}\left(x_{i-1},x_{i}\right)dx_{1}\cdots dx_{n}\label{eq:-36}
	\end{equation}
	and 
	\begin{equation}
	\nu_{x}\left(f\left(\Sigma_{s_{1}},\dots,\Sigma_{s_{n}}\right)\right)=\int_{Y^{n-1}}f\left(x_{1},\dots,x_n\right)\Pi_{i=1}^{n}p_{\Delta s_{i}}\left(x_{i-1},x_{i}\right)dx_{1}\cdots dx_{n-1}\label{eq:-37}
	\end{equation}
	where $p_{t}\left(\cdot,\cdot\right)$ is the heat kernel on $Y$ associated to $\frac{1}{2}\Delta_g$,
	$\Delta_{i}=s_{i}-s_{i-1}$, $x_{0}\equiv o$ and $x_{n}\equiv x$
	in $\left(\ref{eq:-37}\right)$. 
\end{theorem} 
 
\begin{definition}[Brownian motion] A stochastic process $X:\left(\Omega,\mathcal{G}_{s},\left\{ \mathcal{G}\right\} ,P\right)$$ \to\left(W_{y}\left(Y\right),\nu\right)$
	is said to be a Brownian motion on $Y$ if the law of $X$
	is $\nu$ i.e. $X_{*}P:=P\circ X^{-1}=\nu$. \end{definition}
\begin{remark} From Theorem \ref{thm Brown} it is clear that the law
	of the adapted process $\Sigma:W_{y}\left(Y\right)\to W_{y}\left(Y\right)$
	is $\nu$ and $\Sigma$ is a Brownian motion. We will call $\Sigma$ the \textbf{canonical Brownian motion} on $Y$. 
\end{remark}
\begin{remark} Using Theorem \ref{thm Brown}, we can construct Wiener
	measure and pinned Wiener measure on $W_{0}\left(\mathbb{R}^{d}\right)$,
	$W_{o}\left(M\right)$ and $W_{u_{0}}\left(\mathcal{O}\left(M\right)\right)$
	respectively. In order to avoid ambiguity from moving between $W_{0}\left(\mathbb{R}^{d}\right)$
	and $W_{o}\left(M\right)$, we fix the symbol $\mu\left(\mu_{x}\right)$
	as the Wiener $\left(\text{pinned Wiener}\right)$ measure on $W_{0}\left(\mathbb{R}^{d}\right)$
	and reserve the symbol $\nu\left(\nu_{x}\right)$ as the Wiener $\left(\text{pinned Wiener}\right)$ measure on $W_{o}\left(M\right)$. Meanwhile we reserve $\Sigma$
	as the canonical Brownian motion on $M$. 
\end{remark}
\begin{theorem}[Stochastic Horizontal Lift of Brownian Motion]\label{thm sHL}
	If $\Sigma$ is the canonical Brownian motion on $M$, then there exists a unique $(\text{up to  }\nu-\text{equivalence})$ $\tilde{u}\in W_{u_0}\left(\mathcal{O}\left(M\right)\right)$ such that 
	\begin{equation}
	\pi\left(\tilde{u}_s\right)=\Sigma_s.
	\end{equation}
\end{theorem}
\begin{proof}
	See Theorem 2.3.5 in \cite{Hsu01}
\end{proof}
\begin{definition}[Stochastic Anti--rolling Map]\label{def sar}
	If $\Sigma$ is the canonical Brownian motion on $M$, then the stochastic anti--rolling $\beta$ of $\Sigma$ is defined by,
	\begin{equation}
	d\beta_s=\tilde{u}^{-1}_s\delta\Sigma_s\text{ , }\beta_0=0.
	\end{equation}
\end{definition} 
$\tilde{u}$ and $\beta$ defined above are linked through the (stochastic) development map.
\begin{definition}[Stochastic Development Map]\label{def sd}
	Let $\tilde{u}$ and $\beta$ be as defined in Theorem \ref{thm sHL} and Definition $\ref{def sar}$, then $\tilde{u}$ satisfies the following SDE driven by $\beta$,
	\[
	d\tilde{u}_{s}=\sum_{i=1}^{d}B_{e_i}\left(\tilde{u}_{s}\right)\delta \beta_{s}\text{ , }\tilde{u}\left(0\right)=u_{0},
	\]
	and $\tilde{u}$ is said to be the development of $\beta$. 
\end{definition} 
\begin{fact}\label{fact 1} 
	The following facts are well known, the proofs may be found in the references listed at the begining of this section, for example, Theorem 3.3 in \cite{Driver1992}. 
	\begin{itemize}
		\item $\phi$ is a diffeomorphism from $H\left(\mathbb{R}^{d}\right)$ to
		$H\left(M\right),$ 
		\item $\beta$ is a Brownian motion on $\left(W_o\left(\mathbb{R}^d\right), \mu\right)$.
	\end{itemize}
\end{fact}
From now on some notations are fixed for the convenience of consistency.
\begin{notation} \label{not2} For any $\sigma\in H\left(M\right)$, $u_{\left(\cdot\right)}\left(\sigma\right)\in H_{u_0}\left(\mathcal{O}\left(M\right)\right)$ is its horizontal lift and $b_{\left(\cdot\right)}\left(\sigma\right)\in H\left(\mathbb{R}^d\right)$ is its anti-rolling. Recall that $\left\{ \Sigma_s\right\}_{0\leq s\leq 1} $ is fixed to be the canonical Brownian motion on $\left(W_o\left(M\right),\nu \right)$. We also fix $\beta\left(\cdot\right)$ to be the stochastic anti-rolling of $\Sigma$, $($which is a Brownian motion on $\mathbb{R}^d)$ and $\tilde{u}\left(\cdot\right)$ to be the stochastic horizontal lift of $\Sigma$. \end{notation}
\begin{notation}\label{not2.26}Given a partition $\mathcal{P}$ of $[0,1]$, $\beta_{\mathcal{P}}$ is the piecewise linear approximation to the Brownian motion $\beta$
	on $\mathbb{R}^{d}$ given by: 
	\[
	\beta_{\mathcal{P}}\left(s\right):=\beta\left(s_{i-1}\right)+\frac{\Delta_{i}\beta}{\Delta_{i}}\left(s-s_{i-1}\right)\text{ if }s\in\left[s_{i-1},s_{i}\right]
	\]
	where $\Delta_{i}\beta=\beta\left(s_{i}\right)-\beta\left(s_{i-1}\right)$and
	$\Delta_{i}=s_{i}-s_{i-1}$. \end{notation}
\begin{notation}[Geometric Notation]\label{Geo}\text{ } 
	\begin{itemize}
		\item \textbf{curvature tensor} For any $X,Y,Z\in\Gamma\left(TM\right),$ define
		the $(\text{Riemann})$ curvature tensor $R:\Gamma\left(TM\right)\times\Gamma\left(TM\right)\to \Gamma\left(End\left(TM\right)\right)$
		to be: 
		\[
		R\left(X,Y\right)Z =\nabla_{X}\nabla_{Y}Z-\nabla_{Y}\nabla_{X}Z-\nabla_{\left[X,Y\right]}Z
		\]
		\item For any $\sigma\in H\left(M\right)$, define $R_{u\left(\sigma,s\right)}\left(\cdot,\cdot\right)\cdot$ to be a map from $\mathbb{R}^{d}\otimes\mathbb{R}^{d}$ to $End\left(\mathbb{R}^{d}\right)$ given by;
		\begin{equation}
		R_{u\left(\sigma,s\right)}\left(a,b\right)\cdot={u\left(\sigma,s\right)}^{-1}R\left(u\left(\sigma,s\right)a,u\left(\sigma,s\right)b\right)u\left(\sigma,s\right)\text{  }\forall a,b\in \mathbb{R}^{d}\label{n1}
		\end{equation} 
		where $R$ is the curvature tensor of $M$.
		Similarly we define $R_{\tilde{u}\left(\sigma,s\right)}\left(\cdot,\cdot\right)\cdot$ to be a random map $($up to $\nu$-equivalence$)$ from $\mathbb{R}^{d}\otimes\mathbb{R}^{d}$ to $\mathbb{R}^{d}$ as follows:
		\begin{equation}
		R_{\tilde{u}\left(\sigma,s\right)}\left(\cdot,\cdot\right)\cdot={\tilde{u}\left(\sigma,s\right)}^{-1}R\left(\tilde{u}\left(\sigma,s\right)\cdot,\tilde{u}\left(\sigma,s\right)\cdot\right)\tilde{u}\left(\sigma,s\right)\label{n2}
		\end{equation}
		\item $Ric\left(\cdot\right):=\sum_{i=1}^{d}R\left(v_{i},\cdot\right)v_{i}$
		is the Ricci curvature tensor on $M.$ Here $\left\{ v_{i}\right\} _{i=1}^{d}$
		is an orthonormal basis of proper tangent space. Using $u\left(\sigma,s\right)$ or $\tilde{u}\left(\sigma,s\right)$ to pull back $R$ as in $(\ref{n1})$ and $(\ref{n2})$, we can define $Ric_{u\left(\sigma,s\right)}$ and $Ric_{\tilde{u}\left(\sigma,s\right)}$ to be maps $(\text{random maps})$ from $\mathbb{R}^d$ to $\mathbb{R}^d$.
	\end{itemize}
\end{notation}
\begin{convention}
	Since most of our results require a curvature bound, it would be convenient to fix a symbol $N$ for it, i.e. $\left\Vert R\right\Vert\leq N$ when it is viewed as a tensor of order 4. Following this manner, we have $\left\Vert Ric \right\Vert\leq (d-1)N$. A generic constant will be denoted by $C$, it can vary from line to line. Sometimes $C_{(\cdot)}$ or $C(\cdot)$ are used to specify its dependence on some parameters.
\end{convention}
\begin{definition}\label{def.rcf} $f:W_o\left(M\right)\mapsto\mathbb{R}$ is a \textbf{cylinder function} if there exists a partition 
	\[
	\mathcal{P}:=\left\{ 0<s_{1}<\cdots<s_{n}\leq1\right\} 
	\]
	of $\left[0,1\right]$ and a function $F:C^m\left(M^{n},\mathbb{R}\right)$
	such that
	\[
	f=F\left(\Sigma_{s_{1}},\Sigma_{s_{2}},\dots,\Sigma_{s_{n}}\right).
	\]
	We denote this space by $\mathcal{FC}^m$. \end{definition}
\begin{notation}\label{not fspace}
	Denote 
	\[\mathcal{FC}^1_{b}:=\left\{f:=F\left(\Sigma\right)\in \mathcal{FC}^{1}, F\text{ and all its partial differentials }grad_iF\text{ are bounded}\right\}.\]
\end{notation}
\begin{remark} In this paper the partition $\mathcal{P}$ is always
equally spaced, so $\left\vert \mathcal{P}\right\vert \equiv\Delta_{i}\equiv\frac{1}{n}$
for $i=1,...,n$. \end{remark}
%\begin{definition} \label{def.8.8}$f:H_\mathcal{P}\left(M\right)\to\mathbb{R}$
%is called a cylinder function on $H_\mathcal{P}\left(M\right)$ if there exists a partition 
%\[
%\mathcal{P}:=\left\{ 0<s_{1}<\cdots<s_{n}\leq1\right\} 
%\]
%of $\left[0,1\right]$ and a function $F\in C^m\left(\mathcal{O}\left(M\right)^{n}, \mathbb{R}\right)\text{   }\forall 1\leq m\leq \infty$ 
%such that: 
%\[
%f\left(\sigma\right)=F\left({u}_{s_{1}}\left(\sigma\right),{u}_{s_{2}}\left(\sigma\right),\dots,{u}_{s_{n}}\left(\sigma\right)\right)
%\]
%Denote this space by $\mathcal{FC}^{m}_\mathcal{P}$. \end{definition}
%\begin{notation}\label{not fpspace}
%For each $m\in \left[1,\infty\right]$, denote 
%\[\mathcal{FC}^m_{\mathcal{P},b}:=\left\{f:=F\left(u\right)\in \mathcal{FC}^{m}_\mathcal{P}, F\text{ and its derivatives up to order m are bounded}\right\}.\]
%\end{notation}

\begin{definition}[Jacobi equation]\label{def.6.1}For $\sigma\in H\left(M\right)$, $Y\in \Gamma_\sigma\left(TM\right)$,
we say $Y\left(s\right)\in T_{\sigma\left(s\right)}M$ satisfies Jacobi
equation if: 
\[
\text{\ensuremath{\frac{\nabla^{2}}{ds^{2}}}}Y\text{\ensuremath{\left(s\right)}}=R\text{\ensuremath{\left(\sigma^{\prime}\left(s\right),Y\left(s\right)\right)\sigma^{\prime}\left(s\right).}}
\]
Further if the horizontal lift $u\left(s\right)$ of $\sigma$ is used, 
we let $y\left(s\right):=u^{-1}\left(s\right)Y\left(s\right).$ It
then follows that $y\left(s\right)$ satisfies the pulled back Jacobi
equation, 
\begin{equation}
y^{\prime\prime}\left(s\right)=R_{u\left(s\right)}\left(b^{\prime}\left(s\right),y\left(s\right)\right)b^{\prime}\left(s\right),\label{equ.6.1}
\end{equation}
where $b^{\prime}\left(s\right)=u\left(s\right)^{-1}\sigma^{\prime}\left(s\right).$
Once we have Jacobi equation, we can describe the tangent space $TH_{\mathcal{P}}\left(M\right)$
of $H_{\mathcal{P}}\left(M\right)$.\end{definition}
We formalize the tangent space of $H_\mathcal{P}\left(M\right)$ mentioned in Definition $\ref{def.1}$.
\begin{theorem}[Tangent space to $H_{\mathcal{P}}\left(M\right)$]\label{def.6.2}For
all $\sigma\in H_{\mathcal{P}}\left(M\right)$, 
\begin{equation}
T_{\sigma}H_{\mathcal{P}}\left(M\right)=\left\{s\to u\left(s\right)J\left(s\right)|\text{ }J\in C\left(\left[0,1\right],\mathbb{R}^{d}\right)\text{, }J\in H_{\mathcal{P},\sigma}\text{ with }J\left(0\right)=0\right\} .\label{equ.6.2}
\end{equation}
where $J\in H_{\mathcal{P},\sigma}$ iff
\[
J^{\prime\prime}\left(s\right)=R_{u\left(s\right)}\left(b^{\prime}\left(s_{i-1}+\right),J\left(s\right)\right)b^{\prime}\left(s_{i-1}+\right)\ \text{for}\ s\in\lbrack s_{i-1},s_{i})\text{ }i=1,...,n.
\]
\end{theorem}
\begin{proof}
See Theorem 1.3.1 in \cite{Klingenberg77}.  
\end{proof}
\begin{notation}\label{pvf}
Given $h\left(\cdot\right)\in H\left(\mathbb{R}^d\right)$, denote $X^h\left(\sigma,s\right):=u\left(\sigma,s\right)h\left(s\right).$
\end{notation}
\begin{notation}[$\left\{ C_{\mathcal{P},i}\left(\sigma,s\right)\right\} {}_{i=1}^{n}$
and $\left\{ S_{\mathcal{P},i}\left(\sigma,s\right)\right\} _{i=1}^{n}$]\label{not.2.1}Let
\[\mathcal{P}:=\left\{ 0=s_{0}<s_{1}<\dots<s_{n}=1\right\}\] be a
partition of $\left[0,1\right],$ $K_{i}:=\left[s_{i-1},s_{i}\right]$
and $\Delta_{i}:=s_{i}-s_{i-1}$ for $1\leq i\leq n,$ and say that
$f\left(s\right)$ satisfies the $i$ --Jacobi's equation if 
\begin{equation}
f^{\prime\prime}\left(s\right)=R_{u_s}\left(u^{-1}\sigma^{\prime}\left(s_{i-1}+\right),f\left(s\right)\right)u^{-1}\sigma^{\prime}\left(s_{i-1}+\right)\text{for}\ s\in K_{i}.\label{equ.6.3}
\end{equation}
where $u^{-1}\sigma^{\prime}\left(s\right):=u\left(\sigma,s\right)^{-1}\sigma^{\prime}\left(s\right)\in\mathbb{R}^{d}$.

We now let $C_{\mathcal{P},i}\left(\sigma,s\right)$ and $S_{\mathcal{P},i}\left(\sigma,s\right)\in\operatorname*{End}(\mathbb{R}^{d})$
denote the solution to Eq. (\ref{equ.6.3}) with initial conditions,
\[
C_{\mathcal{P},i}\left(s_{i-1}\right)=I,\text{ }C_{\mathcal{P},i}^{\prime}\left(s_{i-1}\right)=0,~S_{\mathcal{P},i}\left(s_{i-1}\right)=0\text{ and }S_{\mathcal{P},i}^{\prime}\left(s_{i-1}\right)=I
\]
and we further let 
\[
C_{\mathcal{P},i}\left(\sigma\right):=C_{\mathcal{P},i}\left(\sigma,s_{i}\right)\text{ and }S_{\mathcal{P},i}\left(\sigma\right):=S_{\mathcal{P},i}\left(\sigma,s_{i}\right).
\]
Here we view $C_{\mathcal{P},i}\left(s\right)$ and $S_{\mathcal{P},i}\left(s\right)$
as maps from $H_{\mathcal{P}}\left(M\right)$ to $\operatorname*{End}(\mathbb{R}^{d})$.
\end{notation}

\begin{definition} \label{def.2.1-1}Define for all $i=1,\cdots,n,$
\[
f_{\mathcal{P},i}\left(\sigma,s\right)=\begin{cases}
0 & s\in\left[0,s_{i-1}\right]\\
\frac{S_{\mathcal{P},i}\left(\sigma,s\right)}{\Delta_{i}} & s\in\left[s_{i-1},s_{i}\right]\\
\frac{C_{\mathcal{P},j}\left(\sigma,s\right)C_{\mathcal{P},j-1}\left(\sigma\right)\cdot\cdots\cdot C_{\mathcal{P},i+1}\left(\sigma\right)S_{\mathcal{P},i}\left(\sigma\right)}{\Delta_{i}} & s\in\left[s_{j-1},s_{j}\right]\text{ for }j=i+1,\cdots,n

\end{cases}
\]
with the convention that $S_{\mathcal{P},0}\equiv\left\vert \mathcal{P}\right\vert I$ and $f_{\mathcal{P},0}\equiv I$.
\end{definition}

\begin{remark}$S_{\mathcal{P},j}\left(s\right)$, $C_{\mathcal{P},j}\left(s\right)$ may be expressed in terms of $\left\{ f_{\mathcal{P},i}\right\} _{i=0}^{n}$ by 
\[
S_{\mathcal{P},j}\left(s\right)=\Delta_{j}f_{\mathcal{P},j}\left(s\right)
\]
\[
C_{\mathcal{P},j}\left(s\right)=f_{\mathcal{P},j-1}\left(s\right)f_{\mathcal{P},j-1}^{-1}\left(s_{j}\right).
\]

\end{remark}

\section{Approximate Pinned Measures\label{cha.3}}

\subsection{Representation of $\delta$
-- function\label{sec.3.1}}
Let $Y$ be a smooth Riemannian manifold, we will denote the distribution on $Y$ by  $\mathcal{D}^{\prime}\left(Y\right)$ and, compactly supported distribution by $\mathcal{E}^{\prime}\left(Y\right)$. For a matrix $A$, let $eig\left(A\right)$ denote the set of eigenvalues of $A$.
For each $x\in Y$, let $\delta_{x}\in \mathcal{E}^\prime\left(Y\right)$ be the $\delta$--function at $x$ defined by
\[
\delta_{x}\left(f\right)=f\left(x\right)\text{ }\forall f\in C^{\infty}\left(Y\right).
\]
\begin{lemma}[Representation of $\delta_0$ on flat space]\label{lem.3.1.1}There exist functions $\left\{g_{i}\right\}^{d}_{i=0}$ with $g_{0}\in C_{0}^{\infty}\left(\mathbb{R}^{d}\right),$ $\left\{ g_{j}\right\} _{j=1}^{d}\subset C^{\infty}\left(\mathbb{R}^{d}/\left\{ 0\right\} \right)$ with supports contained in a compact subset $K\subset \mathbb{R}^d$ and satisfying 
\begin{equation}
\left|g_{j}\left(x\right)\right|\leq c\left|x\right|^{1-d}\text{ for }j=1,\cdots,d, \label{eq:-19}
\end{equation}such that
\begin{equation}
\delta_0=g_{0}+\sum_{j=1}^{d}\frac{\partial g_{j}}{\partial x_{j}}\text{ in  }\mathcal{E}^\prime\left(\mathbb{R}^d\right).\label{eq:-47}
\end{equation}In more detail, for any $f\in C_{0}^{\infty}\left(\mathbb{R}^{d}\right)$, 
\begin{align}
f\left(0\right)=\int_{\mathbb{R}^d}\left(g_{0}+\sum_{j=1}^{d}\frac{\partial g_{j}}{\partial x_{j}}\right)fdx=\int_{\mathbb{R}^d}\left(g_0f-\sum_{j=1}^{d}\frac{\partial f}{\partial x_{j}}g_{j}\right)dx.
\end{align}

\end{lemma}
\begin{proof}
	This lemma can be derived from Lemma 10.10 in \cite{Strichartz1979}.
\end{proof} 

Based on this representation we can get a representation of $\delta_{p}$ for any $p\in M$. 
\begin{theorem}[Representation of $\delta$
-- function on manifold]\label{thm.1.11}For
any $p\in M,$ there exist functions $\left\{ g_{j}\right\} _{j=0}^{d}\subset C^{\infty}\left(M/\left\{ p\right\} \right)\cap L^{\frac{d}{d-1}}\left(M\right)$ with supports in a compact subset $K$ of $M$ and smooth vector fields $\left\{ X_{j}\right\} _{j=1}^{d}\subset\Gamma^{\infty}\left(TM\right)$
with compact support such that 
\begin{equation}
\delta_{p}=g_{0}+\sum_{j=1}^{d}X_{j}g_{j}\text{ in }\mathcal{E}^{\prime}\left(M\right).\label{eq:-3}
\end{equation}

\end{theorem}

\begin{proof} Pick a chart $\left\{ U,x\right\} $ near $p\in M$
such that $x\left(p\right)=0.$ Since $x\left(U\right)=\mathbb{R}^{d},$
one can apply Lemma \ref{lem.3.1.1} on $x\left(U\right)\simeq\mathbb{R}^{d}$
and get: 
\[
\delta_{0}=\tilde{g}_{0}-\sum_{j=1}^{d}\frac{\partial \tilde{g}_{j}}{\partial x_{j}}
\]
where $\delta_{0}$ is the delta mass on $x\left(U\right)$ supported
at the origin. So for any $h\in C^{\infty}\left(U\right)$
\begin{align*}
h\left(p\right)=h\circ x^{-1}\left(0\right)=\int_{\mathbb{R}^{d}}\left(\tilde{g}_{0}-\sum_{j=1}^{d}\frac{\partial \tilde{g}_{j}}{\partial x_{j}}\right)h\circ x^{-1}d\lambda=\int_{\mathbb{R}^{d}}\left(\tilde{g}_{0}+\sum_{j=1}^{d}\tilde{g}_{j}\frac{\partial}{\partial x_{j}}\right)h\circ x^{-1}d\lambda
\end{align*}
where $d\lambda$ is the Lebesgue measure on $\mathbb{R}^{d}.$ Consider
$\left\{ \frac{\tilde{g}_{j}}{\sqrt{\det g}}\circ x\right\} _{j=0}^{d}$
where $g=\left(g_{ij}\right)_{1\leq i,j\leq d}$ is the metric matrix,
i.e. $g_{ij}=\left\langle \frac{\partial}{\partial x_{i}},\frac{\partial}{\partial x_{j}}\right\rangle _{g}.$
From Lemma \ref{lem.3.1.1} we know that $\frac{\tilde{g}_{j}}{\sqrt{\det g}}\circ x$
has compact support in $U$ and therefore $K:=\cup_{j=1}^{d}supp\left(\frac{\tilde{g}_{j}}{\sqrt{\det g}}\circ x\right)$
is compact in $U.$ Using the smooth Urysohn lemma we can construct a
smooth function $\phi\in C^{\infty}\left(M\to\left[0,1\right]\right)$
such that $\phi^{-1}\left(\left\{ 0\right\} \right)=M/U$ and $\phi^{-1}\left(\left\{ 1\right\} \right)=K.$
Define
\[
\hat{g}_{0}=\phi\frac{\tilde{g}_{0}}{\sqrt{\det g}}\circ x
\]
and
\[
\hat{g}_{j}=\phi\frac{\tilde{g}_{j}}{\sqrt{\det g}}\circ x,\text{ }X_{j}=\phi\cdot\left(x^{-1}\right)_{*}\frac{\partial}{\partial x_{j}}\text{ for }j=1,\dots,d.
\]
Then for any $f\in C^{\infty}\left(M\right),$
\begin{align*}
&\int_{M}\left(\hat{g}_{0}+\sum_{j=1}^{d}\hat{g}_{j}X_{j}\right) fdvol\\
&=\int_{U}\left(\hat{g}_{0}+\sum_{j=1}^{d}\hat{g}_{j}X_{j}\right)fdvol\\
&=\int_{U}\frac{\tilde{g}_{0}}{\sqrt{\det g}}\circ x\cdot\phi fdvol+\sum_{j=1}^{d}\int_{U}\phi^{2}\frac{\tilde{g}_{j}}{\sqrt{\det g}}\circ x\left(\left(x^{-1}\right)_{*}\frac{\partial(\phi f)}{\partial x_{j}}-\left(x^{-1}\right)_{*}\frac{\partial\phi}{\partial x_{j}}f\right)dvol
\end{align*}
where $dvol$ is the volume measure on $M$.

Since $\phi\cdot\left(x^{-1}\right)_{*}\frac{\partial\phi}{\partial x_{j}}\equiv0$
and $\phi\equiv1$ on $K$, we have:
\begin{align*}
\int_{M}\left(\hat{g}_{0}+\sum_{j=1}^{d}\hat{g}_{j}X_{j}\right)fdvol & =\int_{U}\left(\frac{\tilde{g}_{0}}{\sqrt{\det g}}\circ x+\sum_{j=1}^{d}\frac{\tilde{g}_{j}}{\sqrt{\det g}}\circ x\left(x^{-1}\right)_*\frac{\partial}{\partial x_j}\right)fdvol\\
 & =\int_{\mathbb{R}^{d}}\left(\frac{\tilde{g}_{0}}{\sqrt{\det g}}+\sum_{j=1}^{d}\frac{\tilde{g}_{j}}{\sqrt{\det g}}\frac{\partial}{\partial x_j}\right)f\circ x^{-1}\sqrt{\det g}d\lambda\\
 & =\int_{\mathbb{R}^{d}}\left(\tilde{g}_{0}+\sum_{j=1}^{d}\tilde{g}_{j}\frac{\partial}{\partial x_{j}}\right)f\circ x^{-1}d\lambda\\
 & =f\circ x^{-1}\left(0\right)\\
 & =f\left(p\right).
\end{align*}
Therefore, by the Divergence Theorem, we can write down $\delta_{p}$ in distributional sense as 
\[
\delta_{p}=g_{0}+\sum_{j=1}^{d}X_{j}g_{j}
\]
where $g_{0}=\hat{g}_{0}-\sum_{j=1}^{d}\hat{g}_{j}\cdot divX_{j}$ and for $j=1,\dots,n$, $g_{j}=-\hat{g}_{j}$.

From the construction one can see that $X_{j}\in\Gamma^{\infty}\left(TM\right)$ with compact support
 and $\left\{ g_{j}\right\} _{j=0}^{d}\subset C^{\infty}\left(M/\left\{ p\right\} \right)\cap L^{\frac{d}{d-1}}\left(M\right)$
with supports being a compact subset of $M$. 
\end{proof}
\begin{remark} \label{rem3.1.2}Since $C^{\infty}_0(M)$ is dense in $L^q(M),\text{ }\forall q\geq 1$, for any $g_{j},j=1,\cdots,d$, we can find a sequence $\left\{ g_{j}^{\left(m\right)}\right\} _{m}\subset C_{0}^{\infty}\left(M\right)$ such that
\[ g_{j}^{\left(m\right)}\to g_{j}\text{ in }L^\frac{d}{d-1}\left(M\right)\]In particular, we can make $\cup_msupp\left(g_{j}^{\left(m\right)}\right)$ to be compact.
\end{remark}
\begin{corollary}\label{col3.7}
Define
\[
\delta_{x}^{(m)}:=g_{0}^{\left(m\right)}+\sum_{j=1}^{d}X_{j}g_{j}^{\left(m\right)}\in C_{0}^{\infty}\left(M\right).
\]
Then $\left\{\delta_{x}^{(m)}
\right\}_m$ is an approximating sequence of delta mass $\delta_x$, i.e. $\delta_{x}^{(m)}\to\delta_{x}\text{ in }\mathcal{D^{\prime}}\left(M\right).\label{Eq:5}
$
\end{corollary}
\begin{proof}
Using integration by parts, we have for any $f\in C^{\infty}\left(M\right)$, 
\begin{align*}
\int_Mf\delta_{x}^{(m)}d\lambda&=\int_M\left(g_{0}^{\left(m\right)}+\sum_{j=1}^{d}X_{j}g_{j}^{\left(m\right)}\right)fd\lambda=\int_M\left(g_{0}^{\left(m\right)}f+\sum_{j=1}^{d}g_{j}^{\left(m\right)}X_{j}^*f\right)d\lambda
\end{align*}
Since $K:=\cup_msupp\left(g_{j}^{\left(m\right)}\right)$ is compact, $f\cdot1_K$ and $X_{j}^*f\cdot1_K\in L^{\infty-}\left(M\right)$, then Corollary \ref{col3.7} easily follows from Holder's inequality.
\end{proof}

\subsection{Definition of $\nu_{\mathcal{P},x}^{1}$\label{sec.3.2}}

In this section we will give an explicit definition of $\nu_{\mathcal{P},x}^{1}$
proposed in Theorem \ref{thm1}. 
\begin{definition}[End point map]
Define $\Ep: H\left(M\right)\to M$ to be $\Ep\left(\sigma\right)=\sigma\left(1\right)$ and let $\Ep^{\mathcal{P}}$ denote $\Ep\mid_{H_\mathcal{P}\left(M\right)}.$
\end{definition}
Recall from Definition \ref{not1} that
\[
H_{\mathcal{P},x}\left(M\right):=\left\{ \sigma\in H_{\mathcal{P}}\left(M\right)\mid\sigma\left(1\right)=x\right\}=\left(\Ep^{\mathcal{P}}\right)^{-1}\left(\left\{x\right\}\right).
\]
In general, it is not
guaranteed that $\Ep^{\mathcal{P}}$ is a submersion, which would guarantee that $H_{\mathcal{P},x}\left(M\right)$ is an embedded
submanifold of $H_{\mathcal{P}}\left(M\right)$. The following is an easy, yet illuminating, example
showing what can go wrong:

\begin{example} \label{ex.aa}If $M=\mathbb{S}^{2}$, $o$ is the north pole and $\mathcal{P}:=\left\{ 0,1\right\}$, then $\dim H_{\mathcal{P}}\left(M\right)=2$. Consider 
\[
X\left(\sigma,s\right):=\left(0,\pi\sin s\pi,0\right)\in T_{\sigma}H_{\mathcal{P}}\left(M\right)
\]
where 
\[
\sigma\left(s\right)=\left(\sin s\pi,0,\cos s\pi\right).
\]
An one parameter family realizing $X\left(\sigma,s\right)$ would
be 
\[
\sigma_{t}\left(s\right)=\left(\sin s\pi\cos t\pi,\sin s\pi\sin t\pi,\cos s\pi\right),
\]
from which one can easily see that: 
\[
{\Ep}^{\mathcal{P}}_{\ast\sigma}\left(X\right)=\frac{d}{dt}|_{0}\Ep^{\mathcal{P}}\left(\sigma_{t}\right)=\frac{d}{dt}|_{0}\sigma_{t}\left(1\right)=X\left(\sigma,1\right)=0.
\]
So by Rank-Nullity theorem, ${\Ep}^{\mathcal{P}}_{\ast\sigma}$ is not surjective.
\end{example}
The problem comes from the conjugate points on $M$. Two points $p$ and $q$ are conjugate points along a geodesic $\sigma$ if there exists non-zero Jacobi field (smooth vector field along $\sigma$ satisfying Jacobi equation) vanishing at $p$ and $q$. This fact will allow the kernel of ${\Ep}^{\mathcal{P}}_{\ast}$ to be  \textquotedblleft overly large \textquotedblright (more accurately dimension exceeds $\left(n-1\right)d$), so by Rank-nullity theorem,  ${\Ep}^{\mathcal{P}}_{\ast}$ can not be surjective. In this paper we consider manifolds with non--positive sectional curvature. These manifolds do not have conjugate points. From the next proposition we will see that $\Ep^{\mathcal{P}}$ is a submersion on these manifolds. 
\begin{notation} \label{not3.1.8}We construct a $G_\mathcal{P}^1$--orthonormal frame 
\[\left\{ X^{h_{\alpha,i}}:1\leq \alpha\leq d, 1\leq i\leq n\right\}\]of $H_{\mathcal{P}}\left(M\right)$
as follows: for any $\sigma\in H_{\mathcal{P}}\left(M\right)$, $X^{h_{\alpha,i}}(\sigma,\cdot)=u_{(\cdot)}(\sigma)h_{\alpha,i}(\sigma,\cdot)$, where
\begin{equation}
h_{\alpha,i}\in H_{\mathcal{P},\sigma}\text{ and }h_{\alpha,i}^{\prime}(s_{j}+)=\frac{\delta_{i-1,j}e_{\alpha}}{\sqrt{\Delta_{j+1}}}\text{ for }j=0,...,n-1\label{eq:-34}
\end{equation}
and the definition of $H_{\mathcal{P},\sigma}$ can be found in
Eq.$(\ref{equ.6.2})$. 
\end{notation}
\begin{remark} \label{rem.8.3-1}Using Proposition
	\ref{pro.6.4}, it is not hard to see that 
\begin{equation}
h_{\alpha,i}\left(s\right)=\frac{1}{\sqrt{n}}f_{\mathcal{P},i}\left(s\right)e_{\alpha}\label{equ.8.5-1}
\end{equation}
where $\left\{ f_{\mathcal{P},i}\left(s\right)\right\} $
is given in Definition \ref{def.2.1-1}. \end{remark}
\begin{proposition} \label{pro.3.2.1}If $M$ is complete with non-positive sectional
curvature, then for any $x\in M$, $H_{\mathcal{P},x}\left(M\right):={(\Ep^{\mathcal{P}})}^{-1}\left(\left\{ x\right\} \right)$
is an embedded submanifold of $H_{\mathcal{P}}\left(M\right).$ \end{proposition}

\begin{proof} It suffices to show $\Ep^{\mathcal{P}}$ is a submersion. Since $M$ is complete, for any $y\in M$, there exists a geodesic $\sigma$ parametrized on $\left[0,1\right]$ and connecting $o$ and $y$. So $\Ep^{\mathcal{P}}$ is surjective. To show ${\Ep}^{\mathcal{P}}_*$ is surjective, we use a class of vector fields $\left\{X^{h_{\alpha,n}}\right\}_{\alpha=1}^d$ in Notation \ref{not3.1.8}. Since
\[{\Ep}^{\mathcal{P}}_*\left(X^{h_{\alpha,n}}\right)=X_1^{h_{\alpha,n}}=\sqrt{n}u\left(1\right)S_{\mathcal{P},n}e_\alpha,\]where $u\left(\cdot\right)=u\left(\sigma,\cdot\right)$ is the horizontal lift of $\sigma\in H_\mathcal{P}\left(M\right)$. From Proposition \ref{prop A-1} we know $S_{\mathcal{P},n}$ is invertible, therefore $\left\{{\Ep^{\mathcal{P}}}_*\left(X^{h_{\alpha,n}}\right)\right\}_{\alpha=1}^d$ spans $T_{\Ep^{\mathcal{P}}\left(\sigma\right)}M$. So ${\Ep^{\mathcal{P}}}_*$ is surjective. \end{proof}

Since $H_{\mathcal{P},x}\left(M\right)$ is an embedded submanifold of $H_{\mathcal{P}}\left(M\right)$, we can restrict the Riemannian metric $G_{\mathcal{P}}^{1}$
on $TH_{\mathcal{P}}\left(M\right)$ in Eq. (\ref{eq:-4-1-1}) to
a Riemannian metric on $TH_{\mathcal{P},x}\left(M\right)$.

\begin{notation} \label{def.aa}Assuming $M$ has non-positive
sectional curvature, for any $x\in M,$ let $G_{\mathcal{P},x}^{1}$
be the restriction of $G_{\mathcal{P}}^{1}$ to $T_{\sigma}H_{\mathcal{P},x}\left(M\right)\subset T_{\sigma}H_{\mathcal{P}}\left(M\right).$
Further, let ${vol}_{G_{\mathcal{P},x}^{1}}$ be the
associated volume measure on $H_{\mathcal{P},x}\left(M\right).$ \end{notation}

Based on the volume measure $vol_{G_{\mathcal{P},x}^{1}}$
on $H_{\mathcal{P},x}\left(M\right),$ we can construct the pinned approximate
measure $\nu_{\mathcal{P},x}^{1}:$

\begin{definition} \label{def.ab}Let $\nu_{\mathcal{P},x}^{1}$
be the measure on $H_{\mathcal{P},x}\left(M\right)$ defined by 
\begin{equation}
d\nu_{\mathcal{P},x}^{1}\left(\sigma\right)=\frac{1}{J_{\mathcal{P}}\left(\sigma\right)}\frac{1}{Z_{\mathcal{P}}^{1}}e^{\frac{-E\left(\sigma\right)}{2}}dvol _{G_{\mathcal{P},x}^{1}}\left(\sigma\right)\label{equ.meas-1}
\end{equation}
where $J_{\mathcal{P}}\left(\sigma\right):=\sqrt{\det\left({\Ep^{\mathcal{P}}}_{\ast\sigma}{\Ep^{\mathcal{P}}}_{\ast\sigma}^{tr}\right)}$, $Z^1_{\mathcal{P}}:=\left(2\pi\right)^{\frac{dn}{2}}$ and $E(\sigma)=\int_{0}^{1}\left<\sigma^\prime(s),\sigma^\prime(s)\right>_gds$ is the energy of $\sigma$.
\end{definition}

\subsection{Continuous Dependence of $\nu_{\mathcal{P},x}^1$ on $x$}\label{sec.3.3}
Throughout this section we further assume $M$ is simply connected, i.e. $M$ is a Hadamard manifold, and the sectional curvature of $M$ is bounded from below by $-N$. The following theorem illustrates that the measures, $\nu_{\mathcal{P},x}^1$, are finite and \textquotedblleft continuously varying\textquotedblright
  with respect to $x$.

\begin{notation} We will denote by $C_{b}(Y)$ bounded continuous
functions on a topological space $Y$. \end{notation}

\begin{theorem} \label{thm3.2.1}For any $x\in M$, $\nu_{\mathcal{P},x}^1$ is a finite measure. Moreover, for any $f\in C_{b}\left(H_{\mathcal{P},x}\left(M\right)\right)$, define
\begin{equation}
h_{\mathcal{P}}^f\left(x\right):=\int_{H_{\mathcal{P},x}\left(M\right)}f\left(\sigma\right)d\nu_{\mathcal{P},x}^{1}\left(\sigma\right).\label{e.hpx}
\end{equation}
If the mesh size  $\left|\mathcal{P}\right|:=\frac{1}{n}$ of the partition $\mathcal{P}$ is small enough, i.e. $n\geq 3dN$, then $h^f_{\mathcal{P}}\left(x\right)\in C\left(M\right)$. \end{theorem}
Before proving this theorem, we need to set up some notations and
auxiliary results.

\begin{notation} \label{not.1}We fix $n\in\mathbb{N}$ and let $s_{i}:=\frac{i}{n}$ and $\tau:=1-\frac{1}{n}=s_{n-1}$. We further define $\mathcal{K}:=H_{\mathcal{P}}\left(\left[0,\tau\right],M\right)$
be the space of piecewise geodesic paths, $\sigma:\left[0,\tau\right]\rightarrow M$
such that $\sigma\left(0\right)=o\in M.$ \end{notation}
\begin{lemma}
For $x,y\in M$, we can choose an unique element $\log_{x}\left(y\right)\in T_{x}M$ so that  
\[
\gamma_{y,x}\left(t\right):=\exp_x\left(\left(t-\tau\right)\frac{1}{n}\log_{x}\left(y\right)\right),\text{ }\tau \leq t\leq 1.
\]
is the unique minimal-lengh-geodesic connecting $x$ to $y$ such that $\gamma_{y,x}\left(\tau\right)=x$
and $\gamma_{y,x}\left(1\right)=y$.
\end{lemma}
\begin{proof}
Since $M$ is a Hadamard manifold, by the Theorem of Hadamard (See Theorem 3.1 in \cite{DoCarmoRie}), $\exp_x:T_xM\to M$ is a diffeomorphism. Therefore we can see that $\log_{x}\left(y\right)=\exp_x^{-1}\left(y\right)$ is unique and it follows that the geodesic $\gamma_{y,x}$ is unique.
\end{proof}

\begin{definition} For any given $y\in M,$ let $\psi_{y}:\mathcal{K}\rightarrow H_{\mathcal{P},y}\left(M\right):=\left(\Ep^\mathcal{P}\right)^{-1}\left(\left\{y\right\}\right)$ be defined by 
\[
\psi_{y}\left(\sigma\right):=\gamma_{y,\sigma\left(\tau\right)}\ast\sigma
\]
where 
\[
\left(\gamma_{y,\sigma\left(\tau\right)}\ast\sigma\right)\left(t\right)=\left\{ \begin{array}{ccc}
\sigma\left(t\right) & \text{if} & 0\leq t\leq\tau\\
\gamma_{y,\sigma\left(\tau\right)}\left(t\right) & \text{if} & \tau\leq t\leq1
\end{array}.\right.
\]

\end{definition}

\begin{notation} For any $\sigma\in H_{\mathcal{P},y}\left(M\right)$, let $\xi_{y,\sigma}:=u\left(\sigma,\tau\right)^{-1}\log_{\sigma\left(\tau\right)}\left(y\right)\in \mathbb{R}^d$ and $A_{\xi_{y}}\left(\sigma,s\right)=R_{u\left(\sigma,1-s\right)}\left(\xi_{y,\sigma},\cdot\right)\xi_{y,\sigma}$ and $0\leq s\leq \frac{1}{n}$. Denote by $C_{y}\left(\sigma,s\right),S_{y}\left(\sigma,s\right)$ the solutions to $y^{\prime\prime}(\sigma,s)=A_{\xi_{y}}\left(\sigma,s\right)y(\sigma,s)$ with initial values $C_{y}\left(\sigma,0\right)=I,C^\prime_{y}\left(\sigma,0\right)=0,S_{y}\left(\sigma,0\right)=0,S^\prime_{y}\left(\sigma,0\right)=I.$ 
\end{notation}
The next lemma characterizes the differential of $\psi_{y}$:
\begin{lemma} \label{lem3.1.3}
For any $\sigma\in \mathcal{K}$ and $X^h(\sigma,\cdot):=u(\sigma,\cdot)h(\sigma,\cdot)\in T_\sigma\mathcal{K}$,
\[\psi_{y\ast}\left(X^{h}\left(\sigma,\cdot\right)\right)=X^{\hat{h}}\left(\psi_y\left(\sigma\right),\cdot\right):=u\left(\psi_y\left(\sigma\right),\cdot\right)\hat{h}\left(\psi_y\left(\sigma\right),\cdot\right)
\]
where 
\begin{equation}
\hat{h}\left(\psi_y\left(\sigma\right), s\right)=\begin{cases}
h\left(\psi_y\left(\sigma\right), s\right) & s\in\left[0,\tau\right]\\
S_{y}\left(\psi_y\left(\sigma\right), 1-s\right)S_{y}\left(\psi_y\left(\sigma\right), \frac{1}{n}\right)^{-1   }h\left(\sigma, \tau\right) & s\in\left[\tau,1\right]
\end{cases}.\label{eq:-51}
\end{equation}

\end{lemma}

\begin{proof} From now on we will suppress the path argument $\psi_y\left(\sigma\right)$ in $\hat{h}$. Suppose that $t\rightarrow\sigma_{t}\in \mathcal{K}$
is an one-parameter family of curves in $\mathcal{K}$
such that $\sigma_{0}=\sigma$ and $\frac{d}{dt}|_{0}\sigma_{t}=X^{h}\left(\sigma\right).$
Then we have 
\[
\psi_{y\ast}\left(X^{h}\left(\sigma\right)\right)=\frac{d}{dt}|_{0}\psi_{y}\left(\sigma_{t}\right)=\frac{d}{dt}|_{0}\gamma_{y,\sigma_{t}\left(\tau\right)}\ast\sigma_{t}.
\]
If $s\in\left[0,\tau\right],$ then 
\[
\frac{d}{dt}|_{0}\left(\gamma_{y,\sigma_{t}\left(\tau\right)}\ast\sigma_{t}\right)\left(s\right)=\frac{d}{dt}|_{0}\sigma_{t}\left(s\right)=X_{s}^{h}\left(\sigma\right).
\]
While if $s\in\left[\tau,1\right]$ we have 
\[
\frac{d}{dt}|_{0}\left(\gamma_{y,\sigma_{t}\left(\tau\right)}\ast\sigma_{t}\right)\left(s\right)=\frac{d}{dt}|_{0}\gamma_{y,\sigma_{t}\left(\tau\right)}\left(t\right)=:X_{s}^{\hat{h}}\left(\psi_y\left(\sigma\right)\right),
\]
where $X_{s}^{\hat{h}}$ is uniquely determined by, 
\begin{enumerate}
\item $\hat{h}$ satisfies Jacobi's equation, 
\item $\hat{h}\left(\tau\right)=h\left(\tau\right)$ and $\hat{h}\left(1\right)=0$. 
\end{enumerate}
Denote $\hat{h}\left(s\right)$ by $g\left(1-s\right)$ for $s\in\left[\tau,1\right]$,
the above conditions are equivalent to $g$ being the solution to
the following boundary value problem: 
\[
\begin{cases}
g^{\prime\prime}\left(s\right)=A_{\xi_{y}}\left(s\right)g\left(s\right)\\
g\left(0\right)=0\\
g\left(\frac{1}{n}\right)=h\left(\tau\right)
\end{cases}.
\]
Then we use $S_y\left(\cdot\right)$ to express the solution. Here we have used Proposition \ref{prop A-1} to see that $S_y\left(s\right)$ is invertible for $s\in \left[0,\frac{1}{n}\right]$, therefore
\[
g\left(s\right)=S_{y}\left(s\right)S_{y}\left(\frac{1}{n}\right)^{-1}h\left(\tau\right)\text{ for }s\in\left[0,\tau\right]
\]
and thus 
\[
\hat{h}\left(s\right)=g\left(1-s\right)=S_{y}\left(1-s\right)S_{y}\left(\frac{1}{n}\right)^{-1}h\left(\tau\right)\text{ for }s\in\left[\tau,1\right].
\]

\end{proof}

\begin{corollary} For any $y\in M$, $\psi_{y}$ is a diffeomorphism.
\end{corollary}

\begin{proof} From Lemma \ref{lem3.1.3} it is easy to see that the
push forward $\left(\psi_{y}\right)_{*}$ of $\psi_{y}$ is one to
one and thus an isomorphism since $\dim\left(\mathcal{K}\right)=\dim\left(H_{\mathcal{P},y}\left(M\right)\right).$
Therefore the inverse function theorem implies that $\psi_{y}$ is
a local diffeomorphism. Furthermore, $M$ being a Hadamard manifold
implies that $\psi_{y}$ is bijective, so $\psi_{y}$ is actually
a diffeomorphism. \end{proof}
\begin{remark}
An orthonormal frame $\left\{ X^{h_{\alpha,i}}:1\leq \alpha\leq d, 1\leq i\leq n-1\right\}$ of $\mathcal{K}$ can be constructed similarly to Notation $\ref{not3.1.8}$, 
\[
h_{\alpha,i}\in H_{\mathcal{P},\sigma}\text{ and }h_{\alpha,i}^{\prime}(s_{j}+)=\frac{\delta_{i-1,j}e_{\alpha}}{\sqrt{\Delta_{j+1}}}\text{ for }j=0,...,n-2.
\]
In this section we will use the same notation for both these two sets
of orthonormal frames as the meaning should be clear from the context.
\end{remark}
\begin{definition}
$f:M\to N$ is a differentiable map between two Riemannian manifolds $M,N$. The \textbf{Normal Jacobian} of $f$ is defined to be $\sqrt{\det \left(f_*f^{tr}_*\right)}.$
\end{definition}
We will use the orthonormal frame $\left\{ X^{h_{\alpha,i}}:1\leq \alpha\leq d, 1\leq i\leq n-1\right\}$ of
$\mathcal{K}$ to estimate the Normal Jacobian $J_{\mathcal{P}}$ of $\Ep$
in Lemma \ref{lem3.1.4} and the \textquotedblleft volume change 
\textquotedblright $V_{x}$ (See precise definition in Lemma \ref{lem3.1.5}) brought by
the diffeomorphism $\psi_{x}$ in Lemma \ref{lem3.1.5} and \ref{lem3.1.6}.
\begin{lemma}\label{lem3.1.4}Let $J_\mathcal{P}:=\sqrt{\det {\Ep^{\mathcal{P}}}_{*} \left( {\Ep^{\mathcal{P}}}_{*} \right)^{tr}}$ be the Normal Jacobian of $\Ep^{\mathcal{P}}$, then \[J_\mathcal{P}\left(\sigma\right)=
\sqrt{\det\left(\frac{1}{n}\sum_{i=1}^{n}f_{\mathcal{P},i}\left(\sigma,1\right)f_{\mathcal{P},i}^{tr}\left(\sigma,1\right)\right)}\text{  }\forall \sigma\in H_\mathcal{P}\left(M\right).
\]
\end{lemma}

\begin{proof}Note that
\begin{equation*}
{\Ep^{\mathcal{P}}}_{* \sigma}X^{h}\left(\sigma\right)=X^{h}\left(\sigma,1\right),
\end{equation*}
so if $v\in T_{{\Ep^{\mathcal{P}}}\left(\sigma\right)}M$, then
\[
\left\langle \left({\Ep^{\mathcal{P}}}_{*}\right)^{tr}v,X^{h}\right\rangle _{G_{\mathcal{P}}^{1}}=\left\langle v,{\Ep^{\mathcal{P}}}_{*}X^{h}\right\rangle _{T_{{\Ep^{\mathcal{P}}}\left(\sigma\right)}M}=\left\langle u\left(1\right)^{-1}v,h\left(1\right)\right\rangle _{\mathbb{R}^{d}}.
\]
Therefore, using the orthonormal frame $\text{ of }TH_{\mathcal{P}}\left(M\right)$ given by \[\left\{ X^{h_{\alpha,i}}:1\leq \alpha\leq d, 1\leq i\leq n\right\},\]
we find
\[
\left({\Ep^{\mathcal{P}}}_{*}\right)^{tr}v=\sum_{i,\alpha}\left\langle \left({\Ep^{\mathcal{P}}}_{*}\right)^{tr}v,X^{h_{\alpha,i}}\right\rangle _{G_{\mathcal{P}}^{1}}X^{h_{\alpha,i}}=\sum_{i,\alpha}\left\langle u\left(1\right)^{-1}v,h_{\alpha,i}\left(1\right)\right\rangle _{\mathbb{R}^{d}}X^{h_{\alpha,i}}.
\]
Let $\left\{e_\alpha\right\}_{\alpha=1}^d$ be the standard basis of $\mathbb{R}^d$, since $u\left(1\right)$ is an isometry, $\left\{u\left(1\right)e_\alpha\right\}_{\alpha=1}^d$ is an O.N. basis of $T_{{\Ep^{\mathcal{P}}}\left(\sigma\right)}M$. Using
\[
h_{k,i}\left(1\right)=\frac{1}{\sqrt{n}}f_{\mathcal{P},i}\left(1\right)e_{k} \text{ for }1\leq k\leq d,
\]
we can compute: 
\begin{align*}
\det\left({\Ep^{\mathcal{P}}}_{*}\left({\Ep^{\mathcal{P}}}_{*}\right)^{tr}\right) & =\det\left\{ \left\langle \left({\Ep^{\mathcal{P}}}_{*}\right)^{tr}u\left(1\right)e_{\alpha},\left({\Ep^{\mathcal{P}}}_{*}\right)^{tr}u\left(1\right)e_{\beta}\right\rangle _{G_{\mathcal{P}}^1}\right\} _{\alpha,\beta}\\
 & =\det\left\{ \mathop{\sum_{i=1}^{n}\sum_{\gamma=1}^{d}}\left\langle h_{\gamma,i}\left(1\right),e_{\alpha}\right\rangle \left\langle h_{\gamma,i}\left(1\right),e_{\beta}\right\rangle \right\} _{\alpha,\beta}\\
 & =\det\left\{ \mathop{\sum_{i=1}^{n}\sum_{\gamma=1}^{d}}\frac{1}{n}\left\langle e_{\gamma},f_{\mathcal{P},i}^{tr}\left(1\right)e_{\alpha}\right\rangle \left\langle e_{\gamma},f_{\mathcal{P},i}^{tr}\left(1\right)e_{\beta}\right\rangle \right\} _{\alpha,\beta}\\
 & =\det\left\{ \sum_{i=1}^{n}\frac{1}{n}\left\langle f_{\mathcal{P},i}^{tr}\left(1\right)e_{\alpha},f_{\mathcal{P},i}^{tr}\left(1\right)e_{\beta}\right\rangle \right\} _{\alpha,\beta}\\
 & =\det\left(\frac{1}{n}\sum_{i=1}^{n}f_{\mathcal{P},i}\left(1\right)f_{\mathcal{P},i}^{tr}\left(1\right)\right).
\end{align*}
\end{proof}

Using the expression of $J_\mathcal{P}$ in Lemma \ref{lem3.1.4}, we can easily derive the following estimate.
\begin{corollary}\label{Col1}Let $J_\mathcal{P}$ be defined as above, then for any $\sigma\in H_\mathcal{P}\left(M\right)$, $J_{\mathcal{P}}\left(\sigma\right)\geq1$.
\end{corollary} 
\begin{proof}
For any $v\in\mathbb{C}^{d}$, using Proposition \ref{prop A-1}, we
have: 
\begin{align*}
\left\langle \frac{1}{n}\sum_{i=1}^{n}f_{\mathcal{P},i}\left(\sigma,1\right)f_{\mathcal{P},i}^{tr}\left(\sigma,1\right)v,v\right\rangle  & =\frac{1}{n}\sum_{i=1}^{n}\left\Vert f_{\mathcal{P},i}^{tr}\left(\sigma,1\right)v\right\Vert ^{2}\\
 & \geq\frac{1}{n}\sum_{i=1}^{n}\left\Vert v\right\Vert ^{2}\\
 & =\left\Vert v\right\Vert ^{2}.
\end{align*}
So by Min-max theorem,  $eig\left(\frac{1}{n}\sum_{i=1}^{n}f_{\mathcal{P},i}\left(\sigma,1\right)f_{\mathcal{P},i}^{tr}\left(\sigma,1\right)\right)\subset[1,+\infty)$
and therefore: 
\[
J_{\mathcal{P}}\left(\sigma\right)=\sqrt{\det\left(\frac{1}{n}\sum_{i=1}^{n}f_{\mathcal{P},i}\left(\sigma,1\right)f_{\mathcal{P},i}^{tr}\left(\sigma,1\right)\right)}\geq1.
\]
\end{proof}

\begin{lemma} \label{lem3.1.5}
For any $\sigma\in \mathcal{K}$, let $V_x:\mathcal{K}\to \mathbb{R}_+$ be the normal Jacobian of $\psi_x:\mathcal{K}\to H_{\mathcal{P},x}\left(M\right)$, i.e. $V_x:=\sqrt{\det\left(\left({\psi_{x}}_{*}\right)^{tr}{\psi_{x}}_{* }\right)}$, then
\begin{equation}
V_{x}\left(\sigma\right)=\sqrt{\det\left(I+L_{x}\left(\sigma\right)F_{\mathcal{P}}\left(\sigma\right)L_{x}\left(\sigma\right)^{tr}\right)}\text{  }\forall \sigma\in \mathcal{K},\nonumber
\end{equation}
where 
\[
L_{x}\left(\sigma\right):=C_{x}\left(\sigma,\frac{1}{n}\right)S_{x}\left(\sigma,\frac{1}{n}\right)^{-1}\text{ and }F_{\mathcal{P}}\left(\sigma\right):=\frac{1}{n^{2}}\sum_{i=0}^{n-2}f_{\mathcal{P},i}\left(\sigma,\tau\right)f_{\mathcal{P},i}\left(\sigma,\tau\right)^{tr}\label{Eq Fps}.
\]

\end{lemma}

\begin{proof} Using (\ref{eq:-51}) and differentiating $\hat{h}$
with respect to $s$, we get: 
\begin{equation}
\hat{h}^{\prime}\left(\sigma,\tau+\right)=-C_{x}\left(\sigma,\frac{1}{n}\right)S_{x}\left(\sigma,\frac{1}{n}\right)^{-1}h\left(\sigma,\tau\right):=-L_{x}\left(\sigma\right)h\left(\sigma,\tau\right).\label{eq:-11}
\end{equation}
Also from Proposition \ref{pro.6.4} we have 
\[
h\left(\sigma,\tau\right)=\frac{1}{n}\sum_{i=0}^{n-1}f_{\mathcal{P},i+1}\left(\sigma,\tau\right)h^{\prime}\left(\sigma,s_{i}+\right),
\]
so
\begin{equation}
\hat{h}^{\prime}\left(\sigma,\tau+\right)=-L_{x}\left(\sigma\right)\frac{1}{n}\sum_{i=0}^{n-1}f_{\mathcal{P},i+1}\left(\sigma,\tau\right)h^{\prime}\left(\sigma,s_{i}+\right).\label{eq:-12}
\end{equation}
For any $\alpha,\beta\in\left\{ 1,...,d\right\} $ and $i,j\in\left\{ 1,...,n-1\right\} $,
\begin{align*}
  &\left\langle \psi_{x\ast}\left(X^{h_{\alpha,i}}\left(\sigma\right)\right),\psi_{x\ast}\left(X^{h_{\beta,j}}\left(\sigma\right)\right)\right\rangle _{T_{\psi_{x}\left(\sigma\right)}H_{\mathcal{P},x}\left(M\right)}\\
 & =\frac{1}{n}\sum_{k=0}^{n-2}\left\langle h_{\alpha,i}^{\prime}\left(s_{k+}\right),h_{\beta,j}^{\prime}\left(s_{k+}\right)\right\rangle +\frac{1}{n}\left\langle \hat{h}_{\alpha,i}^{\prime}\left(\tau+\right),\hat{h}_{\beta,j}^{\prime}\left(\tau+\right)\right\rangle \\
 & =\delta_{\left(\alpha,i\right)}^{\left(\beta,j\right)}+\frac{1}{n}\left\langle L_{x}\left(\sigma\right)\frac{1}{n}\frac{f_{\mathcal{P},i}\left(\tau\right)e_{\alpha}}{\sqrt{\frac{1}{n}}},L_{x}\left(\sigma\right)\frac{1}{n}\frac{f_{\mathcal{P},j}\left(\tau\right)e_{\beta}}{\sqrt{\frac{1}{n}}}\right\rangle \\
 & =\delta_{\left(\alpha,i\right)}^{\left(\beta,j\right)}+\left\langle L_{x}\left(\sigma\right)\frac{1}{n}f_{\mathcal{P},i}\left(\tau\right)e_{\alpha},L_{x}\left(\sigma\right)\frac{1}{n}f_{\mathcal{P},j}\left(\tau\right)e_{\beta}\right\rangle ,
\end{align*}
where \[\delta_{\left(\alpha,i\right)}^{\left(\beta,j\right)}=\begin{cases}
1 & \alpha=\beta, i=j\\ 0 &\text{otherwise}.
\end{cases}\]
It follows that the volume change 
\begin{equation}
V_{x}\left(\sigma\right)=\sqrt{\det\left(I_{\left(\mathbb{R}^{d}\right)^{n-1}}+\hat{T}_{x}\left(\sigma\right)\right)}\label{V_x}
\end{equation}
where $\hat{T}_{x}\left(\sigma\right)\in End\left(\left(\mathbb{R}^{d}\right)^{n-1}\right)$ is defined by
\[
\left(\hat{T_{x}}\left(\sigma\right)\right)_{d\left(i-1\right)+\alpha,d\left(j-1\right)+\beta}=\left\langle L_{x}\left(\sigma\right)\frac{1}{n}f_{\mathcal{P},i}\left(\sigma,\tau\right)e_{\alpha},L_{x}\left(\sigma\right)\frac{1}{n}f_{\mathcal{P},j}\left(\sigma,\tau\right)e_{\beta}\right\rangle,
\]
where $1\leq i,j\leq n-1, 1\leq \alpha,\beta\leq d$.
If 
\[
S_\sigma=\left(\begin{array}{c}
I_{\left(\mathbb{R}^{d}\right)^{n-1}}\\
A_{x}\left(\sigma\right)
\end{array}\right)\in M_{nd\times\left(n-1\right)d}
\]
and 
\[
A_{x}\left(\sigma\right)=\left(\frac{1}{n}L_{x}\left(\sigma\right)f_{\mathcal{P},0}\left(\sigma,\tau\right)e_{1},\cdots,\frac{1}{n}L_{x}\left(\sigma\right)f_{\mathcal{P},n-2}\left(\sigma,\tau\right)e_{d}\right)\in M_{d\times\left(n-1\right)d}, 
\]
then \[
I_{\left(\mathbb{R}^{d}\right)^{n-1}}+\hat{T}_{x}\left(\sigma\right)=S_\sigma^{tr}S_\sigma.
\]
Applying Lemma \ref{Conje1} we get: 
\begin{align*}
\det\left(I_{\left(\mathbb{R}^{d}\right)^{n-1}}+\hat{T}_{x}\left(\sigma\right)\right) & =\det\left(I_{\left(\mathbb{R}^{d}\right)}+A_{x}\left(\sigma\right)A_{x}\left(\sigma\right)^{tr}\right)\\
 & =\det\left(I+\frac{1}{n^{2}}\sum_{i=0}^{n-2}\sum_{\alpha=1}^{d}L_{x}f_{\mathcal{P},i}\left(\tau\right)e_{\alpha}e_{\alpha}^{tr}f_{\mathcal{P},i}\left(\tau\right)^{tr}L_{x}^{tr}\right)\\
 & =\det\left(I+L_{x}F_{\mathcal{P}}L_{x}^{tr}\right)
\end{align*}
where $F_\mathcal{P}\left(\sigma\right)$ is as in Eq. (\ref{Eq Fps}).
\end{proof}

\begin{lemma} \label{lem3.1.6}

For any $\sigma\in\mathcal{K}$, 
\begin{equation}
V_{x}\left(\sigma\right)\leq\sum_{k=0}^{d}\binom{d}{k}n^{\frac{k}{2}}e^{\frac{Nk}{2}d^{2}\left(\sigma\left(\tau\right),x\right)}\Pi_{j=0}^{n-2}e^{kNd^{2}\left(\sigma\left(s_{j}\right),\sigma\left(s_{j+1}\right)\right)}.\label{eq:-52}
\end{equation}

\end{lemma}

\begin{proof} From Lemma \ref{lem3.1.5} and Appendix \ref{app.C}, one can
see, after suppressing $\sigma$,
\begin{align*}
\det\left(I_{\left(\mathbb{R}^{d}\right)^{n-1}}+\hat{T}_{x}\right) =\det\left(I+L_{x}F_{\mathcal{P}}L_{x}^{tr}\right)=\Pi_{i=1}^{d}\left(1+\lambda_{i,x}\right)\leq\left(1+\max_{1\leq i\leq d}\lambda_{i,x}\right)^{d},
\end{align*}
where $\left\{ \lambda_{i,x}\right\}_{i=1}^d =eig\left(L_{x}F_{\mathcal{P}}L_{x}^{tr}\right)$.

Note that 
\begin{align*}
\max_{1\leq i\leq d}\lambda_{i,x}=\left\Vert L_{x}\left(\sigma\right)F_{\mathcal{P}}L_{x}\left(\sigma\right)^{tr}\right\Vert &\leq\left\Vert L_{x}\left(\sigma\right)\right\Vert ^{2}\left\Vert F_{\mathcal{P}}\right\Vert\leq\frac{1}{n}\left\Vert L_{x}\left(\sigma\right)\right\Vert ^{2}\sup_{0\leq i\leq n-2}\left\Vert f_{\mathcal{P},i}\left(\tau\right)\right\Vert ^{2}.
\end{align*}
Using Proposition \ref{prop A-2}, we get: 
\[
\left\Vert C_{x}\left(\sigma,\frac{1}{n}\right)\right\Vert \leq e^{\frac{N}{2}d^{2}\left(\sigma\left(\tau\right),x\right)},
\]
where for any $x,y\in M,$ $d\left(x,y\right)$ is the geodesic distance
between $x$ and $y$,
and 
\[
\left\Vert S_{x}^{-1}\left(\sigma,\frac{1}{n}\right)\right\Vert \leq n,
\]
so 
\[
\left\Vert L_{x}\left(\sigma\right)\right\Vert ^{2}\leq n^{2}e^{Nd^{2}\left(\sigma\left(\tau\right),x\right)}
\]
and 
\[
\max_{1\leq i\leq d}\lambda_{i,x}\leq ne^{Nd^{2}\left(\sigma\left(\tau\right),x\right)}\sup_{0\leq i\leq n-2}\left\Vert f_{\mathcal{P},i}\left(\sigma,\tau\right)\right\Vert ^{2}.
\]
Therefore 
\begin{align}
V_{x}\left(\sigma\right)=\left(1+\max_{1\leq i\leq d}\lambda_{i,x}\right)^{\frac{d}{2}} & \leq\left(1+ne^{Nd^{2}\left(\sigma\left(\tau\right),x\right)}\sup_{0\leq i\leq n-2}\left\Vert f_{\mathcal{P},i}\left(\sigma,\tau\right)\right\Vert ^{2}\right)^{\frac{d}{2}}\nonumber \\
 & \leq\left(1+n^{\frac{1}{2}}e^{\frac{N}{2}d^{2}\left(\sigma\left(\tau\right),x\right)}\sup_{0\leq i\leq n-2}\left\Vert f_{\mathcal{P},i}\left(\sigma,\tau\right)\right\Vert \right)^{d}\nonumber \\
 & =\sum_{k=0}^{d}\binom{d}{k}n^{\frac{k}{2}}e^{\frac{Nk}{2}d^{2}\left(\sigma\left(\tau\right),x\right)}\sup_{0\leq i\leq n-2}\left\Vert f_{\mathcal{P},i}\left(\sigma,\tau\right)\right\Vert ^{k}.\label{eq:-13}
\end{align}
Applying Proposition \ref{prop A-2} to $f_{\mathcal{P},i}\left(\sigma, \tau\right)$ shows
\begin{align*}
&\left\Vert f_{\mathcal{P},i}\left(\tau\right)\right\Vert\leq\left\Vert C_{\mathcal{P},n-1}\right\Vert \cdots\left\Vert C_{\mathcal{P},i+1}\right\Vert \left\Vert \frac{S_{i}}{\Delta_{i}}\right\Vert \\
 & \leq e^{\frac{1}{2}Nd^{2}\left(\sigma\left(s_{n-2}\right),\sigma\left(s_{n-1}\right)\right)}\cdot\cdots\cdot e^{\frac{1}{2}Nd^{2}\left(\sigma\left(s_{i-1}\right),\sigma\left(s_{i}\right)\right)}\left(1+\frac{Nd^{2}\left(\sigma\left(s_{i-1}\right),\sigma\left(s_{i}\right)\right)}{6}\right)\\
 & \leq\Pi_{j=i-1}^{n-2}e^{\frac{1}{2}Nd^{2}\left(\sigma\left(s_{j}\right),\sigma\left(s_{j+1}\right)\right)}\cdot e^{\frac{Nd^{2}\left(\sigma\left(s_{i-1}\right),\sigma\left(s_{i}\right)\right)}{6}}\\
 & \leq\Pi_{j=i-1}^{n-2}e^{Nd^{2}\left(\sigma\left(s_{j}\right),\sigma\left(s_{j+1}\right)\right)}\\
 & \leq\Pi_{j=0}^{n-2}e^{Nd^{2}\left(\sigma\left(s_{j}\right),\sigma\left(s_{j+1}\right)\right)}
\end{align*}
Taking supremum over $i$, we get: 
\begin{equation}
\sup_{0\leq i\leq n-2}\left\Vert f_{\mathcal{P},i}\left(\sigma,\tau\right)\right\Vert \leq\Pi_{j=0}^{n-2}e^{Nd^{2}\left(\sigma\left(s_{j}\right),\sigma\left(s_{j+1}\right)\right)}.\label{eq:-20}
\end{equation}
and Eq.(\ref{eq:-52}) follows. \end{proof}
\begin{definition} For any $X$, $Y\in T\mathcal{K}$(the tangent bundle of $\mathcal{K}$), define two metrics $G_{\mathcal{P},\tau}^{0}$,
$G_{\mathcal{P},\tau}^{1}$ to be: 
\[
\left<X,Y\right>_{G_{\mathcal{P},\tau}^{0}}=\sum_{i=1}^{n-1}\left\langle X\left(s_{i}\right),Y\left(s_{i}\right)\right\rangle \Delta_{i}
\]
and 
\[
\left<X,Y\right>_{G_{\mathcal{P},\tau}^{1}}=\sum_{i=1}^{n-1}\left\langle \frac{\nabla X}{ds}\left(s_{i-1}\right),\frac{\nabla Y}{ds}\left(s_{i-1}\right)\right\rangle \Delta_{i}.
\]
\end{definition}
\begin{lemma}
$G_{\mathcal{P},\tau}^{0}$ is a metric on $\mathcal{K}$.
\end{lemma}
\begin{proof}
The only non--trivial part is to check $\left<X,X\right>_{G_{\mathcal{P},\tau}^{0}}=0\implies X=0$.
Since $M$ has non--positive curvature, there are no conjugate points. For each $0\leq i\leq n-1$, there is a unique piece of Jacobi field $X$ along $\sigma\mid_{[s_i,s_{i+1}]}$ with specified boundary values $X\left(s_i\right)$ and $X\left(s_{i+1}\right)$. $\left<X,X\right>_{G_{\mathcal{P},\tau}^{0}}=0\implies X\left(s_{i}\right)=0\text{ for any }1\leq i \leq n$. By the uniqueness of Jacobi field, $X\equiv 0$.  
\end{proof}

Based on the metric $G_{\mathcal{P},\tau}^{0}$
and $G_{\mathcal{P},\tau}^{1}$, we define measures $\nu_{\mathcal{P},\tau}^{0}$
and $\nu_{\mathcal{P},\tau}^{1}$ on $\mathcal{K}$ as follows.
\begin{definition}  Let
\[
d\nu_{\mathcal{P},\tau}^{0}:=\frac{n^{\left(n-1\right)d}}{\left(2\pi\right)^{\left(n-1\right)\frac{d}{2}}}e^{-\frac{1}{2}E}dvol_{G_{\mathcal{P},\tau}^{0}}
\]
and 
\[
d\nu_{\mathcal{P},\tau}^{1}=\frac{1}{\left(2\pi\right)^{\left(n-1\right)\frac{d}{2}}}e^{-\frac{1}{2}E}dvol_{G_{\mathcal{P},\tau}^{1}}
\]

\end{definition}

\begin{lemma} \label{lem3.1.7}
If
\[
\rho_{\mathcal{P}}\left(\sigma\right):=\Pi_{i=1}^{n-1}\det\left(\frac{S_{\mathcal{P},i}\left(\sigma\right)}{n}\right)\text{  }\forall \sigma\in \mathcal{K},
\]
then $d\nu_{\mathcal{P},\tau}^{0}=\rho_{\mathcal{P}}d\nu_{\mathcal{P},\tau}^{1}$ and moreover, $\rho_{\mathcal{P}}\left(\sigma\right)\geq1$ $\forall \sigma\in \mathcal{K}$.

\end{lemma}

\begin{proof} The argument to show $\rho_{\mathcal{P}}$ is the density
of $\nu_{\mathcal{P},\tau}^{0}$ with respect to $\nu_{\mathcal{P},\tau}^{1}$
is almost exactly the same as Theorem 5.9 in \cite{Andersson1999}
with a slight change of ending point from $1$ to $\tau$. Here we
focus on the lower bound estimate of $\rho_{\mathcal{P}}\left(\sigma\right).$
Since for any $v\in\mathbb{C}^{d}$, 
\[
\left\Vert \frac{S_{\mathcal{P},i}}{n}v\right\Vert \geq\left\Vert v\right\Vert ,
\]
we know from Proposition \ref{prop A-1} that for any $\lambda\in eig\left(\frac{S_{\mathcal{P},i}}{n}\right),$
\[
\left|\lambda\right|\geq1,
\]
and from which we know: 
\[
\rho_{\mathcal{P}}\left(\sigma\right)=\Pi_{i=1}^{n-1}\det\left(\frac{S_{\mathcal{P},i}\left(\sigma\right)}{n}\right)\geq1.
\]

\end{proof}

\begin{proof}[Proof of Theorem \ref{thm3.2.1}]Since $\psi_{x}$
is a diffeomorphism,
\begin{align*}
h^f_{\mathcal{P}}\left(x\right)=&\int_{H_{\mathcal{P},x}\left(M\right)}\frac{1}{Z_\mathcal{P}^1}\frac{f}{J_{\mathcal{P}}}\left(\sigma\right)e^{-\frac{1}{2}E\left(\sigma\right)}dvol_{G_{\mathcal{P},x}^1}\left(\sigma\right)\\
=&\int_\mathcal{K}\frac{1}{Z_{\mathcal{P}}^{1}}\frac{f}{J_{\mathcal{P}}}\circ\psi_{x}\left(\sigma\right)e^{-\frac{1}{2}E\circ\psi_{x}\left(\sigma\right)}V_{x}\left(\sigma\right)dvol_{G_{\mathcal{P},\tau}^{1}}\left(\sigma\right).
\end{align*}
Notice that
\begin{equation}
\frac{1}{Z_{\mathcal{P}}^{1}}e^{-\frac{1}{2}E\circ\psi_{x}\left(\sigma\right)}=\frac{1}{\left(2\pi\right)^{\frac{d}{2}}}\frac{1}{\left(2\pi\right)^{\left(n-1\right)\frac{d}{2}}}e^{-\frac{1}{2}E\left(\sigma\right)}e^{-\frac{n}{2}d^{2}\left(\sigma\left(\tau\right),x\right)},\label{eq:-35}
\end{equation}
so
\begin{equation}
h^f_{\mathcal{P}}\left(x\right)=\frac{1}{\left(2\pi\right)^{\frac{d}{2}}}\int_\mathcal{K}\frac{f}{J_{\mathcal{P}}}\circ\psi_{x}\left(\sigma\right)e^{-\frac{n}{2}d^{2}\left(\sigma\left(\tau\right),x\right)}V_x\left(\sigma\right)d\nu_{G_{\mathcal{P},\tau }^{1}}\left(\sigma\right)\label{eq:-2}
\end{equation}
Combining $\left(\ref{eq:-13}\right)$, $\left(\ref{eq:-20}\right)$ we know that:
\begin{align*}
e^{-\frac{n}{2}d^{2}\left(\sigma\left(\tau\right),x\right)}V_x\left(\sigma\right)&\leq \sum_{k=0}^{d}\binom{d}{k}n^{\frac{k}{2}}e^{\frac{Nk-n}{2}d^{2}\left(\sigma\left(\tau\right),x\right)}\Pi_{j=0}^{n-2}e^{Nd^{2}\left(\sigma\left(s_{j}\right),\sigma\left(s_{j+1}\right)\right)}.
\end{align*}
So 
\begin{equation*}
\sup_{x\in M}e^{-\frac{n}{2}d^{2}\left(\sigma\left(\tau\right),x\right)}V_x\left(\sigma\right)\leq\sup_{x\in M}e^{-\frac{n-Nk}{2}d^{2}\left(\sigma\left(\tau\right),x\right)} \sum_{k=0}^{d}\binom{d}{k}n^{\frac{k}{2}}\Pi_{j=0}^{n-2}e^{Nkd^{2}\left(\sigma\left(s_{j}\right),\sigma\left(s_{j+1}\right)\right)}.
\end{equation*}
When $n$ is large enough, i.e. $n>Nk$, $e^{-\frac{n-Nk}{2}d^{2}\left(\sigma\left(\tau\right),x\right)}\leq 1$ and thus it suffices to show 
\begin{equation*}
\mathbb{E}_{\nu_{G_{\mathcal{P},\tau }^{1}}}\left[\sum_{k=0}^{d}\binom{d}{k}n^{\frac{k}{2}}\Pi_{j=0}^{n-2}e^{Nkd^{2}\left(\sigma\left(s_{j}\right),\sigma\left(s_{j+1}\right)\right)}\right]<\infty.
\end{equation*}
For each $k\leq d$ we have:
\begin{align}
 \mathbb{E}_{\nu_{G_{\mathcal{P},\tau }^{1}}}\left[\binom{d}{k}n^{\frac{k}{2}}\Pi_{j=0}^{n-2}e^{Nkd^{2}\left(\sigma\left(s_{j}\right),\sigma\left(s_{j+1}\right)\right)}\right]
 & =C_n \mathbb{E}_{\mu}\left[\Pi_{j=0}^{n-2}e^{Nk\left|\Delta_{j+1}\beta\right|^2}\right]\label{Eq:-h}
\end{align}
Using Lemma \ref{cBM} in Appendix \ref{App.B} we obtain a bound of the right--hand side of Eq. (\ref{Eq:-h}) (the bound here depends on $n$). 

Since for any $\sigma\in \mathcal{K}$, $\frac{f}{J_{\mathcal{P}}}\circ\psi_{x}\left(\sigma\right)e^{-\frac{n}{2}d^{2}\left(\sigma\left(\tau\right),x\right)}V_x\left(\sigma\right)$ is continuous with respect to $x\in M$, so by dominated convergence theorem, $h^f_\mathcal{P}\left(x\right)\in C\left(M\right)$.
\end{proof}

Not only can we show that $h^f_\mathcal{P}\left(x\right)$ is a continuous function, it is bounded uniformly in $x\in M$ and partition $\mathcal{P}$, as is shown in the following proposition.
\begin{proposition}\label{pro3.31}
$\sup_{\mathcal{P}}\sup_{x\in M}\nu_{\mathcal{P},x}^1\left(H_{\mathcal{P},x}\left(M\right)\right)< \infty.$
\end{proposition}
\begin{proof}
Based on Eq. (\ref{eq:-2}) and Corollary \ref{Col1},
\begin{equation*}
\nu_{\mathcal{P},x}^1\left(H_{\mathcal{P},x}\left(M\right)\right)\leq C_d\int_{\mathcal{K}}e^{-\frac{n}{2}d^{2}\left(\sigma\left(\tau\right),x\right)}V_x\left(\sigma\right)d\nu_{G_{\mathcal{P},\tau }^{1}}\left(\sigma\right)
\end{equation*}
Combining Eq.(\ref{eq:-13}) and Eq.(\ref{eq:-20}) we know that:
\begin{align*}
e^{-\frac{n}{2}d^{2}\left(\sigma\left(\tau\right),x\right)}V_x\left(\sigma\right)&\leq \sum_{k=0}^{d}\binom{d}{k}n^{\frac{k}{2}}e^{\frac{Nk-n}{2}d^{2}\left(\sigma\left(\tau\right),x\right)}\Pi_{j=0}^{n-2}e^{Nd^{2}\left(\sigma\left(s_{j}\right),\sigma\left(s_{j+1}\right)\right)}
\end{align*}
For each $k\leq d$, applying Lemma \ref{lem3.1.7}, we have:
\begin{align}
 \mathbb{E}&_{\nu_{G_{\mathcal{P},\tau }^{1}}}\left[e^{-\frac{n-Nk}{2}d^{2}\left(\sigma\left(\tau\right),x\right)}\binom{d}{k}n^{\frac{k}{2}}\Pi_{j=0}^{n-2}e^{Nkd^{2}\left(\sigma\left(s_{j}\right),\sigma\left(s_{j+1}\right)\right)}\right]\nonumber\\
 & =\binom{d}{k}n^{\frac{k}{2}}\int_{\mathcal{K}}e^{-\frac{n-Nk}{2}d^{2}\left(\sigma\left(\tau\right),x\right)}\Pi_{j=0}^{n-2}e^{Nkd^{2}\left(\sigma\left(s_{j}\right),\sigma\left(s_{j+1}\right)\right)}d\nu_{G_{\mathcal{P},\tau }^{1}}\left(\sigma\right) \nonumber\\
 & =\binom{d}{k}n^{\frac{k}{2}}\int_{\mathcal{K}}e^{-\frac{n-Nk}{2}d^{2}\left(\sigma\left(\tau\right),x\right)}\Pi_{j=0}^{n-2}e^{Nkd^{2}\left(\sigma\left(s_{j}\right),\sigma\left(s_{j+1}\right)\right)}\frac{1}{\rho_{\mathcal{P}}\left(\sigma\right)}d\nu_{\mathcal{P},\tau }^{0}\left(\sigma\right)\nonumber\\
 & 
 \leq\binom{d}{k}n^{\frac{k}{2}}\int_{\mathcal{K}}e^{-\frac{n-Nk}{2}d^{2}\left(\sigma\left(\tau\right),x\right)}\Pi_{j=0}^{n-2}e^{Nkd^{2}\left(\sigma\left(s_{j}\right),\sigma\left(s_{j+1}\right)\right)}d\nu_{\mathcal{P},\tau }^{0}\left(\sigma\right)\label{eq:-48}
\end{align}
Now define the projection map $\pi_{\mathcal{P}}:\mathcal{K}\to M^{n-1},$
for any $\sigma\in \mathcal{K}$,
\[
\pi_{\mathcal{P}}\left(\sigma\right):=\left(\sigma\left(s_{1}\right),\dots,\sigma\left(s_{n-1}\right)\right).
\]
Since $M$ is a Hadamard manifold, $\pi_{\mathcal{P}}$ is
a diffeomorphism. From there one can get: 
\begin{align}
&\binom{d}{k} n^{\frac{k}{2}}\int_{\mathcal{K}}e^{-\frac{n-Nk}{2}d^{2}\left(\sigma\left(\tau\right),x\right)}\Pi_{j=0}^{n-2}e^{Nkd^{2}\left(\sigma\left(s_{j}\right),\sigma\left(s_{j+1}\right)\right)}d\nu_{\mathcal{P},\tau }^{0}\left(\sigma\right)\nonumber\\
&
=\frac{\binom{d}{k}n^{\frac{k+\left(n-1\right)d}{2}}}{\left(2\pi\right)^{\frac{\left(n-1\right)d}{2}}}\int_{M^{n-1}}e^{-\frac{n-Nk}{2}d^{2}\left(x_{n-1},x\right)}\Pi_{j=0}^{n-2}e^{-\frac{1}{2}\left(n-2Nk\right)d^{2}\left(x_{j},x_{j+1}\right)}dx_{1}\cdots dx_{n-1}\label{eq:-49}
\end{align}
Corollary 4.2 in \cite{Sturm1992} gives a lower bound of the heat kernel of manifold $M$, provided $Ric\geq\left(1-d\right)N$:
\[
p_{t}\left(x,y\right)\geq\left(2\pi t\right)^{-\frac{d}{2}}e^{-\frac{\rho^{2}}{2t}}\left(\frac{\sinh\sqrt{N}\rho}{\sqrt{N}\rho}\right)^{\frac{1-d}{2}}e^{-Ct}
\]
where $N$ is the curvature bound and $C$ is some constant depending
only on $d$ and $N$ and $\rho=d\left(x,y\right).$ Using the fact
that
\[
\frac{\sinh\sqrt{N}\rho}{\sqrt{N}\rho}\leq e^{\frac{N\rho^{2}}{2}},
\]
it follows that
\[
p_{t}\left(x,y\right)\geq\left(2\pi t\right)^{-\frac{d}{2}}e^{-\frac{1}{2}\left(\frac{1}{t}+\frac{N\left(d-1\right)}{2}\right)\rho^{2}}e^{-Ct}.
\]
Let $t=\frac{1}{n-N_{1}}$, where $N_{1}=2Nd+\frac{N\left(d-1\right)}{2}$, we have, for any $j\in\left\{ 0,\dots,n-1\right\} $: 
\[
e^{-\frac{1}{2}\left(n-2Nd\right)d^{2}\left(x_{j},x_{j+1}\right)}\leq e^{Ct}p_{t}\left(x_j,x_{j+1}\right)\left(2\pi t\right)^{\frac{d}{2}}.
\]
So
\begin{align}
  &\frac{\binom{d}{k}n^{\frac{k+\left(n-1\right)d}{2}}}{\left(2\pi\right)^{\frac{\left(n-1\right)d}{2}}}\int_{M^{n-1}}e^{-\frac{n-2Nk}{2}d^{2}\left(x_{n-1},x\right)}\Pi_{j=0}^{n-2}e^{-\frac{1}{2}\left(n-2Nd\right)d^{2}\left(x_{j},x_{j+1}\right)}dx_{1}\cdots dx_{n-1}\nonumber \\
 & \leq\frac{\binom{d}{k}n^{\frac{k+\left(n-1\right)d}{2}}}{\left(n-N_{1}\right)^{\frac{nd}{2}}}e^{C\frac{n}{n-N_{1}}}\int_{M^{n-1}}p_{\frac{1}{n-N_1}}\left(x_{n-1},x\right)\Pi_{j=0}^{n-2}p_{\frac{1}{n-N_{1}}}\left(x_{j},x_{j+1}\right)dx_{1}\cdots dx_{n-1}\nonumber \\
 & =\frac{\binom{d}{k}e^{\frac{Cn}{n-N_{1}}}}{n^{\frac{d-k}{2}}\left(1-\frac{N_{1}}{n}\right)^{\frac{nd}{2}}}\int_Mp_{\frac{1}{n-N_1}}\left(x_{n-1},x\right)p_{\frac{n-1}{n-N_{1}}}\left(o,x_{n-1}\right)dx_{n-1}\nonumber\\
 &=\frac{\binom{d}{k}e^{\frac{Cn}{n-N_{1}}}}{n^{\frac{d-k}{2}}\left(1-\frac{N_{1}}{n}\right)^{\frac{nd}{2}}}p_{\frac{n}{n-N_{1}}}\left(o,x\right)\label{eq:-17}
\end{align}
Since the heat kernel is continuous with respect to time t , combining $\left(\ref{eq:-48}\right)$ ,$\left(\ref{eq:-49}\right)$ and $\left(\ref{eq:-17}\right)$ we get 
\[
\frac{\binom{d}{k}e^{\frac{Cn}{n-N_{1}}}}{n^{\frac{d-k}{2}}\left(1-\frac{N_{1}}{n}\right)^{\frac{nd}{2}}}p_{\frac{n}{n-N_{1}}}\left(0,x\right)\leq C.
\]
and hence 
\[
\nu_{\mathcal{P},x}^1\left(H_{\mathcal{P},x}\left(M\right)\right)\leq C.
\]
where $C$ is a constant depending only on $d$ and $N$.
\end{proof}

Theorem \ref{thm3.2.1} shows that the class of approximate pinned
measures $\left\{ \nu_{\mathcal{P},x}^{1}\right\} $ are finite measures
and using the continuity result for $h_{\mathcal{P}}\left(x\right),$
one can see that $\nu_{\mathcal{P},x}^{1}$ is deserved to be formally
expressed as $\delta_{x}\left(\sigma\left(1\right)\right)\nu_{\mathcal{P}}^{1}$
and it should be interpreted in the sense of Corollary \ref{Col3.2.2} below. First we state a co--area formula.
\begin{theorem}[Theorem 2.3 in Federer \cite{Federer1969}]
Let $H$ and $M$ be two Riemannian manifolds with volume measures $dvol_H$ and $dvol_M$ respectively. If $p:H\to M$ is a smooth submersion, $g:H\to [0,\infty)$ is a density function, for each $x\in M$, let $dvol_{H_x}$ be the volume measure on $H_x:=p^{-1}\left(\left\{x\right\}\right)$ and $J\left(y\right):=\sqrt{\det\left(p_{*y}p^{tr}_{*y}\right)}$ on $y\in H_x$, then for any non--negative measurable function $f:M\to [0,\infty)$, 
\begin{equation}
\int_H\left(f\circ p\right)gdvol_H=\int_Mdvol_M\left(x\right)f\left(x\right)\int_{H_x}\frac{1}{J\left(y\right)}g\left(y\right)dvol_{H_x}\left(y\right).\label{co}
\end{equation}
\end{theorem}

\begin{corollary} \label{Col3.2.2}Denote by $\delta_{x}\in\mathcal{E}^{\prime}\left(M\right)$
the delta--function at $x\in M$ and $\left\{ \delta_{x}^{\left(m\right)}\right\} \subset C_{0}^{\infty}\left(M\right)$ is an approximate sequence to $\delta_x$  $(\delta_{x}^{\left(m\right)}\to\delta_{x}\text{ in }\mathcal{E}^{\prime}\left(M\right))$ i.e. for any $h\in C^{\infty}\left(M\right),$ we have: 
\[
\lim_{m\to\infty}\int_{M}h\left(y\right)\delta_{x}^{\left(m\right)}\left(y\right)dy=\int_{M}h\left(y\right)\delta_{x}\left(y\right)dy=:h\left(x\right)
\]
where $dy$ is the volume measure on $M$. 

Then
for any $f\in C_{b}^{\infty}\left(H_{\mathcal{P}}\left(M\right)\right)$,
\[
\lim_{m\to\infty}\int_{H_{\mathcal{P}}\left(M\right)}f\left(\sigma\right)\delta_{x}^{\left(m\right)}\left(\sigma\left(1\right)\right)d\nu_{\mathcal{P}}^{1}\left(\sigma\right)=\int_{H_{\mathcal{P},x}\left(M\right)}f\left(\sigma\right)d\nu_{\mathcal{P},x}^{1}\left(\sigma\right).
\]

\end{corollary}

\begin{proof} Using the co-area formula (\ref{co}) with \[\left(H,M,p,g,f\right)=\left(H_\mathcal{P}\left(M\right),M,\Ep^{\mathcal{P}},\frac{1}{Z_\mathcal{P}^1}e^{-\frac{E}{2}},\delta_{x}^{\left(m\right)}\left(\sigma\left(1\right)\right)f\left(\sigma\right)\right),\] we have
\begin{align*}
\int_{H_{\mathcal{P}}\left(M\right)}\delta_{x}^{\left(m\right)}\left(\sigma\left(1\right)\right)f\left(\sigma\right)d\nu_{\mathcal{P}}^{1}\left(\sigma\right) & =\int_{M}dy\delta_{x}^{\left(m\right)}\left(y\right)\int_{H_{\mathcal{P},y}\left(M\right)}f\left(\sigma\right)d\nu_{\mathcal{P},y}^{1}\left(\sigma\right)\\
 & =\int_{M}h^f_{\mathcal{P}}\left(y\right)\delta_{x}^{\left(m\right)}\left(y\right)dy.
\end{align*}
From Theorem \ref{thm3.2.1} we know $h^f_{\mathcal{P}}\left(x\right)\in C\left(M\right),$
therefore: 
\begin{align*}
\lim_{m\to\infty}\int_{H_{\mathcal{P}}\left(M\right)}\delta_{x}^{\left(m\right)}\left(\sigma\left(1\right)\right)f\left(\sigma\right)d\nu_{\mathcal{P}}^{1}\left(\sigma\right) =\lim_{m\to\infty}\int_{M}h_{\mathcal{P}}\left(y\right)\delta_{x}^{\left(m\right)}\left(y\right)dy&=h^f_{\mathcal{P}}\left(x\right) \\&=\int_{H_{\mathcal{P},x}\left(M\right)}f\left(\sigma\right)d\nu_{\mathcal{P},x}^{1}\left(\sigma\right).
\end{align*}
\end{proof}

\section{Orthogonal Lifting Technique}\label{cha.4}

\subsection{The Orthogonal Lift $\tilde{X}_{\mathcal{P}}$ on $H_{\mathcal{P}}\left(M\right)$\label{cha.5}}
As a remainder, unless mentioned separately, $M$ is a complete Riemannian manifold with non--positive sectional curvature bounded below by $-N$. In this subsection we focus on the unpinned piecewise geodesic space $H_\mathcal{P}\left(M\right)$.

\subsubsection{A Parametrization of $T_{\sigma}H_{\mathcal{P}}\left(M\right)$\label{sec.5.1}}
Recall from Theorem \ref{def.6.2} that $Y\in \Gamma\left(TH_\mathcal{P}\left(M\right)\right)$ iff for each $\sigma\in H_\mathcal{P}\left(M\right)$, $J\left(\sigma, s\right):=u\left(\sigma,s\right)^{-1}Y\left(\sigma,s\right)$ satisfies (in the following equation we suppress $\sigma$)
\begin{equation}
J^{\prime\prime}\left(s\right)=R_{u\left(s\right)}\left(b^{\prime}\left(s_{i-1}+\right),J\left(s\right)\right)b^{\prime}\left(s_{i-1}+\right)\ \text{for}\ s\in\lbrack s_{i-1},s_{i})\text{ }i=1,...,n.\label{A}
\end{equation}
where $b=\phi\left(\sigma\right)\in H\left(\mathbb{R}^d\right)$ is the anti--rolling of $\sigma$.

From above we observe that $J$ can be parametrized by 
\begin{equation}
\left\{ J^{\prime}\left(s_{i}+\right)=k_{i}\right\} _{i=0}^{n-1}\label{B}
\end{equation}
where $\left(k_{0},k_{1},\dots,k_{n-1}\right)$ is an arbitrary element
of $\left(\mathbb{R}^d\right)^{n}.$ Proposition \ref{pro.6.4} explains
this parametrization in more detail.
\begin{proposition} \label{pro.6.4}If $\left(k_{0},k_{1},\dots,k_{n-1}\right)\in\left(\mathbb{R}^d\right)^{n}$, then the unique $J\left(\cdot\right)\in$$C\left(\left[0,1\right],\mathbb{R}^d\right)$ satisfying $(\ref{A})$ and $(\ref{B})$ above is given by
	\begin{equation}
	J\left(s\right)=\frac{1}{n}\sum_{i=0}^{l-1}f_{\mathcal{P},i+1}\left(s\right)k_{i}\text{   for }s\in \left[s_{l-1},s_l\right]\text{ , }1\leq l \leq n.\label{eq:-26}
	\end{equation}
\end{proposition}
\begin{proof}
	From the definition of $f_{\mathcal{P},i+1}$ (see Definition \ref{def.2.1-1}), $J$ in Eq. (\ref{eq:-26}) may be written as
	\[
	J\left(s\right)=C_{\mathcal{P},l}\left(s\right)\left[\sum_{i=0}^{l-2}C_{\mathcal{P},l-1}\dots C_{\mathcal{P},i+2}S_{\mathcal{P},i+1}k_{i}\right]+S_{\mathcal{P},l}\left(s\right)k_{l-1}\text{ when }s\in \left[s_{l-1},s_l\right].
	\]
	To finish the proof, we need only verify that $J$ is continuous, $J^\prime\left(s_i+\right)=k_i$ for $0\leq i\leq n-1$ and $J$ solves the Jacobi equation (\ref{A}).
	Since $C_{\mathcal{P},l}\left(s\right)$ and $S_{\mathcal{P},l}\left(s\right)$ satisfies Jacobi equation for $s\in \left[s_{l-1},s_l\right)$, $J$ satisfies (\ref{A}) and is continuous at $s\notin \mathcal{P}$. For each $s_l$, $1\leq l\leq n-1$, since $C_{\mathcal{P},l+1}\left(s_l\right)=I$, $S_{\mathcal{P},l+1}\left(s_l\right)=0$ and $J$ is right continuous on $\left[0,1\right]$, 
	\begin{align*}
	J\left(s_l-\right)=\lim_{s\uparrow s_l}J\left(s\right)&=C_{\mathcal{P},l}\left[\sum_{i=0}^{l-2}C_{\mathcal{P},l-1}\dots C_{\mathcal{P},i+2}S_{\mathcal{P},i+1}k_{i}\right]+S_{\mathcal{P},l}k_{l-1}\\&=C_{\mathcal{P},l+1}\left(s_l\right)\left[\sum_{i=0}^{l-1}C_{\mathcal{P},l}\dots C_{\mathcal{P},i+2}S_{\mathcal{P},i+1}k_{i}\right]+S_{\mathcal{P},l+1}\left(s_l\right)k_{l}\\&=J\left(s_l\right)=J\left(s_l+\right).
	\end{align*} 
	So $J$ is also continuous at partition points. Then since 
	\begin{align*}
	J^\prime\left(s_{l-1}+\right)&=C^\prime_{\mathcal{P},l}\left(s_{l-1}+\right)\left[\sum_{i=0}^{l-2}C_{\mathcal{P},l-1}\dots C_{\mathcal{P},i+2}S_{\mathcal{P},i+1}k_{i}\right]+S^\prime_{\mathcal{P},l}\left(s_{l-1}+\right)k_{l-1}\\&=0+I\cdot k_{l-1}=k_{l-1},
	\end{align*}
	$J$ satisfies (\ref{B}). The uniqueness of $J$ is easily seen from the uniqueness of solutions to linear ODE with initial values.
	
\end{proof}
\begin{definition} \label{def.6.8}For each $s\in \left[0,1\right]$, define $\mathbf{L}_s:\left(\mathbb{R}^d\right)^n\to \mathbb{R}^d$ as follows: for $s\in \left[s_{l-1},s_l\right]$, 
	\begin{equation}
	\mathbf{L}_s\left(k_{0},\dots,k_{n-1}\right)=\frac{1}{n}\sum_{i=0}^{l-1}f_{\mathcal{P},i+1}\left(s\right)k_{i}.\label{eq:-32}
	\end{equation}
\end{definition}
We now compute the adjoint of $\mathbf{L}_1$.
\begin{lemma}
	For any $v\in\mathbb{R}^{d}$,  let $\mathbf{L}_1^{\ast}:\mathbb{R}^d\to \left(\mathbb{R}^d\right)^n$ be the adjoint of $\mathbf{L}_1$, then
	\begin{equation}
	\mathbf{L}_1^{\ast}v=\frac{1}{n}\left(f_{\mathcal{P},1}^{*}\left(1\right)v,f_{\mathcal{P},2}^{*}\left(1\right)v,\dots,f_{\mathcal{P},n}^{*}\left(1\right)v\right),\label{equ.6.4}
	\end{equation}
	where $f_{\mathcal{P},i}^{*}\left(1\right)$ is the matrix adjoint of $f_{\mathcal{P},i}\left(1\right)$.
\end{lemma}
\begin{proof}
	Equation (\ref{equ.6.4}) immediately follows from the identity;
	\begin{equation}
	\left\langle \mathbf{L}_1\left(k_{0},\dots,k_{n-1}\right),v\right\rangle =\sum_{i=0}^{n-1}\left\langle \frac{1}{n}f_{\mathcal{P},i+1}\left(1\right)k_{i},v\right\rangle =\sum_{i=0}^{n-1}\left\langle k_{i},\frac{1}{n}f_{\mathcal{P},i+1}^{*}\left(1\right)v\right\rangle. \label{eq:-30}
	\end{equation}
\end{proof}
\begin{definition}\label{def Kp}
	We now define 
	\begin{equation}
	\mathbf{K}_{\mathcal{P}}\left(s\right)v:=n\mathbf{L}_s\left(\mathbf{L}_1^{\ast}v\right).\label{eq:-29}
	\end{equation}
\end{definition}
In particular, 
\begin{equation}
\mathbf{K}_{\mathcal{P}}\left(1\right)v=\frac{1}{n}\sum_{i=0}^{n-1}f_{\mathcal{P},i+1}\left(1\right)f_{\mathcal{P},i+1}^{*}\left(1\right)v.\label{eq:-28}
\end{equation}
Recall that given a matrix $A$, $eig\left(A\right)$ denotes the
eigenvalues of $A$.

\begin{lemma}[Invertibility of $\mathbf{K}_{\mathcal{P}}\left(1\right)$]\label{lem.6.10}If
	$M$ has non-positive sectional curvature, then 
	\begin{equation}
	eig\left(\mathbf{K}_{\mathcal{P}}\left(1\right)\right)\subset\lbrack1,\infty)\label{eq:-33}
	\end{equation}
	and thus $\mathbf{K}_{\mathcal{P}}\left(1\right)$
	is invertible. \end{lemma}

\begin{proof} 
	Denote $R_{u_s}\left(b^{\prime}\left(s_{i-1}+\right),\cdot\right)b^{\prime}\left(s_{i-1}+\right)$ by $A_{\mathcal{P},i}\left(s\right):H_\mathcal{P}\left(M\right)\to End\left(\mathbb{R}^d\right)$. Notice that $M$ having non-positive
	sectional curvature guarantees $A_{\mathcal{P},i}\left(s\right)$
	is positive semi-definite. Then apply Proposition \ref{prop A-1} to find, for
	any $i=1,\cdots,n$ and $v\in\mathbb{C}^{d}$, 
	\[
	\left\Vert C_{\mathcal{P},i}v\right\Vert \geq\left\Vert v\right\Vert \text{ and }\left\Vert S_{\mathcal{P},i}v\right\Vert \geq\frac{1}{n}\left\Vert v\right\Vert. 
	\]
	From these inequalities it follows that
	\begin{align*}
	\left\Vert f_{\mathcal{P},i}\left(1\right)v\right\Vert  & =n\left\Vert C_{\mathcal{P},n}C_{\mathcal{P},n-1}\cdots C_{\mathcal{P},i+1}S_{\mathcal{P},i}v\right\Vert \geq n\cdot\frac{1}{n}\left\Vert v\right\Vert=\left\Vert v\right\Vert. 
	\end{align*}
	So $f_{\mathcal{P},i}\left(1\right)$ is invertible and $\left\Vert f^*_{\mathcal{P},i}\left(1\right)^{-1}\right\Vert=\left\Vert f_{\mathcal{P},i}\left(1\right)^{-1}\right\Vert \leq 1$. Therefore for any $v\in\mathbb{C}^{d}$, 
	\[\left\Vert f^*_{\mathcal{P},i}\left(1\right)^{-1}v\right\Vert \leq \left\Vert v\right\Vert,\]
	now replace $v$ by $f^*_{\mathcal{P},i}\left(1\right)v$, we get $\left\Vert f^*_{\mathcal{P},i}\left(1\right)v\right\Vert\geq \left\Vert v \right\Vert$ and thus
	\begin{align*}
	\left\langle \mathbf{K}_{\mathcal{P}}\left(1\right)v,v\right\rangle  & =\frac{1}{n}\sum_{i=0}^{n-1}\left\langle f_{\mathcal{P},i+1}\left(1\right)f_{\mathcal{P},i+1}^{*}\left(1\right)v,v\right\rangle=\frac{1}{n}\sum_{i=0}^{n-1}\left\Vert f_{\mathcal{P},i+1}^{*}\left(1\right)v\right\Vert ^{2}\geq\frac{1}{n}\cdot n\left\Vert v\right\Vert ^{2}=\left\Vert v\right\Vert ^{2}\text{  }\forall v\in \mathbb{C}^d.
	\end{align*}
	This implies that 
	\[
	eig\left(\mathbf{K}_{\mathcal{P}}\left(1\right)\right)\subset\lbrack1,\infty).
	\]
	In particular, $\mathbf{K}_{\mathcal{P}}\left(1\right)$ is invertible.
\end{proof}

\subsubsection{Orthogonal Lift on $H_{\mathcal{P}}\left(M\right)$\label{sec.5.2}}

In this subsection we use the least square method to lift a vector field $X\in\Gamma\left(TM\right)$ to a vector field $\tilde{X}_{\mathcal{P}}\in\Gamma\left(TH_{\mathcal{P}}\left(M\right)\right)$, here lift means $\tilde{X}_{\mathcal{P}}$ satisfies Eq. $\left(\ref{equ.6.5}\right)$. \begin{theorem}[Orthogonal lift]\label{the.6.11}For
	all $X\in\Gamma\left(TM\right)$, there exists a unique vector field $\tilde{X}_{\mathcal{P}}\in\Gamma\left(TH_{\mathcal{P}}\left(M\right)\right)$ satisfying;
	\begin{enumerate}
		\item For all $h\in C^{1}\left(M\right)$,
		\begin{equation}
		\tilde{X}_{\mathcal{\mathcal{P}}}\left(h\circ \Ep\right)\left(\sigma\right)=\left(Xh\right)\left(\Ep\left(\sigma\right)\right)\text{, i.e. }{\Ep}_{*}{\tilde{X}_\mathcal{P}}\left(\sigma\right)=X\left(\sigma\left(1\right)\right).\label{equ.6.5}
		\end{equation}
		
		\item For all $\sigma\in H_{\mathcal{P}}\left(M\right),$ 
		\begin{equation}
		\left\Vert \tilde{X}_{\mathcal{P}}\left(\sigma\right)\right\Vert _{G_{\mathcal{P}}^{1}}=\inf\{\left\Vert Y\left(\sigma\right)\right\Vert {}_{G_{\mathcal{P}}^{1}}:Y\in\Gamma\left(TH_{\mathcal{P}}\left(M\right)\right),Y\text{ satisfies }\left(\ref{equ.6.5}\right)\}.\label{equ.6.6}
		\end{equation}
		
	\end{enumerate}
	In this paper $\tilde{X}_\mathcal{P}$ is referred to as the \textbf{orthogonal lift} of $X$ to $\left(H_\mathcal{P}(M),G^1_\mathcal{P}\right)$.
\end{theorem}
First we use the parametrization in Subsection \ref{sec.5.1} to characterize $\left\{ \operatorname*{Nul}\left({\Ep}_{\ast,\sigma}\right)\right\} {}^{\perp}$.
\begin{lemma}\label{lem min}
	Suppose $Y\in \Gamma\left(TH_\mathcal{P}\left(M\right)\right)$ with $k\left(\cdot\right):=u\left(\cdot\right)^{-1}Y\left(\cdot\right):H_{\mathcal{P}}\left(M\right)\to H\left(\mathbb{R}^d\right)$. Then $Y\in\left\{ \operatorname*{Nul}\left({\Ep}_{\ast}\right)\right\} ^{\perp}$
	iff 
	\[
	\left(k^{\prime}\left(s_{0}+\right),...,k^{\prime}\left(s_{n-1}+\right)\right)\in\left(\operatorname*{Nul}\mathbf{L}_1\right) ^{\perp}=\operatorname*{Ran}\left(\mathbf{L}_1^{\ast}\right).
	\]
\end{lemma}
\begin{proof}
	Given $Y\left(\cdot\right):=u\left(\cdot\right)k\left(\cdot\right)\text{ and }Z\left(\cdot\right):=u\left(\cdot\right)J\left(\cdot\right)\in \Gamma\left(TH_\mathcal{P}\left(M\right)\right)$, then 
	\begin{align*}
	\left\langle Y\left(\sigma\right),Z\left(\sigma\right)\right\rangle {}_{G_{\mathcal{P}}^{1}}=0 & \iff\sum_{i=0}^{n-1}\left\langle J^{\prime}\left(\sigma,s_{i}+\right),k^{\prime}\left(\sigma,s_{i}+\right)\right\rangle \Delta_{i+1}=0\\
	& \iff\sum_{i=0}^{n-1}\left\langle J^{\prime}\left(\sigma,s_{i}+\right),k^{\prime}\left(\sigma,s_{i}+\right)\right\rangle =0, 
	\end{align*}
	and
	\begin{equation*}
	Z\left(\sigma\right)\in\operatorname*{Nul}\left({\Ep}_{\ast,\sigma}\right)\iff {\Ep}_{\ast,\sigma}\left(Z\right)=u_1\left(\sigma\right)J\left(\sigma,1\right)=0\iff J\left(\sigma,1\right)=0.
	\end{equation*}
	Recall from Proposition \ref{pro.6.4} and Definition \ref{def.6.8} that $(\text{here we suppress }\sigma)$ \[J(1)=\mathbf{L}_1\left(J^{\prime}\left(s_{0}+\right),...,J^{\prime}\left(s_{n-1}+\right)\right),\]
	so 
	\begin{equation}
	J(1)=0\iff\left(J^{\prime}\left(s_{0}+\right),...,J^{\prime}\left(s_{n-1}+\right)\right)\in\operatorname*{Nul}\left(\mathbf{L}_1\right).\label{Eq.Nul}
	\end{equation}
	Since
	\[
	\sum_{i=0}^{n-1}\left\langle J^{\prime}\left(s_{i}+\right),k^{\prime}\left(s_{i}+\right)\right\rangle =\left\langle \left(J^{\prime}\left(s_{0}+\right),...,J^{\prime}\left(s_{n-1}+\right)\right),\left(k^{\prime}\left(s_{0}+\right),...,k^{\prime}\left(s_{n-1}+\right)\right)\right\rangle,
	\]
	so $Y\in\left\{ \operatorname*{Nul}\left({\Ep}_{\ast}\right)\right\} ^{\perp}$
	iff 
	\[
	\left(k^{\prime}\left(s_{0}+\right),...,k^{\prime}\left(s_{n-1}+\right)\right)\in\left\{ \operatorname*{Nul}\left(\mathbf{L}_1\right)\right\} ^{\perp}=\operatorname*{Ran}\left(\mathbf{L}_1^{\ast}\right).
	\]
\end{proof}

\begin{remark} \label{rem.6.12}According to $(\ref{equ.6.4})$, it is immediate that 
	\[
	\operatorname*{Ran}\left(\mathbf{L}_1^{\ast}\right)=\left\{ \left(\frac{1}{n}f_{\mathcal{P},1}^{*}\left(1\right)v,\frac{1}{n}f_{\mathcal{P},2}^{*}\left(1\right)v,\dots,\frac{1}{n}f_{\mathcal{P},n}^{*}\left(1\right)v\right),\text{ }~\forall~v\in\mathbb{R}^{d}\right\}.
	\]
\end{remark}
\begin{definition} \label{def.6.13}Given $X\in \Gamma\left(TM\right)$, define $\tilde{X}_{\mathcal{P}}\in \Gamma\left(TH_\mathcal{P}\left(M\right)\right)$ to be $\tilde{X}_{\mathcal{P}}\left(\cdot\right)=u_{(\cdot)}J_{\mathcal{P}}\left(\cdot\right)$
	where 
	\[
	J_{\mathcal{P}}\left(s\right):=\mathbf{K}_{\mathcal{P}}\left(s\right)\mathbf{K}_{\mathcal{P}}\left(1\right)^{-1}u_1{}^{-1}X\circ \Ep.
	\]
\end{definition}
\begin{proof}[Proof of Theorem \ref{the.6.11}]
	We will show $\tilde{X}_\mathcal{P}$ is the unique orthogonal lift of $X$. Since $T_{\sigma}H_{\mathcal{P}}\left(M\right)=\operatorname*{Nul}\left({\Ep}_{\ast,\sigma}\right)\oplus_{G_\mathcal{P}^1}\left\{ \operatorname*{Nul}\left({\Ep}_{\ast,\sigma}\right)\right\} {}^{\perp}$, given a lift $Z\in \Gamma\left(TH_\mathcal{P}\left(M\right)\right)$ of $X\in \Gamma\left(TM\right)$, its orthogonal projection to $\left\{ \operatorname*{Nul}\left({\Ep}_{\ast,\sigma}\right)\right\} {}^{\perp}$ is also a lift but with smaller $G_\mathcal{P}^1$ norm. So if $Z$ is an orthogonal lift, then $Z\in \left\{ \operatorname*{Nul}\left({\Ep}_{\ast}\right)\right\} {}^{\perp}$. From Lemma \ref{lem min} and Remark \ref{rem.6.12} it follows that if $k\left(\cdot\right):=u^{-1}\left(\cdot\right)Z\left(\cdot\right)$, then 
	\[\left(k^\prime\left(s_0\right),\dots,k^\prime\left(s_{n-1}\right)\right)=\left(\frac{1}{n}f_{\mathcal{P},1}^{*}\left(1\right)v,\frac{1}{n}f_{\mathcal{P},2}^{*}\left(1\right)v,\dots,\frac{1}{n}f_{\mathcal{P},n}^{*}\left(1\right)v\right)=\mathbf{L}_1^{\ast}v\]
	for some $v\in\mathbb{R}^{d}$. Then using Definition \ref{def Kp} and Proposition \ref{pro.6.4}, $k$ must have the following form, 
	\[
	k_s=\mathbf{K}_{\mathcal{P}}\left(s\right)v
	\]
	for some $v\in\mathbb{R}^d$ to be determined. 
	To specify $v$, we use condition (\ref{equ.6.5})
	\[
	\tilde{X}_{\mathcal{P}}\left(\sigma,1\right)=X\left(\sigma\left(1\right)\right).
	\]
	This implies $\mathbf{K}_{\mathcal{P}}\left(1\right)v=u^{-1}_1X\circ \Ep$. Since
	$\mathbf{K}_{\mathcal{P}}\left(1\right)$ is invertible, we can just
	pick $v$ to be $\mathbf{K}_{\mathcal{P}}\left(1\right)^{-1}u^{-1}_1X\circ \Ep$.
	\begin{definition}
		We will view $\tilde{X}_{\mathcal{P}}$ as a differential operator with domain, 
		\[\mathcal{D}\left(\tilde{X}_{\mathcal{P}}\right):=C^1_b\left(H_{\mathcal{P}}\left(M\right)\right).\]
	\end{definition}
	Here \[C^1_b\left(H_{\mathcal{P}}\left(M\right)\right):=\left\{f\in C^1\left(H_{\mathcal{P}}\left(M\right)\right):\exists C \text{ s.t. }\left\vert \left(df\right)_\sigma X\right\vert\leq C\left<X,X\right>^{\frac{1}{2}}_{G^1_\mathcal{P}}\text{ }\forall \sigma\in H_\mathcal{P}(M), X\in T_\sigma H_\mathcal{P}(M)\right\}.\]
	Since $C^1_b\left(H_{\mathcal{P}}\left(M\right)\right)$ is dense in $L^2\left(H_\mathcal{P}\left(M\right), \nu_\mathcal{P}^1\right)$, we can also view $\tilde{X}_{\mathcal{P}}$ as a densely defined operator on $L^2\left(H_\mathcal{P}\left(M\right), \nu_\mathcal{P}^1\right)$. 
\end{proof}

\subsubsection{Restricted Adjoint $\tilde{X}_{\mathcal{P}}^{tr,\nu_{\mathcal{P}}^{1}}$
	\label{sec.5.3}}
In this subsection we study $\tilde{X}_{\mathcal{P}}^{tr,\nu_{\mathcal{P}}^{1}}$---the adjoint of $\tilde{X}_\mathcal{P}$ with respect to $\nu_\mathcal{P}^1$ restricted to $\mathcal{D}\left(\tilde{X}_{\mathcal{P}}\right)$, i.e. we require $\mathcal{D}\left(\tilde{X}_{\mathcal{P}}^{tr,\nu_{\mathcal{P}}^{1}}\right)=\mathcal{D}\left(\tilde{X}_{\mathcal{P}}\right)$. 
\begin{lemma} \label{lem.8.1}Given $X\in \Gamma\left(TM\right)$, if $\tilde{X}_\mathcal{P}$ is the orthogonal lift of $X$, then
	\begin{equation}
	\tilde{X}_{\mathcal{P}}^{tr,\nu_{\mathcal{P}}^{1}}=-\tilde{X}_{\mathcal{P}}+M_{\int_{0}^{1}\left\langle J_{\mathcal{P}}^{\prime}\left(s+\right),b^{\prime}\left(s+\right)\right\rangle ds}-M_{div\tilde{X}_{\mathcal{P}}}\label{equ.8.1}
	\end{equation}
	where $M_{\centerdot}$ is the multiplication operator, $b$ is the anti-rolling of $\sigma$ and $div\tilde{X}_{\mathcal{P}}$
	is the divergence of $\tilde{X}_{\mathcal{P}}$ with respect to $vol_{G_{\mathcal{P}}^{1}}$.
\end{lemma}

\begin{proof} In this proof we identify the measure $\nu_{\mathcal{P}}^{1}$
	with the associated $nd$---form. So by \textquotedblleft Cartan's magic formula\textquotedblright, first assume $f\in C^1_b\left(H_{\mathcal{P}}\left(M\right)\right)$ with compact support,
	\[
	\mathcal{L}_{\tilde{X}_{\mathcal{P}}}\left(f\nu_{\mathcal{P}}^{1}\right)=d\left(i_{\tilde{X}_{\mathcal{P}}}\left(f\nu_{\mathcal{P}}^{1}\right)\right)+i_{\tilde{X}_{\mathcal{P}}}\left(d\left(f\nu_{\mathcal{P}}^{1}\right)\right).
	\]
	Since $f\nu_{\mathcal{P}}^{1}$ is a top degree form, $d\left(f\nu_{\mathcal{P}}^{1}\right)=0$.
	By Stokes' theorem, \[\int_{H_{\mathcal{P}}(M)}d\left(i_{\tilde{X}_{\mathcal{P}}}\left(f\nu_{\mathcal{P}}^{1}\right)\right)=0.\]
	Therefore we have: 
	\[
	\int_{H_{\mathcal{P}}\left(M\right)}\mathcal{L}_{\tilde{X}_{\mathcal{P}}}\left(f\nu_{\mathcal{P}}^{1}\right)=0
	\]
	and thus
	\begin{align}
	\int_{H_{\mathcal{P}}\left(M\right)}\left(\tilde{X}_{\mathcal{P}}f\right)d\nu_{\mathcal{P}}^{1} & =\int_{H_{\mathcal{P}}\left(M\right)}\mathcal{L}_{\tilde{X}_{\mathcal{P}}}\left(f\nu_{\mathcal{P}}^{1}\right)-\int_{H_{\mathcal{P}}\left(M\right)}f\mathcal{L}_{\tilde{X}_{\mathcal{P}}}\left(\nu_{\mathcal{P}}^{1}\right)\nonumber \\
	& =-\int_{H_{\mathcal{P}}\left(M\right)}f\mathcal{L}_{\tilde{X}_{\mathcal{P}}}\left(\nu_{\mathcal{P}}^{1}\right).\label{equ.8.2}
	\end{align}
	Recall that $\nu_{\mathcal{P}}^{1}=\frac{1}{Z_{\mathcal{P}}^{1}}e^{-\frac{1}{2}E}vol_{G_{\mathcal{P}}^{1}}$,
	so 
	\begin{equation}
	\mathcal{L}_{\tilde{X}_{\mathcal{P}}}\left(\nu_{\mathcal{P}}^{1}\right)=\left[\tilde{X}_{\mathcal{P}}\left(\frac{1}{Z_{\mathcal{P}}^{1}}e^{-\frac{1}{2}E}\right)\right]vol_{G_{\mathcal{P}}^{1}}+\left(div\tilde{X}_{\mathcal{P}}\right)\nu_{\mathcal{P}}^{1}.\label{equ.8.3}
	\end{equation}
	In (\ref{equ.8.3})
	\begin{align}
	\tilde{X}_{\mathcal{P}}\left(\frac{1}{Z_{\mathcal{P}}^{1}}e^{-\frac{1}{2}E}\right) & =-\frac{1}{2}\tilde{X}_{\mathcal{P}}\left(E\right)\frac{1}{Z_{\mathcal{P}}^{1}}e^{-\frac{1}{2}E}\nonumber \\
	& =-\int_{0}^{1}\left\langle \sigma^{\prime}\left(s+\right),\frac{\nabla\tilde{X}_{\mathcal{P}}}{ds}\left(s+\right)\right\rangle ds\frac{1}{Z_{\mathcal{P}}^{1}}e^{-\frac{1}{2}E}\nonumber \\
	& =-\int_{0}^{1}\left\langle b^{\prime}\left(s+\right),J_{\mathcal{P}}^{\prime}\left(s+\right)\right\rangle ds\frac{1}{Z_{\mathcal{P}}^{1}}e^{-\frac{1}{2}E}.\label{equ.8.4}
	\end{align}
	Combining $\left(\ref{equ.8.2}\right)$, $\left(\ref{equ.8.3}\right)$ and $\left(\ref{equ.8.4}\right)$ we get, if $f\in C^1_b\left(H_{\mathcal{P}}\left(M\right)\right)$ with compact support, then
	\begin{equation}
	\int_{H_\mathcal{P}\left(M\right)}\tilde{X}_\mathcal{P}fd\nu_\mathcal{P}^1=\int_{H_\mathcal{P}\left(M\right)}f\cdot\left(\tilde{X}_{\mathcal{P}}^{tr,\nu_{\mathcal{P}}^{1}}1\right)d\nu_\mathcal{P}^1,\label{p1}
	\end{equation}
	where $\tilde{X}_{\mathcal{P}}^{tr,\nu_{\mathcal{P}}^{1}}$ is defined in Eq. $(\ref{equ.8.1})$.
	For the general case choose a cut--off function $\phi\in C_{0}^{\infty}\left(\mathbb{R},\left[0,1\right]\right)$
	such that $\phi\equiv1$ on $\left[-1,1\right]$ and $\phi\equiv0$
	on $\mathbb{R}/\left[-2,2\right]$. Let $f_n:=f\cdot \phi\left(\frac{E}{n}\right)$, observe, using product rule, that
	\begin{equation}
	\tilde{X}_\mathcal{P}f_n=\phi\left(\frac{E}{n}\right)\cdot \tilde{X}_\mathcal{P}f+\frac{2}{n}f\cdot \phi^\prime\left(\frac{E}{n}\right)\int_{0}^{1}\left\langle J_{\mathcal{P}}^{\prime}\left(s+\right),b^{\prime}\left(s+\right)\right\rangle ds,\label{p2}
	\end{equation}
	so 
	$\tilde{X}_\mathcal{P}f_n\to \tilde{X}_\mathcal{P}f$ as $n\to \infty $ $\nu_{\mathcal{P}}^1$ a.s. 
	
	Using Proposition \ref{pro.8.5} and Lemma \ref{inf} we have for any $q\geq 1$, there exists $M=M(q)>0$ such that $\forall \mathcal{P}$ with $\left\vert\mathcal{P}\right\vert\leq \frac{1}{M}$, 
	\[\int_{0}^{1}\left\langle J_{\mathcal{P}}^{\prime}\left(s+\right),b^{\prime}\left(s+\right)\right\rangle ds\in L^{q}\left(H_\mathcal{P}\left(M\right),d\nu_{\mathcal{P}}^1\right).\]
	Since $f$ has bounded differential, from Definition \ref{def.6.13} and \ref{def Kp} we have
	\begin{align*}
	\left|\tilde{X}_{\mathcal{P}}f\right|&\leq C\left<\tilde{X}_{\mathcal{P}},\tilde{X}_{\mathcal{P}}\right>_{G^1_\mathcal{P}}\\&=C\sum_{i=1}^{n}\left<J^\prime_\mathcal{P}(s_{i-1}+),J^\prime_\mathcal{P}(s_{i-1}+)\right>\Delta_i\\&=\frac{C}{n}\left\Vert f^*_{\mathcal{P},i-1}(1)\mathbf{K}_{\mathcal{P}}\left(1\right)^{-1}u^{-1}_1X\circ \Ep\right\Vert\\&\leq C\max_{1\leq i\leq n}\left\Vert f^*_{\mathcal{P},i-1}(1)\right\Vert^2\left\Vert\mathbf{K}_{\mathcal{P}}\left(1\right)^{-1}u^{-1}_1X\circ \Ep\right\Vert^2
	\end{align*}
Lemma \ref{lem.6.10} states $\mathbf{K}_{\mathcal{P}}\left(1\right)^{-1}$ is bounded. Then utilizing Lemma \ref{lem.7.8} we have for any $q\geq 1$, there exists $M=M(q)>0$ such that $\forall \mathcal{P}$ with $\left\vert\mathcal{P}\right\vert\leq \frac{1}{M}$, 
  \[\tilde{X}_\mathcal{P}f\in L^{q}\left(H_\mathcal{P}\left(M\right),d\nu_{\mathcal{P}}^1\right).\]
	Lemma \ref{lem7.4} shows that $\tilde{X}_{\mathcal{P}}^{tr,\nu_{\mathcal{P}}^{1}}1\in L^q\left(H_\mathcal{P}\left(M\right),d\nu_{\mathcal{P}}^1\right)$ provided $\left\vert\mathcal{P}\right\vert\leq \frac{1}{M}$ for some $M=M(q)$, so applying DCT to both sides of Eq.(\ref{p1}) with $f_n\to f$ gives Eq. (\ref{equ.8.1}).
\end{proof} 

\subsubsection{Computing $div\tilde{X}_{\mathcal{P}}$\label{sec.5.4}}Recall from Notation \ref{not3.1.8} and Remark \ref{rem.8.3-1} that
\[
X^{h_{\alpha,i}}\left(\sigma,s\right)=u\left(\sigma,s\right)\frac{1}{\sqrt{n}}f_{\mathcal{P},i}\left(s\right)e_{\alpha}\text{ , }1\leq \alpha\leq d\text{ , }1\leq i\leq n
\]
is an orthonormal frame on $\left(TH_{\mathcal{P}}\left(M\right),G_{\mathcal{P}}^{1}\right)$. Using this orthonormal frame, one can get an expression of $div\tilde{X}_{\mathcal{P}}$.
\begin{proposition}\label{prop.5.4} Let $\tilde{X}_{\mathcal{P}}$ be the orthogonal lift of $X\in \Gamma\left(TM\right)$, then
	\begin{equation}
	div\tilde{X}_{\mathcal{P}}=\sum_{\alpha=1}^{d}\sum_{j=1}^{n}\left\langle X^{h_{\alpha,j}}J_{\mathcal{P}}^{\prime}\left(s_{j-1}+\right),e_{\alpha}\right\rangle \sqrt{\Delta_{j}} \label{div1}
	\end{equation}
\end{proposition} \begin{proof}By definition
	\begin{equation}
	div\tilde{X}_{\mathcal{P}}=\sum_{\alpha=1}^{d}\sum_{j=1}^{n}\left\langle \left[X^{h_{\alpha,j}},\tilde{X}_{\mathcal{P}}\right],X^{h_{\alpha,j}}\right\rangle_{G_{\mathcal{P}}^{1}},\label{div2}
	\end{equation}
	where $\left[\cdot,\cdot\right]$ is the Lie bracket of vector fields.
	
	Now fix $j$ and $\alpha$, notice that $\tilde{X}_{\mathcal{P}}=X^{J_{\mathcal{P}}}$,
	apply Theorem 3.5 in \cite{Andersson1999} to find
	\[
	\left[X^{h_{\alpha,j}},\tilde{X}_{\mathcal{P}}\right]=X^{f\left(h_{\alpha,j},J_{\mathcal{P}}\right)},
	\]
	where 
	\begin{align*}
	f_s\left(h_{\alpha,j},J_{\mathcal{P}}\right)=&\left(X^{h_{\alpha,j}}J_{\mathcal{P}}\right)\left(s\right)-\left(X^{J_{\mathcal{P}}}h_{\alpha,j}\right)\left(s\right)+ q_{s}\left(X^{h_{\alpha,j}}\right)J_{\mathcal{P}}\left(s\right)-q_{s}\left(X^{J_{\mathcal{P}}}\right)h_{\alpha,j}\left(s\right)
	\end{align*}
	and 
	\[
	q_{s}\left(X^{f}\right)=\int_{0}^{s}R_{u_r}\left(b^{\prime}\left(r+\right),f\left(r\right)\right)dr.
	\]
	Therefore 
	\begin{align}
	\left\langle \left[X^{h_{\alpha,j}},\tilde{X}_{\mathcal{P}}\right],X^{h_{\alpha,j}}\right\rangle_{G_{\mathcal{P}}^{1}}   & =\sum_{i=1}^{n}\left\langle f^{\prime},h_{\alpha,j}^{\prime}\right\rangle _{s_{i-1}+}\Delta_{i}\label{equ.8.6}\\
	& =\sum_{i=1}^{n}\left\langle \left(X^{h_{\alpha,j}}J_{\mathcal{P}}\right)^{\prime}-\left(X^{J_{\mathcal{P}}}h_{\alpha,j}\right)^{\prime},h_{\alpha,j}^{\prime}\right\rangle _{s_{i-1}+}\Delta_{i}\nonumber \\
	& +\sum_{i=1}^{n}\left\langle \left(q_{s}\left(X^{h_{\alpha,j}}\right)J_{\mathcal{P}}\left(s\right)\right)^{\prime}-\left(q_{s}\left(X^{J_{\mathcal{P}}}\right)h_{\alpha,j}\left(s\right)\right)^{\prime},h_{\alpha,j}^{\prime}\right\rangle _{s_{i-1}+}\Delta_{i}\nonumber 
	\end{align}
	Here $'$ is the derivative with respect to (time) $s$.
	
	Since $h_{\alpha,j}^{\prime}\left(s_{i-1}+\right)$
	is independent of $\sigma$, so 
	\[
	\left(X^{J_{\mathcal{P}}}h_{\alpha,j}\right)^{\prime}\left(s_{i-1}+\right)=X^{J_{\mathcal{P}}}\left(\sigma\to h_{\alpha,j}^{\prime}\left(\sigma,s_{i-1}+\right)\right)=0,
	\]
	and thus
	\begin{equation}
	\sum_{i=1}^{n}\left\langle \left(X^{J_{\mathcal{P}}}h_{\alpha,j}\right)^{\prime},h_{\alpha,j}^{\prime}\right\rangle _{s_{i-1}+}\Delta_{i}=0.\label{equ.8.5}
	\end{equation}
	We now claim that 
	\[
	\left(q_{s}\left(X^{h_{\alpha,j}}\right)J_{\mathcal{P}}\left(s\right)\right)^{\prime}=q_{s}^{\prime}\left(X^{h_{\alpha,j}}\right)J_{\mathcal{P}}\left(s\right)+q_{s}\left(X^{h_{\alpha,j}}\right)J_{\mathcal{P}}^{\prime}\left(s\right)=0\text{ for }s\in \mathcal{P}.
	\]
	Since 
	\[
	h_{\alpha,j}^{\prime}\left(s_{i-1}+\right)\neq0\text{ iff }i=j,
	\]
	and when $i=j$, 
	\[
	h_{\alpha,j}\left(s\right)=0\text{ for }s\leq s_{i-1},
	\]
	so both $q_{s_{i-1}}^{\prime}\left(X^{h_{\alpha,j}}\right)=0$ and
	$q_{s_{i-1}}\left(X^{h_{\alpha,j}}\right)=0$. It then follows that the claim is true and 
	\begin{equation}
	\sum_{i=1}^{n}\left\langle \left(q_{s}\left(X^{h_{\alpha,j}}\right)J_{\mathcal{P}}\left(s\right)\right)^{\prime},h_{\alpha,j}^{\prime}\right\rangle _{s_{i-1}+}\Delta_{i}=0,\label{equ.8.7}
	\end{equation}
	and 
	\begin{equation}
	\sum_{i=1}^{n}\left\langle q_{s}^{\prime}\left(X^{J_{\mathcal{P}}}\right)h_{\alpha,j}\left(s\right),h_{\alpha,j}^{\prime}\right\rangle _{s_{i-1}+}\Delta_{i}=0.\label{equ.8.8}
	\end{equation}
	Lastly because $q_{s}\left(X^{J_{\mathcal{P}}}\right)$ is skew-symmetric,
	\begin{equation}
	\sum_{i=1}^{n}\left\langle q_{s}\left(X^{J_{\mathcal{P}}}\right)h_{\alpha,j}^{\prime},h_{\alpha,j}^{\prime}\right\rangle _{s_{i-1}+}\Delta_{i}=0.\label{equ.8.9}
	\end{equation}
	Combining Eq.(\ref{equ.8.5}), (\ref{equ.8.7}), (\ref{equ.8.8}) and (\ref{equ.8.9}) shows,  
	\begin{align}
	\left\langle \left[X^{h_{\alpha,j}},\tilde{X}_{\mathcal{P}}\right],X^{h_{\alpha,j}}\right\rangle_{G_{\mathcal{P}}^{1}}  & =\sum_{i=1}^{n}\left\langle X^{h_{\alpha,j}}J_{\mathcal{P}}^{\prime},h_{\alpha,j}^{\prime}\right\rangle _{s_{i-1}+}\Delta_{i}=\left\langle X^{h_{\alpha,j}}J_{\mathcal{P}}^{\prime}\left(s_{j-1}+\right),e_{\alpha}\right\rangle \sqrt{\Delta_{j}}.\label{equ.8.20}
	\end{align}
	Summing Eq.(\ref{equ.8.20}) on $\alpha$ and $j$ while making use of (\ref{div2}) gives (\ref{div1}). 
\end{proof}

\subsection{The Orthogonal Lift $\tilde{X}$ on $W_o\left(M\right)$}

\begin{definition}[Cameron-Martin vector field]
	A \textbf{Cameron-Martin process}, $h$, is an $\mathbb{R}^d$-valued process on $W_o(M)$ such that $s\to h(s)$ is in $H(\mathbb{R}^d)\text{ }\nu-a.s.$ and a $TM$-valued process $X^h$ on $(W_o(M),\nu)$ is called a \textbf{Cameron-Martin vector field} (denote this space by $\mathcal{X}$) if $\pi(X_s)=\Sigma_s\text{ }\nu-a.s.$, $h(s):=\tilde{u}^{-1}_sX^h_s$ is a Cameron-Martin process and\[ 
	\left<X^h,X^h\right>_{\mathcal{X}}:=\mathbb{E}\left[ \left\Vert h\right\Vert^2_{H(\mathbb{R}^d)}\right]<\infty.\]
\end{definition}
Cameron-Martin vector field is the key concept in path space analysis. In this section we are going to introduce a non-adapted Cameron-Martin vector field (see Definition \ref{def Oc}) which \textquotedblleft lift\textquotedblright a vector field on a manifold $M$ to a \textquotedblleft vector field\textquotedblright on the corresponding path space $W_o(M)$. 
\begin{definition} \label{def.T}Define $\tilde{T}_{\left(\cdot\right)}:\left[0,1\right]\times W_o\left(M\right)\to End\left(\mathbb{R}^d\right)$ to be
	the solution to the following initial value problem: 
	\begin{equation}
	\begin{cases}
	\frac{d}{ds}\tilde{T}_s+\frac{1}{2}Ric_{\tilde{u}_s}\tilde{T}_s=0\\
	\tilde{T}_0=I.
	\end{cases}\label{equ.4.3-1}
	\end{equation}
\end{definition}
\begin{definition}\label{def.4.4-1}  Using $\tilde{T}_s$, we define $\mathbf{\tilde{K}}:\left[0,1\right]\times W_o\left(M\right)\to End\left(\mathbb{R}^d\right)$:
	\begin{equation}
	\mathbf{\tilde{K}}_s:=\tilde{T}_s\left[\int_{0}^{s}\tilde{T}_r^{-1}\left(\tilde{T}_r^{-1}\right)^*dr\right]\tilde{T}^{\ast}_1.\label{equ.4.6-1}
	\end{equation}
\end{definition}
\begin{remark}
	Both $\tilde{T}$ and $\mathbf{\tilde{K}}$ are defined up to $\nu-$equivalence. We can pick a version at first place in order to avoid stating $\nu-a.s.$ in the following results. 
\end{remark}
\begin{lemma} \label{lem.4.3}For all $s\in\left[0,1\right]$, $\tilde{T}_s$
	is invertible. Further both $\underset{0\leq s\leq1}{\sup}\left\Vert \tilde{T}_s\right\Vert$ and
	$\underset{0\leq s\leq1}{\sup}\left\Vert \tilde{T}_s^{-1}\right\Vert $
	are bounded by $e^{\frac{1}{2}\left(d-1\right)N}$, where $\left(d-1\right)N$
	is a bound of $\left\Vert\operatorname*{Ric}\right\Vert.$
\end{lemma}
\begin{proof} Denote by $U_s\in \operatorname*{End}\left(\mathbb{R}^{d}\right)$ the solution to the following initial value problem:
	\[\frac{d}{ds}U_s=-\frac{1}{2}U_sRic_{\tilde{u}_s},\text{ }U_0=I,\label{equ.4.4}\]
	then direct computation shows that $Y_s:=\tilde{T}_sU_s\in \operatorname*{End}\left(\mathbb{R}^{d}\right)$ satisfies 
	\[\frac{d}{ds}Y_s=\frac{1}{2}\left(Ric_{\tilde{u}_s}Y_s-Y_sRic_{\tilde{u}_s}\right),\text{ }Y_0=I.\]
	By the uniqueness of solutions for linear ODE, we get $Y_s\equiv I$,  
	and this shows that $U_s$ is a left inverse to $\tilde{T}_s$.
	As we are in finite dimensions it follows that $\tilde{T}_s^{-1}$
	exists and is equal to $U_s$. The stated bounds now follow
	by Gronwall's inequality. 
\end{proof}
\begin{lemma} \label{lem.4.6}$\tilde{\mathbf{K}}_1$ is invertible
	and $\left\Vert \tilde{\mathbf{K}}_1^{-1}\right\Vert \leq e^{\left(d-1\right)N}$, provided $\left\Vert\operatorname*{Ric}\right\Vert\leq \left(d-1\right)N$.
\end{lemma}

\begin{proof} 
	Since
	\[
	\tilde{\mathbf{K}}_1:=\int_{0}^{1}\left(\tilde{T}_1\tilde{T}_r^{-1}\right)\left(\tilde{T}_1\tilde{T}_r^{-1}\right)^{\ast}dr
	\]
	is a symmetric positive semi-definite operator such that
	\[
	\left\langle \tilde{\mathbf{K}}_1v,v\right\rangle =\int_{0}^{1}\left\Vert \left(\tilde{T}_1\tilde{T}_r^{-1}\right)^{\ast}v\right\Vert ^{2}dr\text{  }\forall v\in\mathbb{C}^{d}.
	\]
	Apply Lemma \ref{lem.4.3} to the expression given;
	\begin{align*}
	\left\langle \tilde{\mathbf{K}}_1v,v\right\rangle  & \geq\int_{0}^{1}e^{-\left(d-1\right)N}\left\Vert \left(\tilde{T}_r^{-1}\right)^{\ast}v\right\Vert ^{2}dr\geq\int_{0}^{1}e^{-2\left(d-1\right)N}\left\Vert v\right\Vert ^{2}dr=e^{-2\left(d-1\right)N}\left\Vert v\right\Vert ^{2}
	\end{align*}
	from which it follows that $eig\left(\tilde{\mathbf{K}}_1\right)\subset\lbrack e^{-\left(d-1\right)N},\infty)$ and $\left\Vert\tilde{\mathbf{K}}_1^{-1}\right\Vert=\frac{1}{\min\left\{\lambda:\lambda\in eig(\tilde{\mathbf{K}}_1)\right\}}\leq e^{(d-1)N}$.
\end{proof}
\begin{definition} \label{def.4.7-1}For each $X\in\Gamma\left(\tilde{T}M\right)$
	define two $\nu-$equivalent maps $\tilde{H}:W_o\left(M\right)\to \mathbb{R}^d$ and $\tilde{J}:W_o\left(M\right)\to H\left(\mathbb{R}^d\right)$ by
	\begin{equation}
	\tilde{H}=\tilde{u}_1^{-1}X\circ \Ep\label{equ.4.8-1}
	\end{equation}
	and 
	\begin{equation}
	\tilde{J}_s:=\mathbf{\tilde{K}}_s\mathbf{\tilde{K}}_1^{-1}\tilde{H}\text{ for }s\in \left[0,1\right].\label{equ.4.9-1}
	\end{equation}
\end{definition}
\begin{notation} \label{not.4.9}Given a measurable function $h:W_o\left(M\right)\to H\left(\mathbb{R}^d\right)$, let $Z_h:W_o\left(M\right)\to H\left(\mathbb{R}^d\right)$ be the solution to the following initial value problem:
	\[
	\begin{cases}
	{Z_h}^{\prime}\left(s\right)=-\frac{1}{2}Ric_{\tilde{u}_s}Z_h\left(s\right)+h^{\prime}_s\\
	Z_h\left(0\right)=0.
	\end{cases}
	\]

\end{notation}

\begin{definition}[Orthogonal Lift on $W_o(M)$] \label{def Oc}For any $X\in\Gamma\left(TM\right),$
	define $\tilde{X}\in \mathcal{X}$ as follows.
	\[
	\tilde{X}_s=X^{Z_{\Phi}}_s:=\tilde{u}_sZ_{\Phi}\left(s\right)\text{ for }0\leq s \leq 1
	\]
	where 
	\[
	\Phi_s=\int_{0}^{s}\left(\tilde{T}_{\tau}^{-1}\right){}^{\ast}\left[\int_{0}^{1}\left(\tilde{T}^{\ast}_r\tilde{T}_r\right){}^{-1}dr\right]^{-1}\tilde{T}_1^{-1}\tilde{H}d\tau.
	\]
Given $f\in \mathcal{FC}_b^1$, define the \textbf{gradient operator} $Df\in \mathcal{X}$ as follows,
\begin{equation}
D_sf:=\tilde{u}_s\sum_{i=1}^{n}(s\wedge s_i)\tilde{u}^{-1}_{s_i}grad_iF\label{gradient}
\end{equation}
where $F(\Sigma_{s_1},\cdots,\Sigma_{s_n})$ is a representation of $f$ and $grad_iF$ is the differential of $F$ with respect to the $ith$ variable. 

Then we define $\tilde{X}f:=\left<Df,\tilde{X}\right>_{G^1}$.
\end{definition}
Since $\mathcal{FC}_b^1$ is dense in $L^2\left(W_o(M),\nu\right)$, $\tilde{X}$ can be viewed as a densely defined operator on $L^2\left(W_o(M),\nu\right)$ which admits an integration by parts formula as below.
\begin{theorem}\label{lem.5.1}
	For any $f,g\in \mathcal{FC}^1_b$, we have
	\begin{equation}
	\mathbb{E}_\nu\left[\tilde{X}f\cdot g\right]=\mathbb{E}_\nu\left[f\cdot \tilde{X}^{tr,\nu}g\right],
	\end{equation}
	where 
	\begin{align*}
	\tilde{X}^{\operatorname{tr},\nu}= & -\tilde{X}+\sum_{\alpha=1}^{d}\left\langle \tilde{C}\tilde{H},e_{\alpha}\right\rangle \int_{0}^{1}\left\langle \left(\tilde{T}_s^{-1}\right)^{\ast}e_{\alpha},d\beta_{s}\right\rangle+\sum_{\alpha=1}^{d}\left\langle -X^{Z_{\alpha}}\left(\tilde{C}\tilde{H}\right),e_{\alpha}\right\rangle
	\end{align*}
	and
	\[\tilde{C}=\left[\int_{0}^{1}\left(\tilde{T}^{\ast}_r\tilde{T}_r\right){}^{-1}dr\right]^{-1}\tilde{T}_1^{-1}.\]
\end{theorem}
\begin{proof}
	See Lemma 4.23 of IVP.
\end{proof}
\begin{remark}
The orthogonal lift $\tilde{X}$ on $W_o(M)$ can be viewed as a stochastic extension of the orthogonal lift in the sense of Theorem \ref{the.6.11} where the path space is the curved Cameron-Martin space $H(M)$ and the Riemannian metric is a damped metric related to Ricci curvature. Interested readers may refer to IVP for more details in this topic.
\end{remark}

\section{Convergence Result \label{cha.6}}
In this section $M$ is a complete Riemannian manifold with non--positive and bounded sectional curvature. Other conditions will be mentioned specifically in theorems if needed.   
First we modify and abuse a few notations we have defined before in order to avoid messy arguments.
\begin{notation}\label{Not1}
Recall that $\beta:W_o\left(M\right)\to W_0\left(\mathbb{R}^d\right)$ is the Brownian motion on $\mathbb{R}^d$ defined in Definition \ref{def sar}. We have also defined $\beta_\mathcal{P}:W_o\left(M\right)\to H_\mathcal{P}\left(\mathbb{R}^d\right)$ to be the linear approximation to Brownian motion on $\mathbb{R}^d$ as in Notation \ref{not2.26}. Now denote by $u_\mathcal{P}:=\eta\circ \beta_\mathcal{P}$ the development map of $\beta_\mathcal{P}$. Notice that $\phi\circ \beta_\mathcal{P}\in H_\mathcal{P}\left(M\right)$--$\nu$ a.s, here $\phi$ is the rolling map onto $H\left(M\right)$. So after identifying  $C_{\mathcal{P},i}$, $S_{\mathcal{P},i}$
and hence $f_{\mathcal{P},i}$ with $C_{\mathcal{P},i}\circ\phi\circ \beta_\mathcal{P}$,
$S_{\mathcal{P},i}\circ\phi\circ \beta_\mathcal{P}$
and $f_{\mathcal{P},i}\circ\phi\circ \beta_\mathcal{P}$,
we can view them as maps from $W_{o}\left(M\right)$ to $End\left(\mathbb{R}^d\right)$. The point here is to make the notations short and it should not cause confusions after this explanation. 
\end{notation}
\iffalse
\begin{remark}
Let $L^{\infty-}\left(W_o\left(M\right)\right):=\cap_{q\geq 1}L^{q}\left(W_o\left(M\right)\right)$. This is a Frechet space and for any $\left\{f_n\right\}_n$ and $f$ in $L^{\infty-}\left(W_o\left(M\right)\right)$, $f_n\to f$ as $n\to \infty$ in $L^{\infty-}\left(W_o\left(M\right)\right)$ iff $f_n\to f$ in $L^{q}\left(W_o\left(M\right)\right)\text{  }\forall q\geq 1$. 
\end{remark}
\fi
\begin{convention}
We use $C$ to denote a generic constant. It can vary from line to line. In this section it depends only on an upper bound of the mesh size $\left|\mathcal{P}\right|:=\frac{1}{n}$ of the partition $\mathcal{P}$ (We may allow $C$ to depend on some other factors as well, but this is good enough for our purpose of taking the limit as $\left|\mathcal{P}\right|\to 0$. )
\end{convention}
\subsection{Convergence of $\tilde{X}_{\mathcal{P}}$ to $\tilde{X}$}
\subsubsection{Some Useful Estimates for $\left\{ C_{\mathcal{P},i}\right\} _{i=1}^{n}$ and
$\left\{ S_{\mathcal{P},i}\right\} _{i=1}^{n}$\label{sub.6.1.1}}

We apply Proposition \ref{prop A-2} to get some commonly used estimates listed as Lemmas \ref{lem1}.
\begin{lemma}\label{lem1}For any $i\in\{1,...,n\}$ and $s\in\left[s_{i-1},s_{i}\right]$, we have
\begin{equation}
\left\vert C_{\mathcal{P},i}\left(s\right)\right\vert \leq\cosh\left(\sqrt{N}\left\vert \Delta_{i}\beta\right\vert \right)\leq e^{\frac{1}{2}N\left\vert \Delta_{i}\beta\right\vert ^{2}},\label{lem.6.1}
\end{equation}
\begin{equation}
\left\vert S_{\mathcal{P},i}(s)\right\vert \leq\sqrt{N}\left\vert \Delta_{i}\beta\right\vert e^{\frac{1}{2}N\left\vert \Delta_{i}\beta\right\vert ^{2}},\label{lem.6.2}
\end{equation}
\begin{equation}
\left\vert S_{\mathcal{P},i}-\Delta_{i}I\right\vert \leq\frac{N\left\vert \Delta_{i}\beta\right\vert ^{2}\Delta_{i}}{6}e^{\frac{1}{2}N\left\vert \Delta_{i}\beta\right\vert ^{2}},\label{lem.6.3}
\end{equation}
\begin{equation}
\left\vert C_{\mathcal{P},i}-I\right\vert \leq\frac{N\left\vert \Delta_{i}\beta\right\vert ^{2}}{2}e^{\frac{1}{2}N\left\vert \Delta_{i}\beta\right\vert ^{2}}\label{lem.6.4}
\end{equation}
\end{lemma}
\begin{lemma} \label{lem.6.5}For all $\gamma\in\left(0,\frac{1}{2}\right),$
define $K_{\gamma}:=\underset{s,t\in\left[0,1\right],s\neq t}{sup}\left\{ \frac{\left\vert \beta_{t}-\beta_{s}\right\vert }{\left\vert t-s\right\vert ^{\gamma}}\right\} $,
then there exists an $\epsilon_{\gamma}>0$ such that $\mathbb{E}\left[e^{\epsilon K_{\gamma}^{2}}\right]<\infty.$
\end{lemma}

\begin{proof} See Fernique's Theorem (Theorem 3.2) in \cite{Kuo1975}.
\end{proof}

\begin{remark} From Lemma \ref{lem.6.5}, it is easy to see any polynomial
of $\epsilon K_{\gamma}$ has finite moments of all orders. \end{remark}

\subsubsection{Size Estimates of $f_{\mathcal{P},i}\left(s\right)$ \label{sub.6.1.2}}
Recall from Definition \ref{def.2.1-1} that $f_{\mathcal{P},i}:W_o\left(M\right)\times \left[0,1\right]\to End\left(\mathbb{R}^d\right)\text{  }0\leq i\leq n$ is given by
\[
f_{\mathcal{P},i}\left(s\right)=\begin{cases}
0 & s\in\left[0,s_{i-1}\right]\\
\frac{S_{\mathcal{P},i}\left(s\right)}{\Delta_{i}} & s\in\left[s_{i-1},s_{i}\right]\\
\frac{C_{\mathcal{P},j}\left(s\right)C_{\mathcal{P},j-1}\cdot\cdots\cdot C_{\mathcal{P},i+1}S_{\mathcal{P},i}}{\Delta_{i}} & s\in\left[s_{j-1},s_{j}\right]\text{ for }j=i+1,\cdots,n

\end{cases}
\]
with the convention that $S_{\mathcal{P},0}\equiv\left\vert \mathcal{P}\right\vert I$ and $f_{\mathcal{P},0}\equiv I$.

Using the estimates in Subsection \ref{sub.6.1.1},
it is easy to get an estimate of $f_{\mathcal{P},i}\left(s\right)$.

\begin{lemma} \label{lem.7.8}Recall from the begining of this section that $n:=\frac{1}{\left|\mathcal{P}\right|}$ and $N$ is the sectional curvature bound. For each $q\geq1$, we have
\begin{equation}
\sup_{n\geq 2qN}\mathbb{E}\left[\underset{i,j\in\left\{ 0,\cdots,n\right\} }{\sup}\left\vert f_{\mathcal{P},i}\left(s_j\right)\right\vert ^{q}\right]<\infty.\label{eq2}
\end{equation}
\end{lemma}

\begin{proof} For all $i,j\in\left\{ 0,\cdots,n\right\}$, if $j<i$, $f_{\mathcal{P},i}\left(s_{j}\right)\equiv0$. So we only
need to consider the case when $j\geq i$. 
Since
\[
f_{\mathcal{P},i}\left(s_{j}\right)=\frac{C_{\mathcal{P},j}C_{\mathcal{P},j-1}\cdot\cdots\cdot C_{\mathcal{P},i+1}S_{\mathcal{P},i}}{\Delta_{i}}, 
\]
so 
\[
\left\vert f_{\mathcal{P},i}\left(s_{j}\right)\right\vert ^{q}\leq\left\vert C_{\mathcal{P},j}\right\vert ^{q}\left\vert C_{\mathcal{P},j-1}\right\vert ^{q}\cdot\cdots\cdot\left\vert C_{\mathcal{P},i+1}\right\vert ^{q}\left\vert \frac{S_{\mathcal{P},i}}{\Delta_{i}}\right\vert ^{q}.
\]
Applying Eq.(\ref{lem.6.1}) and (\ref{lem.6.3}) to find
\begin{align*}
\left\vert f_{\mathcal{P},i}\left(s_{j}\right)\right\vert ^{q} & \leq e^{\frac{1}{2}qN\sum_{k=i}^{j}\left\vert \Delta_{k}\beta\right\vert ^{2}}\left(e^{-\frac{N}{2}\left\vert \Delta_{i}\beta\right\vert ^{2}}+\frac{N{\left\vert \Delta_{i}\beta\right\vert}^2 }{6}\right)^{q}\\
 & \leq e^{\frac{1}{2}qN\sum_{k=i}^{j}\left\vert \Delta_{k}\beta\right\vert ^{2}}\left(1+\frac{N{\left\vert \Delta_{i}\beta\right\vert}^2}{6}\right)^{q}\label{equ.7.1}\\
 & \leq e^{\frac{1}{2}qN\sum_{k=i}^{j}\left\vert \Delta_{k}\beta\right\vert ^{2}}e^{\frac{Nq{\left\vert \Delta_{i}\beta\right\vert}^2 }{6}}\\& \leq e^{qN\sum_{k=1}^{n}\left\vert \Delta_{k}\beta\right\vert ^{2}}.
\end{align*}
Since $e^{qN\sum_{k=1}^{n}\left\vert \Delta_{k}\beta\right\vert ^{2}}$
is independent of $i$ and $j$, we have
\begin{equation}
\underset{i,j\in\left\{ 1,\cdots,n\right\} }{\sup}\left\vert f_{\mathcal{P},i}\left(s_j\right)\right\vert ^{q}\leq e^{qN\sum_{k=1}^{n}\left\vert \Delta_{k}\beta\right\vert ^{2}}.\label{equ.7.2}
\end{equation}
Then Eq.(\ref{eq2}) follows from Lemma \ref{cBM} in Appendix \ref{App.B}.
\end{proof}

\begin{notation} \label{not.7.9}Given $n\in\mathbb{N}$ and $s\in\left[0,1\right],$
let $\underline{s}=s_{k-1}$ when $s\in\left[s_{k-1},s_{k}\right)$, $\left|\mathcal{P}\right|=\frac{1}{n}$ is the mesh size of the partition $\mathcal{P}$ and also let
\[
A_{\mathcal{P},k}\left(s\right):=R_{u_{\mathcal{P}}\left(s\right)}\left(\beta_{\mathcal{P}}^{\prime}\left(s_{k-1}+\right),\cdot\right)\beta_{\mathcal{P}}^{\prime}\left(s_{k-1}+\right).
\]
\end{notation}
\begin{lemma} \label{lem.7.10}For each $q\geq1$, $\gamma\in\left(0,\frac{1}{2}\right)$ there exists a constant $C$ such that for all $n>5qN$,
\begin{equation}
\mathbb{E}\left[\underset{i\in\left\{ 0,\cdots,n\right\} ,s\in\left[0,1\right]}{\sup}\left\vert f_{\mathcal{P},i}\left(s\right)-f_{\mathcal{P},i}\left(\underline{s}\right)\right\vert ^{q}\right]\leq C\left\vert \mathcal{P}\right\vert ^{2q\gamma}.\label{5qN}
\end{equation}
\end{lemma}
\begin{proof} For $s\in\left[s_{k-1},s_{k}\right)$, Taylor's expansion gives
\begin{align*}
f_{\mathcal{P},i}&\left(s\right)- f_{\mathcal{P},i}\left(\underline{s}\right) =\int_{\underline{s}}^{s}A_{\mathcal{P},k}\left(r\right)f_{\mathcal{P},i}\left(r\right)\left(s-r\right)dr\\
 & =\int_{\underline{s}}^{s}A_{\mathcal{P},k}\left(r\right)\left(f_{\mathcal{P},i}\left(r\right)-f_{\mathcal{P},i}\left(\underline{r}\right)\right)\left(s-r\right)dr+\int_{\underline{s}}^{s}A_{\mathcal{P},k}\left(r\right)f_{\mathcal{P},i}\left(\underline{r}\right)\left(s-r\right)dr.
\end{align*}
Since $\left\vert A_{\mathcal{P},k}\left(s\right)\right\vert \leq N\left\vert \frac{\Delta_{k}\beta}{\Delta_{k}}\right\vert ^{2}$
, we have
\[
\left\vert f_{\mathcal{P},i}\left(s\right)-f_{\mathcal{P},i}\left(\underline{s}\right)\right\vert \leq\frac{N}{\Delta_{k}}\left\vert \Delta_{k}\beta\right\vert ^{2}\int_{\underline{s}}^{s}\left\vert f_{\mathcal{P},i}\left(r\right)-f_{\mathcal{P},i}\left(\underline{r}\right)\right\vert dr+\frac{1}{2}N\left\vert \Delta_{k}\beta\right\vert ^{2}\underset{j:j\geq i}{\sup}\left\vert f_{\mathcal{P},i}\left(s_j\right)\right\vert. 
\]
By Gronwall's inequality, we have: 
\begin{align*}
\left\vert f_{\mathcal{P},i}\left(s\right)-f_{\mathcal{P},i}\left(\underline{s}\right)\right\vert  & \leq\frac{1}{2}N\left\vert \Delta_{k}\beta\right\vert ^{2}\underset{j:j\geq i}{\sup}\left\vert f_{\mathcal{P},i}\left(s_j\right)\right\vert e^{N\left\vert \Delta_{k}\beta\right\vert ^{2}}
\end{align*}
Using estimate $\left(\ref{equ.7.2}\right)$ gives
\begin{align}
\left\vert f_{\mathcal{P},i}\left(s\right)-f_{\mathcal{P},i}\left(\underline{s}\right)\right\vert ^{q} & \leq\frac{N^{q}}{2^{q}}\left\vert \Delta_{k}\beta\right\vert ^{2q}e^{qN\left\vert \Delta_{k}\beta\right\vert ^{2}}e^{qN\sum_{j=1}^{n}\left\vert \Delta_{j}\beta\right\vert ^{2}} \nonumber\\
 & \leq C\left\vert \mathcal{P}\right\vert ^{2q\gamma}e^{2qN\sum_{k=1}^{n}\left\vert \Delta_{k}\beta\right\vert ^{2}}K_{\gamma}^{2q}.\label{equ.7.3}
\end{align}
Then Eq.(\ref{5qN}) follows from Lemma \ref{cBM}, \ref{lem.6.5} and Holder's inequality.
\end{proof}
\begin{theorem} \label{the.7.11}Let $\tilde{T}_{\left(\cdot\right)}$ be as in Definition \ref{def.T}, then for each $q\geq1$, $\gamma\in\left(0,\frac{1}{2}\right)$, there exists a constant $C$ such that for all $n>5q\gamma$,
\begin{equation}
\mathbb{E}\left[\underset{i\in\left\{ 0,\cdots,n\right\} }{\sup}\underset{s\in\left[s_{i},1\right]}{\sup}\left\vert f_{\mathcal{P},i}\left(s\right)-\tilde{T}_s\tilde{T}_{s_{i}}^{-1}\right\vert ^{q}\right]\leq C\left\vert \mathcal{P}\right\vert ^{q\gamma}.\nonumber
\end{equation}
\end{theorem} In order to prove Theorem \ref{the.7.11}, we need
the following result.

\begin{lemma}\label{lem.7.12}For each $q\geq1$, $\gamma\in\left(0,\frac{1}{2}\right)$, there exists a constant $C$ such that for all $n>5q\gamma$,
\begin{align}
& \mathbb{E}\left[\underset{i\in\left\{ 1,\cdots,n\right\} }{\sup}\underset{j\geq i}{\sup}\left\vert f_{\mathcal{P},i}\left(s_j\right)-\left(f_{\mathcal{P},i}\left(s_{i}\right)-\int_{s_{i}}^{s_j}Ric_{u_{\mathcal{P}}\left(\underline{r}\right)}f_{\mathcal{P},i}\left(\underline{r}\right)dr\right)\right\vert ^{q}\right]\leq C\left\vert \mathcal{P}\right\vert ^{q\gamma}.\label{eq:-46}
\end{align}

\end{lemma}

\begin{proof} For each $s_{j}\in\mathcal{P}$ with $j\geq i+1$ and
for $k=i,\cdots,j-1,$ we have
\begin{align}
f_{\mathcal{P},i}\left(s_{k+1}\right) & =f_{\mathcal{P},i}\left(s_{k}\right)+\int_{s_{k}}^{s_{k+1}}A_{\mathcal{P},k+1}\left(r\right)f_{\mathcal{P},i}\left(r\right)\left(s_{k+1}-r\right)dr\label{eq:-42}\\
& =f_{\mathcal{P},i}\left(s_{k}\right)+\frac{\Delta^2_{k+1}}{2}A_{\mathcal{P},k+1}\left(s_k\right)f_{\mathcal{P},i}\left(s_{k}\right)+e_{i,k}\nonumber 
\end{align}
where 
\begin{align*}
e_{i,k} &
=\int_{s_{k}}^{s_{k+1}}A_{\mathcal{P},k+1}\left(r\right)f_{\mathcal{P},i}\left(r\right)\left(s_{k+1}-r\right)dr-\int_{s_{k}}^{s_{k+1}}A_{\mathcal{P},k+1}\left(s_k\right)f_{\mathcal{P},i}\left(s_k\right)\left(s_{k+1}-r\right)dr.
\end{align*}
Since $\left\{ f_{\mathcal{P},i}\left(s_{j}\right)\right\} _{j}$ is adapted, by Ito's lemma
\begin{align*}
\frac{\Delta^2_{k+1}}{2}A_{\mathcal{P},k+1}\left(s_k\right)f_{\mathcal{P},i}\left(s_{k}\right)&=
\frac{1}{2}R_{u_{\mathcal{P}}\left(s_{k}\right)}\left(\Delta_{k+1}\beta,f_{\mathcal{P},i}\left(s_{k}\right)\right)\Delta_{k+1}\beta\\& =\frac{1}{2}\int_{s_{k}}^{s_{k+1}}R_{u_{\mathcal{P}}\left(s_{k}\right)}\left(\beta_{r}-\beta_{s_{k}},f_{\mathcal{P},i}\left(s_{k}\right)\right)d\beta_{r}\\
& +\frac{1}{2}\int_{s_{k}}^{s_{k+1}}R_{u_{\mathcal{P}}\left(s_{k}\right)}\left(d\beta_{r},f_{\mathcal{P},i}\left(s_{k}\right)\right)\left(\beta_{r}-\beta_{s_{k}}\right)\\
& -\frac{1}{2}Ric_{u_{\mathcal{P}}\left(s_{k}\right)}f_{\mathcal{P},i}\left(s_{k}\right)\Delta_{k}.
\end{align*}
Summing (\ref{eq:-42}) over $k$ from $i$ to $j-1$ , we have
\begin{equation}
f_{\mathcal{P},i}\left(s_{j}\right)=f_{\mathcal{P},i}\left(s_{i}\right)-\frac{1}{2}\int_{s_{i}}^{s_{j}}Ric_{u_{\mathcal{P}}\left(\underline{r}\right)}f_{\mathcal{P},i}\left(\underline{r}\right)dr+M_{\mathcal{P},s_{j}}+\sum_{k=i}^{j-1}e_{i,k}\label{eq:-1}
\end{equation}
where 
\[
M_{\mathcal{P},s}:=\frac{1}{2}\int_{s_{i}}^{s}R_{u_{\mathcal{P}}\left(\underline{r}\right)}\left(\beta_{r}-\beta_{\underline{r}},f_{\mathcal{P},i}\left(\underline{r}\right)\right)d\beta_{r}+\frac{1}{2}\int_{s_{i}}^{s}R_{u_{\mathcal{P}}\left(\underline{r}\right)}\left(d\beta_{r},f_{\mathcal{P},i}\left(\underline{r}\right)\right)\left(\beta_{r}-\beta_{\underline{r}}\right)
\]
is a $\mathbb{R}^{d}$-valued martingale starting from $s_{i}$. By Burkholder-Davis-Gundy inequality, for $q\geq1,$ 
\begin{equation}
\mathbb{E}\left[\underset{s\in\left[s_{i},1\right]}{\sup}\left\vert M_{\mathcal{P},s}\right\vert ^{q}\right]\leq C\mathbb{E}\left[\left\langle M_{\mathcal{P}}\right\rangle _{1}^{\frac{q}{2}}\right]\label{eq:-43}
\end{equation}
where $\left\langle M_{\mathcal{P}}\right\rangle$ is the
quadratic variation process of $M_{\mathcal{P}}$. An estimate
of $\left\langle M_{\mathcal{P}}\right\rangle $ gives
\[
\left\langle M_{\mathcal{P}}\right\rangle _{1}\leq dN^{2}\int_{s_{i}}^{1}\left\vert \beta_{r}-\beta_{\underline{r}}\right\vert ^{2}\left\vert f_{\mathcal{P},i}\left(\underline{r}\right)\right\vert ^{2}dr\leq dN^{2}\int_{0}^{1}\left\vert \beta_{r}-\beta_{\underline{r}}\right\vert ^{2}\left\vert f_{\mathcal{P},i}\left(\underline{r}\right)\right\vert ^{2}dr,
\]
and by Jensen's inequality, 
\[
\left\langle M_{\mathcal{P}}\right\rangle _{1}^{\frac{q}{2}}\leq d^{\frac{q}{2}}N^{q}\int_{0}^{1}\left\vert \beta_{r}-\beta_{\underline{r}}\right\vert ^{q}\left\vert f_{\mathcal{P},i}\left(\underline{r}\right)\right\vert ^{q}dr.
\]
Since $\left\{ f_{\mathcal{P},i}\left(\underline{r}\right)\right\} _{r\in\left[0,1\right]}$
is adapted to the filtration generated by $\beta$, using the independence
of $\left\vert \beta_{r}-\beta_{\underline{r}}\right\vert ^{q}$ and
$f_{\mathcal{P},i}\left(\underline{r}\right)$ we have: 
\begin{align*}
\mathbb{E}\left[\left\langle M_{\mathcal{P}}\right\rangle _{1}^{\frac{q}{2}}\right] & \leq d^{\frac{q}{2}}N^{q}\int_{0}^{1}\mathbb{E}\left[\left\vert \beta_{r}-\beta_{\underline{r}}\right\vert ^{q}\right]\mathbb{E}\left[\left\vert f_{\mathcal{P},i}\left(\underline{r}\right)\right\vert ^{q}\right]dr=C\underset{j\in \left\{0,\dots,n\right\}}{\sup}\mathbb{E}\left[\left\vert f_{\mathcal{P},i}\left(s_j\right)\right\vert ^{q}\right]\left\vert \mathcal{P}\right\vert ^{\frac{q}{2}}.
\end{align*}
By Lemma \ref{lem.7.8}, we know 
\begin{equation}
\mathbb{E}\left[\left\langle M_{\mathcal{P}}\right\rangle _{1}^{\frac{q}{2}}\right]\leq C\left\vert \mathcal{P}\right\vert ^{\frac{q}{2}}.\label{eq:-44}
\end{equation}
So to finish the proof of Lemma \ref{lem.7.12}, it suffices to show: 
\begin{equation}
\mathbb{E}\left[\underset{i\in\left\{ 0,\cdots,n\right\} ,j\in\left\{ i+1,\cdots,n\right\} }{\sup}\left\vert \sum_{k=i}^{j-1}e_{i,k}\right\vert ^{q}\right]\leq C\left\vert \mathcal{P}\right\vert ^{\gamma q}.\label{eq:-45}
\end{equation}
Since $\left\vert e_{i,k}\right\vert \leq I_{\mathcal{P}}\left(i,k\right)+II_{\mathcal{P}}\left(i,k\right)$, where 
\[
I_{\mathcal{P}}\left(i,k\right)=\left\vert \int_{s_{k}}^{s_{k+1}}A_{\mathcal{P},k}\left(r\right)\left(f_{\mathcal{P},i}\left(r\right)-f_{\mathcal{P},i}\left(s_{k}\right)\right)\left(s_{k+1}-r\right)dr\right\vert 
\]
and
\[
II_{\mathcal{P}}\left(i,k\right)=\left\vert \int_{s_{k}}^{s_{k+1}}\left(A_{\mathcal{P},k}\left(r\right)-A_{\mathcal{P},k}\left(s_k\right)\right)f_{\mathcal{P},i}\left(s_{k}\right)\left(s_{k+1}-r\right)dr\right\vert,
\]
using (\ref{equ.7.3}) we know
\begin{align*}
I_{\mathcal{P}}\left(i,k\right) & \leq\frac{N}{2}\underset{i\in\left\{ 1,\cdots,n\right\} ,r\in\left[0,1\right]}{\sup}\left\vert f_{\mathcal{P},i}\left(r\right)-f_{\mathcal{P},i}\left(\underline{r}\right)\right\vert \left\vert \Delta_{k+1}\beta\right\vert ^{2} \leq CK_{\gamma}^{4}\left\vert \mathcal{P}\right\vert ^{4\gamma}e^{2N\sum_{k=1}^{n}\left\vert \Delta_{k}\beta\right\vert ^{2}}.
\end{align*}
Since 
\[
\left\vert R_{u_{\mathcal{P}}\left(s_{k}\right)}-R_{u_{\mathcal{P}}\left(r\right)}\right\vert \leq C\int_{s_{k}}^{s_{k+1}}\left\vert \beta_{\mathcal{P}}^{\prime}\left(s\right)\right\vert ds=C\left\vert \Delta_{k+1}\beta\right\vert \leq CK_{\gamma}\left\vert \mathcal{P}\right\vert ^{\gamma},
\]using (\ref{equ.7.2}) we have
\begin{align*}
II_{\mathcal{P}}\left(i,k\right) & \leq C \underset{i,j\in\left\{ 1,\cdots,n\right\} }{\sup}\left\vert f_{\mathcal{P},i}\left(s_j\right)\right\vert \left\vert \Delta_{k+1}\beta\right\vert ^{2}\underset{r\in\left[s_{k},s_{k+1}\right]}{\sup}\left\vert R_{u_{\mathcal{P}}\left(s_{k}\right)}-R_{u_{\mathcal{P}}\left(r\right)}\right\vert \\
 & \leq CK_{\gamma}^{3}\left\vert \mathcal{P}\right\vert ^{3\gamma}e^{N\sum_{k=1}^{n}\left\vert \Delta_{k}\beta\right\vert ^{2}}.
\end{align*}
So 
\begin{align*}
\left\vert \sum_{k=i}^{j-1}e_{i,k}\right\vert  & \leq\frac{1}{\left\vert \mathcal{P}\right\vert }\left(I_{\mathcal{P}}\left(i,k\right)+II_{\mathcal{P}}\left(i,k\right)\right) \leq C\left(K_{\gamma}^{4}\left\vert \mathcal{P}\right\vert ^{4\gamma-1}+K_{\gamma}^{3}\left\vert \mathcal{P}\right\vert ^{3\gamma-1}\right)e^{2N\sum_{k=1}^{n}\left\vert \Delta_{k}\beta\right\vert ^{2}}.
\end{align*}
For any $\gamma^\prime\in \left(0,\frac{1}{2}\right)$, we can choose $\lambda\in \left(\frac{1}{3},\frac{1}{2}\right)$ such that $\gamma^\prime=3\gamma-1$ and thus using Lemma \ref{lem.6.5} we get
\[
\mathbb{E}\left[\underset{i\in\left\{ 0,\cdots,n\right\} ,j\in\left\{ i+1,\cdots,n\right\} }{\sup}\left\vert \sum_{k=i}^{j-1}e_{i,k}\right\vert ^{q}\right]\leq C\left\vert \mathcal{P}\right\vert ^{q\gamma^\prime}.
\]
Combining Eq. (\ref{eq:-1}), $\left(\ref{eq:-44}\right)$ and $\left(\ref{eq:-45}\right)$ we obtain $\left(\ref{eq:-46}\right)$. 
\end{proof}

\begin{proof}[Proof of Theorem \ref{the.7.11}]For $s\geq s_i$, define
\begin{equation}
\hat{f}_{\mathcal{P},i}\left(s\right):=f_{\mathcal{P},i}\left(s_{i}\right)-\frac{1}{2}\int_{s_{i}}^{s}Ric_{u_{\mathcal{P}}\left(r\right)}f_{\mathcal{P},i}\left(r\right)dr.
\end{equation}
Then
\begin{align*}
\left\vert \hat{f}_{\mathcal{P},i}\left(s_{j}\right)-f_{\mathcal{P},i}\left(s_{j}\right)\right\vert  & \leq\left\vert \frac{1}{2}\int_{s_{i}}^{s_j}\left(Ric_{u_{\mathcal{P}}\left(r\right)}-Ric_{u_{\mathcal{P}}\left(\underline{r}\right)}\right)f_{\mathcal{P},i}\left(\underline{r}\right)dr\right\vert \\
 & +\left\vert \frac{1}{2}\int_{s_{i}}^{s_j}Ric_{u_{\mathcal{P}}\left(r\right)}\left(f_{\mathcal{P},i}\left(r\right)-f_{\mathcal{P},i}\left(\underline{r}\right)\right)dr\right\vert+\left\vert M_{\mathcal{P},s_{j}}\right\vert+\left\vert\sum_{k=i}^{j-1}e_{i,k}\right\vert.
\end{align*}
Since 
\[
\left\vert Ric_{u_{\mathcal{P}}\left(r\right)}-Ric_{u_{\mathcal{P}}\left(\underline{r}\right)}\right\vert=\left\vert\nabla_{\left(r-\underline{r}\right)\beta^\prime_{\mathcal{P}}(\underline{r})}Ric\right\vert\leq C\underset{i}{\sup}\left\vert\Delta_i\beta\right\vert\leq CK_{\gamma}\left\vert \mathcal{P}\right\vert ^{\gamma},
\]
using Lemma \ref{lem.7.8}, \ref{lem.6.5} and Holder's inequality we know
\begin{equation}
\mathbb{E}\left[\left\vert \int_{s_{i}}^{s_j}\left(Ric_{u_{\mathcal{P}}\left(r\right)}-Ric_{u_{\mathcal{P}}\left(\underline{r}\right)}\right)f_{\mathcal{P},i}\left(\underline{r}\right)dr\right\vert ^{q}\right]\leq C\left\vert \mathcal{P}\right\vert ^{\gamma q}.\label{equ.7.4} \end{equation}
Then we consider 
\[
\left\vert \int_{s_{i}}^{s_j}Ric_{u_{\mathcal{P}}\left(r\right)}\left(f_{\mathcal{P},i}\left(r\right)-f_{\mathcal{P},i}\left(\underline{r}\right)\right)dr\right\vert.
\]
By Lemma \ref{lem.7.10}, one can easily see
\begin{equation}
\mathbb{E}\left[\underset{i\in\left\{ 0,\cdots,n\right\} }{\sup}\left\vert \int_{s_{i}}^{s_j}Ric_{u_{\mathcal{P}}\left(r\right)}\left(f_{\mathcal{P},i}\left(r\right)-f_{\mathcal{P},i}\left(\underline{r}\right)\right)dr\right\vert ^{q}\right]\leq C\left\vert \mathcal{P}\right\vert ^{2q\gamma}.\label{equ.7.5}
\end{equation}
Combining Eq.$\left(\ref{equ.7.4}\right)$ and $\left(\ref{equ.7.5} \right)$ and Lemma \ref{lem.7.12} we get
\begin{equation*}
\mathbb{E}\left[\underset{i\in\left\{ 0,\cdots,n\right\} ,j\geq i}{\sup}\left\vert \hat{f}_{\mathcal{P},i}\left(s_{j}\right)-f_{\mathcal{P},i}\left(s_{j}\right)\right\vert ^{q}\right]\leq C\left\vert \mathcal{P}\right\vert ^{q\gamma}.
\end{equation*}
Then using Lemma \ref{lem7.10} and notice that 
\[\left\vert\hat{f}_{\mathcal{P},i}\left(s\right)-\hat{f}_{\mathcal{P},i}\left(\underline{s}\right)\right\vert\leq C\left(s-\underline{s}\right)\underset{0\leq s\leq 1}{\sup}\left\vert f_{\mathcal{P},i}(s)\right\vert,\]
we have
\begin{equation}
\mathbb{E}\left[\underset{i\in\left\{ 0,\cdots,n\right\} ,s\geq s_i}{\sup}\left\vert \hat{f}_{\mathcal{P},i}\left(s\right)-f_{\mathcal{P},i}\left(s\right)\right\vert ^{q}\right]\leq C\left\vert \mathcal{P}\right\vert ^{ q\gamma}.\label{equ.7.6}
\end{equation}
Then for $s\geq s_i$, define $\tilde{f}_{\mathcal{P},i}\left(s\right)$ to be the solution
to the following ODE
\[
\begin{cases}
\frac{d}{ds}\tilde{f}_{\mathcal{P},i}\left(s\right)+\frac{1}{2}Ric_{u_{\mathcal{P}}\left(s\right)}\tilde{f}_{\mathcal{P},i}\left(s\right)=0\\
\tilde{f}_{\mathcal{P},i}\left(s_{i}\right)=I.
\end{cases}
\]
Therefore
\[
\tilde{f}_{\mathcal{P},i}\left(s\right)=I-\frac{1}{2}\int_{s_{i}}^{s}Ric_{u_{\mathcal{P}}\left(r\right)}\tilde{f}_{\mathcal{P},i}\left(r\right)dr
\]
and
\[
\left\vert \tilde{f}_{\mathcal{P},i}\left(s\right)-\hat{f}_{\mathcal{P},i}\left(s\right)\right\vert \leq\left\vert f_{\mathcal{P},i}\left(s_{i}\right)-I\right\vert +C\int_{s_{i}}^{s}\left\vert \tilde{f}_{\mathcal{P},i}\left(r\right)-\hat{f}_{\mathcal{P},i}\left(r\right)\right\vert dr+C\underset{s\geq s_i}{\sup}\left\vert f_{\mathcal{P},i}\left(s\right)-\hat{f}_{\mathcal{P},i}\left(s\right)\right\vert.
\]
By Gronwall's inequality we have
\[
\left\vert \tilde{f}_{\mathcal{P},i}\left(s\right)-\hat{f}_{\mathcal{P},i}\left(s\right)\right\vert \leq\left(\left\vert f_{\mathcal{P},i}\left(s_{i}\right)-I\right\vert +C\underset{s\geq s_i}{\sup}\left\vert f_{\mathcal{P},i}\left(s\right)-\hat{f}_{\mathcal{P},i}\left(s\right)\right\vert\right)e^{\frac{1}{2}N}.
\]
Thus by Lemma \ref{lem.6.3} and Eq.(\ref{equ.7.6}) it follows that
\begin{equation}
\mathbb{E}\left[\underset{i\in\left\{ 0,\cdots,n\right\} ,s\geq s_{i}}{\sup}\left\vert \tilde{f}_{\mathcal{P},i}\left(s\right)-\hat{f}_{\mathcal{P},i}\left(s\right)\right\vert ^{q}\right]\leq C\left\vert \mathcal{P}\right\vert ^{q\gamma}.\label{equ.7.7}
\end{equation}
Lastly, we look at $\tilde{f}_{\mathcal{P},i}\left(s\right)-\tilde{T}_s\tilde{T}_{s_{i}}^{-1}$ where $s\geq s_i$. Note that $\tilde{T}_s\tilde{T}_{s_{i}}^{-1}$ satisfies the
following ODE,
\[
\begin{cases}
\left(\tilde{T}_s\tilde{T}_{s_{i}}^{-1}\right)^\prime+\frac{1}{2}Ric_{\tilde{u}_s}\left(\tilde{T}_s\tilde{T}_{s_{i}}^{-1}\right)=0\\
\left(\tilde{T}_{s_{i}}\tilde{T}_{s_{i}}^{-1}\right)=I.
\end{cases}
\]
So 
\[
\tilde{f}_{\mathcal{P},i}\left(s\right)-\tilde{T}_s\tilde{T}_{s_{i}}^{-1}=\frac{1}{2}\int_{s_{i}}^{s}\left(Ric_{u_{\mathcal{P}}\left(r\right)}-Ric_{\tilde{u}_r}\right)\left(\tilde{f}_{\mathcal{P},i}\left(r\right)-\tilde{T}_r\tilde{T}_{s_{i}}^{-1}\right)dr.
\]
By Gronwall's inequality again we have
\[
\left\vert \tilde{f}_{\mathcal{P},i}\left(s\right)-\tilde{T}_s\tilde{T}_{s_{i}}^{-1}\right\vert \leq CK_{\gamma}\left\vert \mathcal{P}\right\vert ^{\gamma}e^{\frac{1}{2}N},
\]
so 
\begin{equation}
\mathbb{E}\left[\underset{i\in\left\{ 0,\cdots,n\right\} ,s\geq s_{i}}{\sup}\left\vert \tilde{f}_{\mathcal{P},i}\left(s\right)-\tilde{T}_s\tilde{T}_{s_{i}}^{-1}\right\vert ^{q}\right]\leq C\left\vert \mathcal{P}\right\vert ^{\gamma q}.\label{equ.7.8}
\end{equation}
The proof is completed by combining Lemma \ref{lem.7.12} and $\left(\ref{equ.7.5}\right)$, $\left(\ref{equ.7.6}\right)$, $\left(\ref{equ.7.7}\right)$ and $\left(\ref{equ.7.8}\right)$. \end{proof}

\subsubsection{Convergence of $\mathbf{K}_{\mathcal{P}}\left(s\right)$ to $\mathbf{\tilde{K}}_s$\label{sub.6.1.3}}

Recall from Definition \ref{def Kp} that $\mathbf{K}_{\mathcal{P}}\left(s\right)$
satisfies the piecewise Jacobi equation: 
\begin{equation}
\begin{cases}
\mathbf{K}_{\mathcal{P}}^{\prime\prime}\left(s\right)=A_{\mathcal{P},i}(s)\mathbf{K}_{\mathcal{P}}\left(s\right)\text{ for }s\in \left[s_{i-1},s_i\right)\\
\mathbf{K}_{\mathcal{P}}^{\prime}\left(s_{i-1}+\right)=f_{\mathcal{P},i}^{\ast}\left(1\right)\text{ and }\mathbf{K}_{\mathcal{P}}\left(0\right)=0, \text{for }i=1,...,n
\end{cases}\label{equ.7.9}
\end{equation}
where $f_{\mathcal{P},i}\left(1\right)$ is given in Definition \ref{def.2.1-1}.

Before we state the main theorem in this section, we need some
supplementary lemmas.

\begin{lemma} \label{lem.7.14}Recall that $n:=\frac{1}{\left|\mathcal{P}\right|}$ and $N$ is the curvature bound. For each $q\geq1$,
\begin{equation}
\sup_{n>2qN}\mathbb{E}\left[\underset{j\in\left\{0,\dots,n\right\}}{\sup}\left\vert \mathbf{K}_{\mathcal{P}}\left(s_j\right)\right\vert ^{q}\right]<\infty.\label{K_p}
\end{equation}
\end{lemma}

\begin{proof} For all $i\in\left\{ 1,\cdots,n\right\} ,$ recall
from $\left(\ref{eq:-29}\right)$ that
\[
\mathbf{K}_{\mathcal{P}}\left(s_{i}\right)=\frac{1}{n}\sum_{j=0}^{i-1}f_{\mathcal{P},j+1}\left(s\right)f_{\mathcal{P},j+1}^{*}\left(1\right).
\]
So for all $q\geq1,$ we have
\[
\left\vert \mathbf{K}_{\mathcal{P}}\left(s_{i}\right)\right\vert ^{q}\leq i^{q-1}\frac{1}{n^{q}}\sum_{j=0}^{i-1}\left\vert f_{\mathcal{P},j+1}\left(s_{i}\right)\right\vert ^{q}\left\vert f_{\mathcal{P},j+1}\left(1\right)\right\vert ^{q}.
\]
Using $\left(\ref{equ.7.2}\right)$ we have 
\begin{equation}
\left\vert \mathbf{K}_{\mathcal{P}}\left(s_{i}\right)\right\vert ^{q}\leq e^{2qN\sum_{k=1}^{n}\left\vert \Delta_{k}\beta\right\vert ^{2}}.\label{equ.7.10}
\end{equation}
Then taking expectations as in Lemma \ref{lem.7.8} gives (\ref{K_p}). \end{proof}

\begin{lemma} \label{lem.7.15}For each $q\geq1$ and $\gamma\in \left(0,\frac{1}{2}\right)$, there exists a constant $C>0$ such that for all $n>5qN$,  
\[
\mathbb{E}\left[\underset{i\in\left\{ 1,\cdots,n\right\} ,r\in\left[0,1\right]}{\sup}\left\vert \mathbf{K}_{\mathcal{P}}\left(r\right)-\mathbf{K}_{\mathcal{P}}\left(\underline{r}\right)\right\vert ^{q}\right]\leq C\left\vert \mathcal{P}\right\vert ^{2q\gamma}
\]
\end{lemma}
\begin{proof}
For $s\in\left[ s_{i-1},s_i\right]$, 
\begin{align*}
\mathbf{K}&_{\mathcal{P}}\left(s\right)=\mathbf{K}_{\mathcal{P}}\left(s_{i-1}\right)+f_{\mathcal{P},i}^{*}\left(1\right)\left(s-s_{i-1}\right)+\int_{s_{i-1}}^{s}A_{\mathcal{P},i}(r)\mathbf{K}_{\mathcal{P}}\left(r\right)\left(s-r\right)dr.
\end{align*}
Therefore
\begin{align}
 &\left\vert \mathbf{K}_{\mathcal{P}}\left(s\right)-\mathbf{K}_{\mathcal{P}}\left(s_{i-1}\right)\right\vert\nonumber \\&\leq\left\vert f_{\mathcal{P},i}\left(1\right)\right\vert \left(s-s_{i-1}\right)+\left\vert \int_{s_{i-1}}^{s}A_{\mathcal{P},i}\left(r\right)\left(\mathbf{K}_{\mathcal{P}}\left(r\right)-\mathbf{K}_{\mathcal{P}}\left(s_{i-1}\right)+\mathbf{K}_{\mathcal{P}}\left(s_{i-1}\right)\right)\left(s-r\right)dr\right\vert \nonumber \\
 & \leq\left\vert f_{\mathcal{P},i}\left(1\right)\right\vert \left(s-s_{i-1}\right)+N\frac{\left\vert \Delta_{i}\beta\right\vert ^{2}}{\Delta_{i}^{2}}\int_{s_{i-1}}^{s}\left\vert \mathbf{K}_{\mathcal{P}}\left(r\right)-\mathbf{K}_{\mathcal{P}}\left(s_{i-1}\right)\right\vert \left(s-r\right)dr+\frac{1}{2}N\left\vert \Delta_{i}\beta\right\vert ^{2}\left\vert \mathbf{K}_{\mathcal{P}}\left(s_{i-1}\right)\right\vert. \label{equ.7.11}
\end{align}
We use the shorthand 
\begin{align*}
f\left(s\right)&:=\left\vert f_{\mathcal{P},i}\left(1\right)\right\vert \left(s-s_{i-1}\right)+N\frac{\left\vert \Delta_{i}\beta\right\vert ^{2}}{\Delta_{i}^{2}}\int_{s_{i-1}}^{s}\left\vert \mathbf{K}_{\mathcal{P}}\left(r\right)-\mathbf{K}_{\mathcal{P}}\left(s_{i-1}\right)\right\vert \left(s-r\right)dr\\&+\frac{1}{2}N\left\vert \Delta_{i}\beta\right\vert ^{2}\left\vert \mathbf{K}_{\mathcal{P}}\left(s_{i-1}\right)\right\vert.
\end{align*}
Then it is easily seen that
\[
f^{\prime}\left(s\right)=\left\vert f_{\mathcal{P},i}\left(1\right)\right\vert +N\frac{\left\vert \Delta_{i}\beta\right\vert ^{2}}{\Delta_{i}^{2}}\int_{s_{i-1}}^{s}\left\vert \mathbf{K}_{\mathcal{P}}\left(r\right)-\mathbf{K}_{\mathcal{P}}\left(s_{i-1}\right)\right\vert dr,
\]
\[
f^{\prime\prime}\left(s\right)=N\frac{\left\vert \Delta_{i}\beta\right\vert ^{2}}{\Delta_{i}^{2}}\left\vert \mathbf{K}_{\mathcal{P}}\left(s\right)-\mathbf{K}_{\mathcal{P}}\left(s_{i-1}\right)\right\vert \leq N\frac{\left\vert \Delta_{i}\beta\right\vert ^{2}}{\Delta_{i}^{2}}f\left(s\right), 
\]
and $f\left(s\right)$ satisfies the following ODE
\begin{equation}
\begin{cases}
f^{\prime\prime}\left(s\right)=N\frac{\left\vert \Delta_{i}\beta\right\vert ^{2}}{\Delta_{i}^{2}}f\left(s\right)+\delta\left(s\right)\\
f^{\prime}\left(s_{i-1}\right)=\left\vert f_{\mathcal{P},i}\left(1\right)\right\vert \\
f\left(s_{i-1}\right)=\frac{1}{2}N\left\vert \Delta_{i}\beta\right\vert ^{2}\left\vert \mathbf{K}_{\mathcal{P}}\left(s_{i-1}\right)\right\vert 
\end{cases}\label{equ.7.12}
\end{equation}
where 
\[
\delta\left(s\right)=f^{\prime\prime}\left(s\right)-N\frac{\left\vert \Delta_{i}\beta\right\vert ^{2}}{\Delta_{i}^{2}}f\left(s\right)\leq0.
\]
This ODE can be solved exactly to obtain
\begin{align*}
f\left(s\right) & =\mathcal{C}_{s_{i-1}}\left(s\right)\frac{1}{2}N\left\vert \Delta_{i}\beta\right\vert ^{2}\left\vert \mathbf{K}_{\mathcal{P}}\left(s_{i-1}\right)\right\vert +\mathcal{S}_{s_{i-1}}\left(s\right)\left\vert f_{\mathcal{P},i}\left(1\right)\right\vert +\int_{s_{i-1}}^{s}\mathcal{C}_r\left(s\right)\delta\left(r\right)dr
\end{align*}
where 
\[
\mathcal{C}_{r}\left(s\right):=\cosh\left(\sqrt{N}\left\vert \beta_{\mathcal{P}}^{\prime}\left(s_{i-1}+\right)\right\vert \left(s-r\right)\right)
\]
and 
\[
\mathcal{S}_{r}\left(s\right):=\frac{\sinh\left(\sqrt{N}\left\vert \beta_{\mathcal{P}}^{\prime}\left(s_{i-1}+\right)\right\vert \left(s-r\right)\right)}{\sqrt{N}\left\vert \beta_{\mathcal{P}}^{\prime}\left(s_{i-1}+\right)\right\vert }.
\]
Since $\delta\left(r\right)\leq 0$ and $\mathcal{C}_r\left(s\right)\geq 0$, we have
\[f\left(s\right)\leq\mathcal{C}_{s_{i-1}}\left(s\right)\frac{1}{2}N\left\vert \Delta_{i}\beta\right\vert ^{2}\left\vert \mathbf{K}_{\mathcal{P}}\left(s_{i-1}\right)\right\vert +\mathcal{S}_{s_{i-1}}\left(s\right)\left\vert f_{\mathcal{P},i}\left(1\right)\right\vert.\]
Then using the following estimate 
\[ \frac{\mathcal{S}_{s_{i-1}}\left(s\right)}{\Delta_i}\leq \mathcal{C}_{s_{i-1}}\left(s\right)\frac{s-s_{i-1}}{\Delta_i}\leq e^{N\left|\Delta_i\beta \right|^2},
\]
we obtain 
\begin{align}
f\left(s\right) & \leq e^{N\left\vert \Delta_{i}\beta\right\vert ^{2}}\left(\frac{1}{2}N\left\vert \Delta_{i}\beta\right\vert ^{2}\left\vert \mathbf{K}_{\mathcal{P}}\left(s_{i-1}\right)\right\vert +\left\vert \mathcal{P}\right\vert \left\vert f_{\mathcal{P},i}\left(1\right)\right\vert \right)\label{equ.7.14}\\
 & \leq e^{NK_{\gamma}^{2}\left\vert \mathcal{P}\right\vert ^{2\gamma}}\left(\frac{1}{2}NK_{\gamma}^{2}\left\vert \mathcal{P}\right\vert ^{2\gamma}\underset{i\in\left\{ 1,\cdots,n\right\} }{\sup}\left\vert \mathbf{K}_{\mathcal{P}}\left(s_{i-1}\right)\right\vert +\left\vert \mathcal{P}\right\vert \underset{i\in\left\{ 1,\cdots,n\right\} }{\sup}\left\vert f_{\mathcal{P},i}\left(1\right)\right\vert \right).\nonumber 
\end{align}
Note that $f\geq 0$, using (\ref{equ.7.2}) and (\ref{equ.7.10}) we have for $q\geq 1$,
\[
f^q\left(s\right)\leq U_{q}\left\vert P\right\vert ^{2q\gamma},
\]
where 
\[
U_{q}=e^{qNK_{\gamma}^{2}\left\vert \mathcal{P}\right\vert ^{2\gamma}}\left(\frac{1}{2}NK_{\gamma}^{2}+\left|\mathcal{P}\right|^{1-2\gamma}\right)^qe^{qN\sum_{k=1}^{n}\left|\Delta_{k}\beta\right|^{2}}
\]
is a random variable with finite first moment which can be bounded uniformly for $n>5qN$. Therefore,
\begin{equation}
\mathbb{E}\left[\underset{i\in\left\{ 1,\cdots,n\right\} ,r\in\left[0,1\right]}{\sup}\left\vert \mathbf{K}_{\mathcal{P}}\left(r\right)-\mathbf{K}_{\mathcal{P}}\left(\underline{r}\right)\right\vert ^{q}\right]\leq C\left\vert \mathcal{P}\right\vert ^{2q\gamma}\label{equ.6.40}
\end{equation}
\end{proof}

\begin{proposition} \label{lem.7.20}Let $\mathbf{K}_{\mathcal{P}}$ and $\tilde{\mathbf{K}}$ be defined as in Definition \ref{def Kp} and \ref{def.4.4-1}. Then for each $q\geq1$ and $\gamma\in \left(0,\frac{1}{2}\right)$, there exists a constant $C>0$ such that for all $n>5qN$,  
\begin{equation}
\mathbb{E}\left[\underset{i\in\left\{0,\dots,n\right\}}{\sup}\left\vert \mathbf{K}_{\mathcal{P}}\left(s_i\right)-\tilde{\mathbf{K}}_{s_i}\right\vert ^{q}\right]\leq C\left\vert \mathcal{P}\right\vert ^{q}.\label{equ.7.20}
\end{equation}
\end{proposition}
\begin{proof}For all $i\in\left\{ 1,\cdots,n\right\}$, $\mathbf{K}_{\mathcal{P}}\left(s_{i}\right)$ and $\tilde{\mathbf{K}}_{s_{i}}$ can be rewritten as
\begin{equation}
\mathbf{K}_{\mathcal{P}}\left(s_{i}\right)=f_{\mathcal{P},i-1}\left(s_{i}\right)f_{\mathcal{P},i-1}\left(1\right)^{-1}\left(\sum_{j=0}^{i-1}f_{\mathcal{P},j+1}\left(1\right)f_{\mathcal{P},j+1}^{\ast}\left(1\right)\right)\Delta_{j+1} \label{equ.7.21}
\end{equation}
and 
\[
\tilde{\mathbf{K}}_{s_{i}}=\tilde{T}_{s_{i}}\tilde{T}_1^{-1}\int_{0}^{s_{i}}\left(\tilde{T}_1\tilde{T}_r^{-1}\right)\left(\tilde{T}_1\tilde{T}_r^{-1}\right)^{\ast}dr.
\]
First define 
\[
\bar{\mathbf{K}}_\mathcal{P}\left(s_{i}\right):=\tilde{T}_{s_{i}}\tilde{T}_1^{-1}\int_{0}^{s_{i}}\left(\tilde{T}_1\tilde{T}_{\overline{r}}^{-1}\right)\left(\tilde{T}_1\tilde{T}_{\overline{r}}^{-1}\right)^{\ast}dr,
\]
where $\overline{r}=s_i$ if $s\in \left[s_{i-1},s_i\right)$.
We will show, for each $q\geq1$, 
\begin{equation}
\underset{i\in\left\{0,\dots,n\right\}}{\sup}\left\vert \tilde{\mathbf{K}}_{s_i}-\bar{\mathbf{K}}_\mathcal{P}\left(s_i\right)\right\vert ^{q}\leq C\left\vert \mathcal{P}\right\vert ^{q}.\label{equ.7.22}
\end{equation}
Recall from (\ref{equ.4.4}) that $\tilde{T}_1\tilde{T}_r^{-1}$ satisfies
the following ODE,
\[
\frac{d}{dr}\left(\tilde{T}_1\tilde{T}_r^{-1}\right)=\frac{1}{2}\left(\tilde{T}_1\tilde{T}_r^{-1}\right)Ric_{\tilde{u}_r}.
\]
So by Lemma \ref{lem.4.3}, 
\[
\left\vert \frac{d}{dr}\left(\tilde{T}_1\tilde{T}_r^{-1}\right)\right\vert \leq N\left\vert \tilde{T}_1\tilde{T}_r^{-1}\right\vert \leq N.
\]
Therefore
\begin{align*}
 \left\vert \left(\tilde{T}_1\tilde{T}_r^{-1}\right)\left(\tilde{T}_1\tilde{T}_r^{-1}\right)^{\ast}-\left(\tilde{T}_1\tilde{T}_{\overline{r}}^{-1}\right)\left(\tilde{T}_1\tilde{T}_{\overline{r}}^{-1}\right)^{\ast}\right\vert & \leq\int_{r}^{\overline{r}}\left\vert \frac{d}{ds}\left[ \left(\tilde{T}_1\tilde{T}_s^{-1}\right)\left(\tilde{T}_1\tilde{T}_s^{-1}\right)^{\ast}\right] \right\vert ds\\
 & \leq2\int_{r}^{\overline{r}}\left\vert \frac{d}{ds}\left(\tilde{T}_1\tilde{T}_s^{-1}\right)\right\vert \left\vert \left(\tilde{T}_1\tilde{T}_s^{-1}\right)^{\ast}\right\vert ds\\
 & \leq C\left(\overline{r}-r\right)\\&\leq C\left\vert \mathcal{P}\right\vert ,
\end{align*}
and
\begin{align*}
 \left\vert \tilde{\mathbf{K}}_{s_{i}}-\bar{\mathbf{K}}_\mathcal{P}\left(s_{i}\right)\right\vert &\leq\left\vert \tilde{T}_{s_{i}}\tilde{T}_1^{-1}\right\vert \int_{0}^{s_{i}}\left\vert \left(\tilde{T}_1\tilde{T}_r^{-1}\right)\left(\tilde{T}_1\tilde{T}_r^{-1}\right)^{\ast}-\left(\tilde{T}_1\tilde{T}_{\overline{r}}^{-1}\right)\left(\tilde{T}_1\tilde{T}_{\overline{r}}^{-1}\right)^{\ast}\right\vert dr\leq C\left\vert \mathcal{P}\right\vert.
\end{align*}
Since the right--hand side is independent of $i,$ we proved $\left( \ref{equ.7.22}\right)$.
Secondly, define
\[
\hat{\mathbf{K}}_\mathcal{P}\left(s_{i}\right):=\tilde{T}_{s_{i}}\tilde{T}_1^{-1}\left(\sum_{j=0}^{i-1}f_{\mathcal{P},j+1}\left(1\right)f_{\mathcal{P},j+1}^{\ast}\left(1\right)\right)\Delta_{j+1}. 
\]
We will show, for each $q\geq1,\gamma\in\left(0,\frac{1}{2}\right),$
there exists a constant $C>0$ such that for all $n>5qN$, we have
\begin{equation}
\mathbb{E}\left[\underset{s\in\mathcal{P}}{\sup}\left\vert \hat{\mathbf{K}}_\mathcal{P}\left(s\right)-\bar{\mathbf{K}}_\mathcal{P}\left(s\right)\right\vert ^{q}\right]\leq C\left\vert \mathcal{P}\right\vert ^{q\gamma}.\label{equ.7.23}
\end{equation}
For each $j\in\left\{ 1,\cdots,n\right\} ,$ 
\begin{align*}
 & \left\vert f_{\mathcal{P},j+1}\left(1\right)f_{\mathcal{P},j+1}^{\ast}\left(1\right)-\left(\tilde{T}_1\tilde{T}_{s_{j+1}}^{-1}\right)\left(\tilde{T}_1\tilde{T}_{s_{j+1}}^{-1}\right)^{\ast}\right\vert \\
 & \leq\left\vert f_{\mathcal{P},j+1}\left(1\right)f_{\mathcal{P},j+1}^{\ast}\left(1\right)-f_{\mathcal{P},j+1}\left(1\right)\left(\tilde{T}_1\tilde{T}_{s_{j+1}}^{-1}\right)^{\ast}\right\vert \\
 & +\left\vert f_{\mathcal{P},j+1}\left(1\right)\left(\tilde{T}_1\tilde{T}_{s_{j+1}}^{-1}\right)^{\ast}-\left(\tilde{T}_1\tilde{T}_{s_{j+1}}^{-1}\right)\left(\tilde{T}_1\tilde{T}_{s_{j+1}}^{-1}\right)^{\ast}\right\vert \\
 & \leq\left(\left\vert f_{\mathcal{P},j+1}\left(1\right)\right\vert +\left\vert \tilde{T}_1\tilde{T}_{s_{j+1}}^{-1}\right\vert \right)\left\vert f_{\mathcal{P},j+1}\left(1\right)-\tilde{T}_1\tilde{T}_{s_{j+1}}^{-1}\right\vert. 
\end{align*}
Since $\left\vert f_{\mathcal{P},j+1}\left(1\right)\right\vert \leq e^{N\sum_{k=1}^{n}\left\vert \Delta_{k}\beta\right\vert ^{2}}$ by (\ref{equ.7.2}), ans also $\left\vert \tilde{T}_1\tilde{T}_{s_{j+1}}^{-1}\right\vert \leq1$, we have
\begin{align*}
 & \left\vert f_{\mathcal{P},j+1}\left(1\right)f_{\mathcal{P},j+1}^{\ast}\left(1\right)-\left(\tilde{T}_1\tilde{T}_{s_{j+1}}^{-1}\right)\left(\tilde{T}_1\tilde{T}_{s_{j+1}}^{-1}\right)^{\ast}\right\vert \\
 & \leq\left(e^{N\sum_{k=1}^{n}\left\vert \Delta_{k}\beta\right\vert ^{2}}+1\right)\underset{j\in\left\{ 1,\cdots,n\right\} }{\sup}\left\vert f_{\mathcal{P},j+1}\left(1\right)-\tilde{T}_1\tilde{T}_{s_{j+1}}^{-1}\right\vert. 
\end{align*}
Thus for all $i\in\left\{ 1,\cdots,n\right\}$, 
\begin{align*}
\left\vert \hat{\mathbf{K}}_\mathcal{P}\left(s_{i}\right)-\tilde{\mathbf{K}}_\mathcal{P}\left(s_{i}\right)\right\vert^q & \leq\left\vert \mathcal{P}\right\vert^q i^{-q}\sum_{j=0}^{i-1}\left\vert f_{\mathcal{P},j+1}\left(1\right)f_{\mathcal{P},j+1}^{\ast}\left(1\right)-\left(\tilde{T}_1\tilde{T}_{s_{j+1}}^{-1}\right)\left(\tilde{T}_1\tilde{T}_{s_{j+1}}^{-1}\right)^{\ast}\right\vert^q \\
 & \leq\left(e^{N\sum_{k=1}^{n}\left\vert \Delta_{k}\beta\right\vert ^{2}}+1\right)^{q}\underset{j\in\left\{ 1,\cdots,n\right\} }{\sup}\left\vert f_{\mathcal{P},j+1}\left(1\right)-\tilde{T}_1\tilde{T}_{s_{j+1}}^{-1}\right\vert ^{q}.
\end{align*}
Since $\left(e^{N\sum_{k=1}^{n}\left\vert \Delta_{k}\beta\right\vert ^{2}}+1\right)^{q}\leq e^{qN\sum_{k=1}^{n}\left\vert \Delta_{k}\beta\right\vert ^{2}}$, using Holder's inequality and Theorem \ref{the.7.11} we get 
\[
\mathbb{E}\left[\underset{s\in\mathcal{P}}{\sup}\left\vert \hat{\mathbf{K}}_\mathcal{P}\left(s\right)-\tilde{\mathbf{K}}_\mathcal{P}\left(s\right)\right\vert ^{q}\right]\leq C\left\vert \mathcal{P}\right\vert ^{q\gamma}.
\]
Lastly, we estimate $\hat{\mathbf{K}}_{\mathcal{P}}\left(s_{i}\right)-\mathbf{K}_{\mathcal{P}}\left(s_{i}\right)$. Using $\left(\ref{equ.7.21}\right)$ we have 
\begin{align*}
  \left\vert \hat{\mathbf{K}}_\mathcal{P}\left(s_i\right)-\mathbf{K}_{\mathcal{P}}\left(s_{i}\right)\right\vert
 & \leq\left\vert f_{\mathcal{P},i-1}\left(s_{i}\right)f^{-1}_{\mathcal{P},i-1}\left(1\right)-\tilde{T}_{s_{i}}\tilde{T}_1^{-1}\right\vert \left\vert \left(\sum_{j=0}^{i-1}f_{\mathcal{P},j+1}\left(1\right)f_{\mathcal{P},j+1}^{\ast}\left(1\right)\right)\Delta_{j+1} \right\vert \\
 & \leq\left\vert f_{\mathcal{P},i-1}\left(s_{i}\right)f^{-1}_{\mathcal{P},i-1}\left(1\right)-\tilde{T}_{s_{i}}\tilde{T}_1^{-1}\right\vert \underset{j\in\left\{ 1,\cdots,n\right\} }{\sup}\left\vert f_{\mathcal{P},j+1}\left(1\right)\right\vert ^{2}.
\end{align*}
Since
\begin{align*}
 & \left\vert f_{\mathcal{P},i-1}\left(s_{i}\right)f^{-1}_{\mathcal{P},i-1}\left(1\right)-\tilde{T}_{s_{i}}\tilde{T}_1^{-1}\right\vert \\
 & \quad=\left\vert f_{\mathcal{P},i-1}\left(s_{i}\right)-\tilde{T}_{s_{i}}\tilde{T}_{s_{i-1}}^{-1}\right\vert \left\vert f^{-1}_{\mathcal{P},i-1}\left(1\right)\right\vert+\left\vert \tilde{T}_{s_{i}}\tilde{T}_{s_{i-1}}^{-1}\right\vert \left\vert \left(\tilde{T}_1\tilde{T}_{s_{i-1}}^{-1}\right)^{-1}-f^{-1}_{\mathcal{P},i-1}\left(1\right)\right\vert,
\end{align*}
and from Proposition \ref{prop A-1}, we know $\left\vert f_{\mathcal{P},i-1}\left(1\right)^{-1}\right\vert \leq1$,
and 
\begin{align*}
  \left\vert \left(\tilde{T}_1\tilde{T}_{s_{i-1}}^{-1}\right)^{-1}-f_{\mathcal{P},i-1}\left(1\right)^{-1}\right\vert& \leq\left\vert \left(\tilde{T}_1\tilde{T}_{s_{i-1}}^{-1}\right)^{-1}\right\vert \left\vert \tilde{T}_1\tilde{T}_{s_{i-1}}^{-1}-f_{\mathcal{P},i-1}\left(1\right)\right\vert \left\vert f_{\mathcal{P},i-1}\left(1\right)^{-1}\right\vert \\
 & \leq\left\vert \tilde{T}_1\tilde{T}_{s_{i-1}}^{-1}-f_{\mathcal{P},i-1}\left(1\right)\right\vert. 
\end{align*}
So 
\[
\left\vert f_{\mathcal{P},i-1}\left(s_{i}\right)f_{\mathcal{P},i-1}\left(1\right)^{-1}-\tilde{T}_{s_{i}}\tilde{T}_1^{-1}\right\vert \leq2\underset{1\leq i,j\leq n}{\sup}\left\vert \tilde{T}_{s_{j}}\tilde{T}_{s_{i}}^{-1}-f_{\mathcal{P},i}\left(s_{j}\right)\right\vert. 
\]
Then using Theorem \ref{the.7.11}, Lemma \ref{lem.7.8} and Holder's
inequality we have
\begin{equation}
\mathbb{E}\left[\underset{s\in\mathcal{P}}{\sup}\left\vert \hat{\mathbf{K}}_\mathcal{P}\left(s\right)-\mathbf{K}_{\mathcal{P}}\left(s\right)\right\vert ^{q}\right]\leq C\left\vert \mathcal{P}\right\vert ^{q\gamma}\label{equ.7.24}
\end{equation}
Finally Lemma \ref{lem.7.20} is proved by combining $\left(\ref{equ.7.22}\right)$,$\left(\ref{equ.7.23}\right)$ and $\left(\ref{equ.7.24}\right)$. \end{proof}

\begin{lemma} \label{lem.7.21}For each $q\geq1$, there exists a constant $C>0$ such that 
\[
\underset{s\in\left[0,1\right]}{\sup}\left\vert \tilde{\mathbf{K}}_{\underline{s}}-\tilde{\mathbf{K}}_s\right\vert ^{q}\leq C\left\vert \mathcal{P}\right\vert ^{q}
\]

\end{lemma}

\begin{proof} By the fundamental theorem of calculus, we have
\[
\tilde{\mathbf{K}}_s=-\frac{1}{2}\int_{0}^{s}Ric_{\tilde{u}_r}\tilde{\mathbf{K}}_rdr+\int_{0}^{s}\left(\tilde{T}_1\tilde{T}_r^{-1}\right)^{\ast}dr.
\]
Using Lemma \ref{lem.4.3}, note that $Ric$ is bounded by $\left(d-1\right)N$, we have
\[
\left\vert \tilde{\mathbf{K}}_s\right\vert \leq \left(d-1\right)N\int_{0}^{s}\left\vert \tilde{\mathbf{K}}_r\right\vert dr+C
\]
where $C$ and $\left(d-1\right)N$ are two constants independent of $s$. Then using Gronwall's inequality we get
\begin{equation}
\left\vert \tilde{\mathbf{K}}_s\right\vert \leq Ce^{Ns}\leq Ce^{N}\label{equ.7.25}
\end{equation}
so $\underset{s\in\left[0,1\right]}{\sup}\left\vert \tilde{\mathbf{K}}_s\right\vert $
is bounded. Then using the fundamental theorem of calculus again from
$\underline{s}$ to $s$ we have
\begin{align*}
\tilde{\mathbf{K}}_s-\tilde{\mathbf{K}}_{\underline{s}}= & -\frac{1}{2}\int_{\underline{s}}^{s}Ric_{\tilde{u}_r}\tilde{\mathbf{K}}_rdr+\int_{\underline{s}}^{s}\left(\tilde{T}_1\tilde{T}_r^{-1}\right)^{\ast}dr\\
= & -\frac{1}{2}\int_{\underline{s}}^{s}Ric_{\tilde{u}_r}\left(\tilde{\mathbf{K}}_r-\tilde{\mathbf{K}}_{\underline{r}}\right)dr+\int_{\underline{s}}^{s}\left(\tilde{T}_1\tilde{T}_r^{-1}\right)^{\ast}dr+\frac{1}{2}\int_{\underline{s}}^{s}Ric_{\tilde{u}_r}\tilde{\mathbf{K}}_{\underline{r}}dr.
\end{align*}
Therefore 
\[
\left\vert \tilde{\mathbf{K}}_s-\tilde{\mathbf{K}}_{\underline{s}}\right\vert \leq\frac{N}{2}\int_{\underline{s}}^{s}\left\vert \tilde{\mathbf{K}}_r-\tilde{\mathbf{K}}_{\underline{r}}\right\vert dr+C\left\vert \mathcal{P}\right\vert. 
\]
By Gronwall's inequality again we have
\[
\left\vert \tilde{\mathbf{K}}_s-\tilde{\mathbf{K}}_{\underline{s}}\right\vert \leq C\left\vert \mathcal{P}\right\vert e^{\frac{N}{2}}
\]
and thus
\[
\underset{s\in\left[0,1\right]}{\sup}\left\vert \tilde{\mathbf{K}}_{\underline{s}}-\tilde{\mathbf{K}}_s\right\vert ^{q}\leq C\left\vert \mathcal{P}\right\vert ^{q}
\]
\end{proof}

The next theorem is a generalization to Proposition \ref{lem.7.20} in the sense that $s$ now can be taken to be arbitrary between 0 and 1.  
\begin{theorem} \label{the.7.17}For each $q\geq1$ and $\gamma\in \left(0,\frac{1}{2}\right)$, there exists a constant $C>0$ such that for all $n>5qN$,  
\begin{equation}
\mathbb{E}\left[\underset{s\in\left[0,1\right]}{\sup}\left\vert \tilde{\mathbf{K}}_s-\mathbf{K}_{\mathcal{P}}\left(s\right)\right\vert ^{q}\right]\leq C\left\vert \mathcal{P}\right\vert ^{\gamma q}\label{equ.7.15}
\end{equation}

\end{theorem}

\begin{proof}For any $s\in\left[0,1\right]$, $s\in\left[s_{i-1},s_{i}\right]$
for some $i\in\left\{ 1,\cdots,n\right\} $. So
\begin{align*}
\left\vert \mathbf{K}_{\mathcal{P}}\left(s\right)-\tilde{\mathbf{K}}_s\right\vert  & \leq\left\vert \mathbf{K}_{\mathcal{P}}\left(s\right)-\mathbf{K_{\mathcal{P}}}\left(s_{i-1}\right)\right\vert +\left\vert \mathbf{K}_{\mathcal{P}}\left(s_{i-1}\right)-\tilde{\mathbf{K}}_{s_{i-1}}\right\vert +\left\vert \tilde{\mathbf{K}}_{s_{i-1}}-\tilde{\mathbf{K}}_s\right\vert. 
\end{align*}
Then using Lemma \ref{lem.7.15}, Proposition \ref{lem.7.20} and \ref{lem.7.21} we prove this theorem. \end{proof}

\subsubsection{Convergence of $J_{\mathcal{P}}\left(s\right)$ to $\tilde{J}_s$\label{sub.6.1.4}}
Recall from Definition \ref{def.6.13} that $J_{\mathcal{P}}\left(s\right) :=\mathbf{K}_{\mathcal{P}}\left(s\right)\mathbf{K}_{\mathcal{P}}\left(1\right)^{-1}H_\mathcal{P}$, where $H_\mathcal{P}:W_o\left(M\right)\to \mathbb{R}^d$ is given by $H_\mathcal{P}=u_{\mathcal{P}}\left(1\right)^{-1}X\left(\pi\circ u_{\mathcal{P}}\left(1\right)\right)$ and $u_{\mathcal{P}}$ is interpreted in Notation \ref{Not1}. 

\begin{proposition} \label{pro.7.22}Let $\tilde{J}_s$
be as in Definition \ref{def.4.7-1} and $X\in \Gamma\left(TM\right)$ with compact support, then for any $q\geq 1$,  \[\underset{\left\vert\mathcal{P}\right\vert\to 0}{\lim}\mathbb{E}\left[\underset{s\in\left[0,1\right]}{\sup}\left\vert J_{\mathcal{P}}\left(s\right)-\tilde{J}_s\right\vert^q\right]=0.\] \end{proposition}

\begin{proof}
\[
\left\vert J_{\mathcal{P}}\left(s\right)-\tilde{J}_s\right\vert \leq I_\mathcal{P}\left(s\right)+II_\mathcal{P}\left(s\right)+III_\mathcal{P}\left(s\right),
\]
where
\begin{align*}
I_\mathcal{P}\left(s\right) & =\left\vert \tilde{\mathbf{K}}_s-\mathbf{K}_{\mathcal{P}}\left(s\right)\right\vert \left\vert \mathbf{K}_{\mathcal{P}}\left(1\right)^{-1}\right\vert \left\vert H_{\mathcal{P}}\right\vert \\
II_\mathcal{P}\left(s\right) & =\left\vert \tilde{\mathbf{K}}_s\right\vert \left\vert \mathbf{K}_{\mathcal{P}}\left(1\right)^{-1}-\tilde{\mathbf{K}}_1^{-1}\right\vert \left\vert H_{\mathcal{P}}\right\vert \\
III_\mathcal{P}\left(s\right) & =\left\vert \tilde{\mathbf{K}}_s\right\vert \left\vert \tilde{\mathbf{K}}_1^{-1}\right\vert \left\vert H_{\mathcal{P}}-\tilde{H}\right\vert. 
\end{align*}
For $I_\mathcal{P}\left(s\right)$, since $X$ has compact support, $\left\vert H_{\mathcal{P}}\left(\sigma\right)\right\vert $
is bounded. By Lemma \ref{lem.6.10} $\left\vert \mathbf{K}_{\mathcal{P}}\left(1\right)^{-1}\right\vert \leq1$.
Then using Theorem \ref{the.7.17} we have
\begin{equation}
\mathbb{E}\left[\sup_{0\leq s\leq 1}I^q_\mathcal{P}\left(s\right)\right]\leq C\left\vert \mathcal{P}\right\vert ^{q\gamma}\text{ for }n>5qN.\label{equ.7.26}
\end{equation}
For $II_\mathcal{P}\left(s\right):$ since
\begin{equation}
\mathbf{K}_{\mathcal{P}}\left(1\right)^{-1}-\tilde{\mathbf{K}}_1^{-1}=\mathbf{K}_{\mathcal{P}}\left(1\right)^{-1}\left(\tilde{\mathbf{K}}_1-\mathbf{K}_{\mathcal{P}}\left(1\right)\right)\tilde{\mathbf{K}}_1^{-1},\label{eq.1}
\end{equation}
so
\begin{align*}
II_\mathcal{P}\left(s\right) & \leq\left\vert \tilde{\mathbf{K}}_s\right\vert \left\vert \mathbf{K}_{\mathcal{P}}\left(1\right)^{-1}\right\vert \left\vert \tilde{\mathbf{K}}_1-\mathbf{K}_{\mathcal{P}}\left(1\right)\right\vert \left\vert \tilde{\mathbf{K}}_1^{-1}\right\vert \left\vert H_{\mathcal{P}}\right\vert\leq C\underset{s\in\left[0,1\right]}{\sup}\left\vert \tilde{\mathbf{K}}_s\right\vert \left\vert \tilde{\mathbf{K}}_1-\mathbf{K}_{\mathcal{P}}\left(1\right)\right\vert .
\end{align*}
Recall from $\left(\ref{equ.7.25} \right)$ that $\underset{s\in\left[0,1\right]}{\sup}\left\vert \tilde{\mathbf{K}}_s\right\vert $ is bounded (the bound is deterministic), using Theorem \ref{the.7.17} again we have 
\begin{equation}
\mathbb{E}\left[\sup_{0\leq s\leq 1}II^{q}_\mathcal{P}\left(s\right)\right]\leq C\left\vert \mathcal{P}\right\vert ^{q\gamma}\text{ for }n>5qN.\label{equ.7.27}
\end{equation}
For $III_\mathcal{P}\left(s\right)$: Since $F:\mathcal{O}\left(M\right)\to \mathbb{R}^d$ given by $F\left(y\right)=y^{-1}X\circ \pi\left(y\right)$ is bounded and continuous, and by the Wong--Zakai approximation Theorem (for example, see Theorem 10 in \cite{Elworthy82}), $u_{\mathcal{P}}\left(1\right)\to \tilde{u}_1$ in probability as $\left|\mathcal{P}\right|\to 0$, by DCT, \begin{equation}
H_\mathcal{P}\to \tilde{H}\text{ in }L^{\infty-}\left(W_o\left(M\right)\right)\text{ as }\left|\mathcal{P}\right|\to 0.\label{eq.2}
\end{equation}
Also since $\underset{s\in\left[0,1\right]}{\sup}\left\vert \tilde{\mathbf{K}}_s\right\vert $ and $\left\vert \tilde{\mathbf{K}}_1^{-1}\right\vert$ are bounded, we have
\begin{equation}
\sup_{0\leq s\leq 1}III_\mathcal{P}\left(s\right)\to 0\text{ in }L^{\infty-}\left(W_o\left(M\right)\right)\text{ as }\left|\mathcal{P}\right|\to 0.\label{equ.7.28}
\end{equation}
Combining Eq.$\left(\ref{equ.7.26}\right)$, $\left(\ref{equ.7.27}\right)$ and $\left(\ref{equ.7.28}\right)$ we prove
this proposition. \end{proof}

\subsection{Convergence of $\tilde{X}_{\mathcal{P}}^{tr,\nu_{\mathcal{P}}^{1}}$
to $\left(\tilde{X}\right)^{tr,\nu}$\label{sec.6.2}}
Recall from Lemma \ref{lem.8.1} and Theorem \ref{lem.5.1} that
\begin{equation}
\tilde{X}_{\mathcal{P}}^{tr,\nu_{\mathcal{P}}^{1}}=-\tilde{X}_{\mathcal{P}}+\int_{0}^{1}\left\langle J_{\mathcal{P}}^{\prime}\left(s+\right),d\beta_{\mathcal{P},s}\right\rangle-div\tilde{X}_{\mathcal{P}}\label{1}
\end{equation}
and 
\begin{equation}
\tilde{X}^{tr,\nu} =-\tilde{X}+\sum_{\alpha=1}^{d}\left\langle \tilde{C}\tilde{H},e_{\alpha}\right\rangle \int_{0}^{1}\left\langle \left(\tilde{T}_s^{-1}\right)^{\ast}e_{\alpha},d\beta_{s}\right\rangle -\sum_{\alpha=1}^{d}\left\langle X^{Z_{\alpha}}\left(\tilde{C}\tilde{H}\right),e_{\alpha}\right\rangle .\label{3}
\end{equation}
\begin{theorem} \label{the.8.4}If $M$ has parallel curvature tensor, i.e. $\nabla R\equiv 0$, then for any $f\in\mathcal{FC}_b^1$ and $q\geq 1$, \[\underset{\left|\mathcal{P}\right|\to 0}{\lim}\mathbb{E}\left[\left\vert\tilde{X}_{\mathcal{P}}^{tr,\nu_{\mathcal{P}}^{1}}f-\tilde{X}^{tr,\nu}f\right\vert^q\right]=0.\]where according to Notation \ref{Not1}, $\tilde{X}_{\mathcal{P}}^{tr,\nu_{\mathcal{P}}^{1}}f$ is interpreted as $\left(\tilde{X}_{\mathcal{P}}^{tr,\nu_{\mathcal{P}}^{1}}\left(f\mid_{H_\mathcal{P}\left(M\right)}\right)\right)\circ \phi\circ \beta_{\mathcal{P}}$. \end{theorem}
\begin{proof} 
In correspondence with the three--term formulae (\ref{1}) and (\ref{3}), this theorem is decomposed as three propositions: Proposition \ref{pro.8.4} states that
\[\underset{\left|\mathcal{P}\right|\to 0}{\lim}\mathbb{E}\left[\left\vert\tilde{X}_{\mathcal{P}}f\to \tilde{X}f\right\vert^q\right]=0,\] Proposition \ref{pro.8.5} states that
\[\underset{\left|\mathcal{P}\right|\to 0}{\lim}\mathbb{E}\left[\left\vert\int_{0}^{1}\left\langle J_{\mathcal{P}}^{\prime}\left(s+\right),d\beta_{\mathcal{P},s}\right\rangle- \sum_{\alpha=1}^{d}\left\langle \tilde{C}\tilde{H},e_{\alpha}\right\rangle \int_{0}^{1}\left\langle \left(\tilde{T}_s^{-1}\right)^{\ast}e_{\alpha},d\beta_{s}\right\rangle\right\vert^q\right]=0\]
and Proposition \ref{pro.8.6} states that
\[\underset{\left|\mathcal{P}\right|\to 0}{\lim}\mathbb{E}\left[\left\vert div\tilde{X}_{\mathcal{P}}-\sum_{\alpha=1}^{d}\left\langle X^{Z_{\alpha}}\left(\tilde{C}\tilde{H}\right),e_{\alpha}\right\rangle\right\vert^q\right]=0.\]
Thus the proof will be complete once the stated propositions are proved.
\end{proof}
\begin{remark}
For Proposition \ref{pro.8.4} and \ref{pro.8.5} we assume the assumption of bounded sectional curvature as is mentioned in the beginning of this section. For Proposition \ref{pro.8.6} we further require the curvature tensor to be covariantly constant.
\end{remark}
\begin{proposition}\label{pro.8.4}
If $X\in \Gamma\left(TM\right)$ with compact support and $f\in \mathcal{FC}^1$, then for any $q\geq 1$,
\[\underset{\left|\mathcal{P}\right|\to 0}{\lim}\mathbb{E}\left[\left\vert\tilde{X}_{\mathcal{P}}f\to \tilde{X}f\right\vert^q\right]=0.\]
\end{proposition}
\begin{proposition} \label{pro.8.5}Keeping the notation above, we have for any $q\geq 1$,
\[\underset{\left|\mathcal{P}\right|\to 0}{\lim}\mathbb{E}\left[\left\vert\int_{0}^{1}\left\langle J_{\mathcal{P}}^{\prime}\left(s+\right),d\beta_{\mathcal{P},s}\right\rangle- \sum_{\alpha=1}^{d}\left\langle \tilde{C}\tilde{H},e_{\alpha}\right\rangle \int_{0}^{1}\left\langle \left(\tilde{T}_s^{-1}\right)^{\ast}e_{\alpha},d\beta_{s}\right\rangle\right\vert^q\right]=0.\]
\end{proposition}
\begin{proposition} \label{pro.8.6}Continuing the notation above, if we further assume $\nabla R\equiv 0$, then for any $q\geq 1$,
\[\underset{\left|\mathcal{P}\right|\to 0}{\lim}\mathbb{E}\left[\left\vert div\tilde{X}_{\mathcal{P}}-\sum_{\alpha=1}^{d}\left\langle X^{Z_{\alpha}}\left(\tilde{C}\tilde{H}\right),e_{\alpha}\right\rangle\right\vert^q\right]=0.\]
\end{proposition}
\begin{proof}[Proof of Proposition \ref{pro.8.4}] 
Using Eq.(\ref{gradient}) and the fact that $\pi\circ u_\mathcal{P}=\phi\circ \beta_\mathcal{P}$, we have
\begin{align*}
\tilde{X}_\mathcal{P}f&=\sum_{i=1}^{n}\left<\left(grad_iF\right)\left(\pi\circ u_\mathcal{P}\right),u_\mathcal{P}\left(s_i\right)J_\mathcal{P}\left(s_i\right)\right>=\sum_{i=1}^{n}\left<u^{-1}_\mathcal{P}\left(s_i\right)\left(grad_iF\right)\left(\pi\circ u_\mathcal{P}\right),J_\mathcal{P}\left(s_i\right)\right>,
\end{align*}	
and
\begin{align*}
\tilde{X}f &=\sum_{i=1}^{n}\left<\left(grad_iF\right)\left(\pi\circ \tilde{u}\right),\tilde{u}_{s_i}\tilde{J}_{s_i}\right>=\sum_{i=1}^{n}\left<\tilde{u}^{-1}_{s_i}\left(grad_iF\right)\left(\pi\circ \tilde{u}\right),\tilde{J}_{s_i}\right>.
\end{align*}
where $F$ is a representation of $f$ as in Definition \ref{def.rcf}.

Since $W\left(\mathcal{O}(M)\right)\ni y\to y^{-1}_{s_i}grad_iF\left(\pi\circ y\right)\in \mathbb{R}^d$ is continuous and bounded, using Theorem 10 in \cite{Elworthy82} and DCT, we know
\begin{equation}
u^{-1}_\mathcal{P}\left(s_i\right)\left(grad_iF\right)\left(\pi\circ u_\mathcal{P}\right)\to \tilde{u}^{-1}_{s_i}\left(grad_iF\right)\left(\pi\circ \tilde{u}\right)\text{ in }L^{\infty-}\left(W_o\left(M\right)\right)\text{ as }\left|\mathcal{P}\right|\to 0.\label{grad}
\end{equation}
The proof is then completed by making use of (\ref{grad}) and Proposition \ref{pro.7.22}.
\end{proof}

\begin{proof}[Proof of Proposition \ref{pro.8.5}] 
\begin{align*}
\int_{0}^{1}\left\langle J_{\mathcal{P}}^{\prime}\left(s+\right),d\beta_{\mathcal{P}}\left(s\right)\right\rangle  & =\sum_{i=1}^{n}\left\langle \frac{J_{\mathcal{P}}\left(s_{i}\right)-J_{\mathcal{P}}\left(s_{i-1}\right)}{\Delta_{i}},\Delta_{i}\beta\right\rangle \\
 & =\sum_{i=1}^{n}\left\langle J_{\mathcal{P}}^{\prime}\left(s_{i-1}\right),\Delta_{i}\beta\right\rangle +\sum_{i=1}^{n}\frac{1}{\Delta_i}\left\langle \int_{s_{i-1}}^{s_{i}}J_{\mathcal{P}}^{\prime\prime}\left(s\right)\left(s-s_{i-1}\right)ds,\Delta_{i}\beta\right\rangle \\
 & =I_\mathcal{P}+II_\mathcal{P},
\end{align*}
where 
\[
I_\mathcal{P}=\sum_{i=1}^{n}\left\langle J_{\mathcal{P}}^{\prime}\left(s_{i-1}\right),\Delta_{i}\beta\right\rangle 
\]
and
\[II_\mathcal{P}=\sum_{i=1}^{n}\frac{1}{\Delta_i}\left\langle \int_{s_{i-1}}^{s_{i}}J_{\mathcal{P}}^{\prime\prime}\left(s\right)\left(s-s_{i-1}\right)ds,\Delta_{i}\beta\right\rangle.\]
Using the fact that $J_\mathcal{P}$ satisfies Jacobi equation, we further have
\begin{align*}
II_\mathcal{P} & =\sum_{i=1}^{n}\left\langle \frac{1}{\Delta_{i}^{3}}\int_{s_{i-1}}^{s_{i}}R_{u_{\mathcal{P}}\left(s\right)}\left(\Delta_{i}\beta,J_{\mathcal{P}}\left(s\right)\right)\Delta_{i}\beta\left(s-s_{i-1}\right)ds,\Delta_{i}\beta\right\rangle \\&=\sum_{i=1}^{n}\frac{1}{\Delta_{i}^{3}}\int_{s_{i-1}}^{s_{i}}\left\langle R_{u_{\mathcal{P}}\left(s\right)}\left(\Delta_{i}\beta,J_{\mathcal{P}}\left(s\right)\right)\Delta_{i}\beta,\Delta_{i}\beta\right\rangle \left(s-s_{i-1}\right)ds. 
\end{align*}
Since the curvature tensor is anti-symmetric, \[\left\langle R_{u_{\mathcal{P}}\left(s\right)}\left(\Delta_{i}\beta,J_{\mathcal{P}}\left(s\right)\right)\Delta_{i}\beta,\Delta_{i}\beta\right\rangle\equiv 0\text{ }\nu\text{--a.s,}\] so $II_\mathcal{P}\equiv 0\text{  }\nu\text{--a.s}$. 
\begin{align*}
I_\mathcal{P} & =\sum_{i=1}^{n}\left\langle f_{\mathcal{P},i}^{*}\left(1\right)\mathbf{K_{\mathcal{P}}}\left(1\right)^{-1}H_{\mathcal{P}},\Delta_{i}\beta\right\rangle \\
 & =\sum_{i=1}^{n}\left\langle \mathbf{K_{\mathcal{P}}}\left(1\right)^{-1}H_{\mathcal{P}},f_{\mathcal{P},i}\left(1\right)\Delta_{i}\beta\right\rangle=\left\langle \mathbf{K_{\mathcal{P}}}\left(1\right)^{-1}H_{\mathcal{P}},\sum_{i=1}^{n}f_{\mathcal{P},i}\left(1\right)\Delta_{i}\beta\right\rangle. 
\end{align*}
For each $i\geq1,s\in\left[s_{i-1},s_{i}\right],$ define $g_{i}\left(s\right):=S_{\mathcal{P},i}\left(s\right)-C_{\mathcal{P},i}\left(s\right)S_{\mathcal{P},i-1}$. Then Taylor's expansion of $g_{i}$ at $s_{i-1}$ gives
\[
g_{i}\left(s\right)=-S_{\mathcal{P},i-1}+\left(s-s_{i-1}\right)I+\int_{s_{i-1}}^{s}A_{\mathcal{P},i}(r)g_i(r)\left(s-r\right)dr.
\]
So 
\[
\left\vert g_{i}\left(s\right)\right\vert \leq\left\vert S_{\mathcal{P},i-1}-\left(s-s_{i-1}\right)I\right\vert +N\left\vert \beta_{\mathcal{P}}^{\prime}\left(s_{i-1}\right)\right\vert ^{2}\int_{s_{i-1}}^{s}\left\vert g_{i}\left(r\right)\right\vert \left(s-r\right)dr.
\]
By Gronwall's inequality and Eq.(\ref{lem.6.3}), we have
\[
\left\vert g_{i}\left(s_{i}\right)\right\vert \leq\frac{N}{6}K_{\gamma}^{2}\left\vert \mathcal{P}\right\vert ^{2\gamma+1}e^{\frac{1}{2}N\left\vert \Delta_{i}\beta\right\vert ^{2}}.
\]
Note that $g_i\left(s_i\right)=S_{\mathcal{P},i}-C_{\mathcal{P},i}S_{\mathcal{P},i-1}$, so by Eq.(\ref{lem.6.1}), 
\begin{align*}
\left\vert f_{\mathcal{P},i}\left(1\right)-f_{\mathcal{P},i-1}\left(1\right)\right\vert &\leq\frac{1}{\left\vert \mathcal{P}\right\vert }\left\vert C_{\mathcal{P},n}\right\vert \cdot\cdots\cdot\left\vert C_{\mathcal{P},i+1}\right\vert \cdot\left\vert S_{\mathcal{P},i}-C_{\mathcal{P},i}S_{\mathcal{P},i-1}\right\vert \\& \leq CK_{\gamma}^{2}\left\vert \mathcal{P}\right\vert ^{2\gamma}e^{\sum_{i=1}^{n}N\left\vert \Delta_{i}\beta\right\vert ^{2}}
\end{align*}
and thus
\begin{align*}
\left\vert \sum_{i=1}^{n}f_{\mathcal{P},i}\left(1\right)\Delta_{i}\beta-\sum_{i=1}^{n}f_{\mathcal{P},i-1}\left(1\right)\Delta_{i}\beta\right\vert ^{q} & \leq\left\vert \mathcal{P}\right\vert ^{1-q}\left[\sum_{i=1}^{n}\left\vert f_{\mathcal{P},i}\left(1\right)-f_{\mathcal{P},i-1}\left(1\right)\right\vert ^{q}\left\vert \Delta_{i}\beta\right\vert ^{q}\right]\\
 & \leq CK_{\gamma}^{3q}\left\vert \mathcal{P}\right\vert ^{3q\gamma-q}e^{\sum_{i=1}^{n}qN\left\vert \Delta_{i}\beta\right\vert ^{2}}.
\end{align*}
Picking $\gamma\in \left(\frac{1}{2},\frac{1}{3}\right)$ we know for any $q\geq 1$,
\begin{equation}
\mathbb{E}\left[\left\vert \sum_{i=1}^{n}f_{\mathcal{P},i}\left(1\right)\Delta_{i}\beta-\sum_{i=1}^{n}f_{\mathcal{P},i-1}\left(1\right)\Delta_{i}\beta\right\vert ^{q}\right]\to 0\text{ as }\left|\mathcal{P}\right|\to 0.\label{equ.8.11}
\end{equation}
Since $f_{\mathcal{P},i-1}\left(1\right)=f_{\mathcal{P},0}\left(1\right)f_{\mathcal{P},0}^{-1}\left(s_{i-1}\right)\frac{S_{\mathcal{P},i-1}}{\Delta_{i-1}},$
so
\begin{align}
&\left\langle \mathbf{K_{\mathcal{P}}}\left(1\right)^{-1}H_{\mathcal{P}},\sum_{i=1}^{n}f_{\mathcal{P},i-1}(1)\Delta_{i}\beta\right\rangle=\left\langle f_{\mathcal{P},0}^{\ast}\left(1\right)\mathbf{K_{\mathcal{P}}}\left(1\right)^{-1}H_{\mathcal{P}},\sum_{i=1}^{n}f_{\mathcal{P},0}^{-1}\left(s_{i-1}\right)\frac{S_{\mathcal{P},i-1}}{\Delta_{i-1}}\Delta_{i}\beta\right\rangle. 
\end{align}
Using Proposition \ref{prop A-1} we have $\left\vert f_{\mathcal{P},0}^{-1}\left(s_{i-1}\right)\right\vert \leq1$. Then using Eq.(\ref{lem.6.3}) we obtain 
\begin{align*}
\left\vert f_{\mathcal{P},0}^{-1}\left(s_{i-1}\right)\frac{S_{\mathcal{P},i-1}}{\Delta_{i-1}}-f_{\mathcal{P},0}^{-1}\left(s_{i-1}\right)\right\vert \left\vert \Delta_{i}\beta\right\vert & \leq\left\vert \frac{S_{\mathcal{P},i-1}}{\Delta_{i-1}}-I\right\vert \left\vert \Delta_{i}\beta\right\vert \leq\frac{NK_{\gamma}^{3}\left\vert \mathcal{P}\right\vert ^{3\gamma}}{6}e^{\frac{N}{2}\left\vert \Delta_{i-1}\beta\right\vert ^{2}}.
\end{align*}
Therefore for each $q\geq1$, 
\begin{align}
\left\vert \sum_{i=1}^{n}f_{\mathcal{P},0}^{-1}\left(s_{i-1}\right)\frac{S_{\mathcal{P},i-1}}{\Delta_{i-1}}\Delta_{i}\beta-\sum_{i=1}^{n}f_{\mathcal{P},0}^{-1}\left(s_{i-1}\right)\Delta_{i}\beta\right\vert ^{q}& \leq\left\vert \mathcal{P}\right\vert ^{1-q}\sum_{i=1}^{n}\frac{N^{q}K_{\gamma}^{3q}\left\vert \mathcal{P}\right\vert ^{3q\gamma}}{6^{q}}e^{\frac{Nq}{2}\left\vert \Delta_{i-1}\beta\right\vert ^{2}}\nonumber\\
 & \leq C\left\vert \mathcal{P}\right\vert ^{\left(3\gamma-1\right) q}K_{\gamma}^{3q}e^{\sum_{i=1}^{n}\frac{Nq}{2}\left\vert \Delta_{i-1}\beta\right\vert ^{2}}.
\end{align}
Picking $\gamma\in \left(\frac{1}{3},\frac{1}{2}\right)$, we have
\[
\mathbb{E}\left[\left\vert \sum_{i=1}^{n}f_{\mathcal{P},0}^{-1}\left(s_{i-1}\right)\frac{S_{\mathcal{P},i-1}}{\Delta_{i-1}}\Delta_{i}\beta-\sum_{i=1}^{n}f_{\mathcal{P},0}^{-1}\left(s_{i-1}\right)\Delta_{i}\beta\right\vert ^{q}\right]\to 0\text{ as }\left\vert\mathcal{P}\right\vert\to 0.
\]
Rewrite $\sum_{i=1}^{n}f_{\mathcal{P},0}^{-1}\left(s_{i-1}\right)\Delta_{i}\beta\text{ as }\int_{0}^{1}f_{\mathcal{P}}\left(s\right)d\beta_{s}
$, where $f_{\mathcal{P}}\left(s\right):=\sum_{i=1}^{n}f_{\mathcal{P},0}^{-1}\left(s_{i-1}\right)1_{[s_{i-1},s_{i})}\left(s\right)$. Define a martingale
$M_{r}:=\int_{0}^{r}f_{\mathcal{P}}\left(s\right)d\beta_{s}-\int_{0}^{r}\tilde{T}_s^{-1}d\beta_{s}$. Then by Burkholder-Davis-Gundy inequality, for each $q\geq1$, 
\[
\mathbb{E}\left[\underset{r\in\left[0,1\right]}{\sup}\left\vert M_{r}\right\vert ^{q}\right]\leq C\mathbb{E}\left[\left\langle M\right\rangle _{1}^{\frac{q}{2}}\right],
\]
where
\begin{align*}
\left\langle M\right\rangle _{1}&\leq\int_{0}^{1}\left\vert f_{\mathcal{P}}\left(s\right)-\tilde{T}_s^{-1}\right\vert ^{2}ds\leq 2\int_{0}^{1}\left\vert f_{\mathcal{P}}\left(s\right)-\tilde{T}_{\underline{s}}^{-1}\right\vert ^{2}ds+2\int_{0}^{1}\left\vert \tilde{T}_{\underline{s}}^{-1}-\tilde{T}_s^{-1}\right\vert ^{2}ds.
\end{align*}
Since 
\begin{align}
\int_{0}^{1}\left\vert f_{\mathcal{P}}\left(s\right)-\tilde{T}_{\underline{s}}^{-1}\right\vert ^{2}ds & =\sum_{i=1}^{n}\left\vert f_{\mathcal{P},0}^{-1}\left(s_{i-1}\right)-\tilde{T}_{s_{i-1}}^{-1}\right\vert ^{2}\Delta_{i}\nonumber \\
 & \leq\sum_{i=1}^{n}\left\vert f_{\mathcal{P},0}^{-1}\left(s_{i-1}\right)\right\vert ^{2}\left\vert f_{\mathcal{P},0}\left(s_{i-1}\right)-\tilde{T}_{s_{i-1}}\right\vert ^{2}\left\vert \tilde{T}_{s_{i-1}}^{-1}\right\vert ^{2}\Delta_{i}\nonumber \\
 & \leq\underset{i\in \left\{0,\dots,n\right\}}{\sup}\left\vert f_{\mathcal{P},0}\left(s_i\right)-\tilde{T}_{s_i}\right\vert ^{2}\label{equ.8.12}
\end{align}
and 
\begin{align}
\int_{0}^{1}\left\vert \tilde{T}_{\underline{s}}^{-1}-\tilde{T}_s^{-1}\right\vert ^{2}ds & =\int_{0}^{1}\left\vert \int_{s}^{\underline{s}}\left(\tilde{T}_r^{-1}\right)^{\prime}dr\right\vert ^{2}ds\nonumber  \leq\int_{0}^{1}N\left\vert s-\underline{s}\right\vert ^{2}ds\leq N\left\vert \mathcal{P}\right\vert ^{2},\label{equ.8.13}
\end{align}
so 
\begin{align*}
\left\langle M\right\rangle _{1}^{\frac{q}{2}} & \leq C\left(\int_{0}^{1}\left\vert f_{\mathcal{P}}\left(s\right)-\tilde{T}_{\underline{s}}^{-1}\right\vert ^{2}ds\right)^{\frac{q}{2}}+C\left(\int_{0}^{1}\left\vert \tilde{T}_{\underline{s}}^{-1}-\tilde{T}_s^{-1}\right\vert ^{2}ds\right)^{\frac{q}{2}}\\
 & \leq C\left(\underset{i\in \left\{0,\dots,n\right\}}{\sup}\left\vert f_{\mathcal{P},0}\left(s_i\right)-\tilde{T}_{s_i}\right\vert ^{q}+\left\vert \mathcal{P}\right\vert ^{q}\right).
\end{align*}
Then using Theorem \ref{the.7.11} we have
\begin{equation}
\mathbb{E}\left[\left\langle M\right\rangle _{1}^{\frac{q}{2}}\right]\leq C\left\vert \mathcal{P}\right\vert ^{q\gamma}.\label{2}
\end{equation}
Then it follows that for each $q\geq 1$, 
\[
\int_{0}^{1}f_{\mathcal{P}}\left(s\right)d\beta_{s}-\int_{0}^{1}\tilde{T}_s^{-1}d\beta_{s}\to 0\text{ in }L^{q}\left(W_o\left(M\right)\right)\text{ as }\left|\mathcal{P}\right|\to 0.
\]
Then using Eq.(\ref{eq.1}), Eq.(\ref{eq.2}) and Theorem \ref{the.7.17} we have
\[
\mathbf{K_{\mathcal{P}}}\left(1\right)^{-1}H_{\mathcal{P}}\to\tilde{\mathbf{K}}_1^{-1}\tilde{H}\text{ in }L^{\infty-}\left(W_o\left(M\right)\right)\text{ as }\left|\mathcal{P}\right|\to 0
\]
and 
\[
f_{\mathcal{P},0}^{\ast}\left(1\right)\to \tilde{T}_1^{\ast}\text{ in }L^{\infty-}\left(W_o\left(M\right)\right)\text{ as }\left|\mathcal{P}\right|\to 0,
\]
therefore
\begin{equation}
I_\mathcal{P}\rightarrow\left\langle T^{\ast}_1\tilde{\mathbf{K}}_1^{-1}\tilde{H},\int_{0}^{1}\tilde{T}_s^{-1}d\beta_{s}\right\rangle \text{ in }L^{\infty-}\left(W_o\left(M\right)\right)\text{ as }\left|\mathcal{P}\right|\to 0.\label{equ.8.14}
\end{equation}
Lastly, notice that $\tilde{\mathbf{K}}_1=\tilde{T}_1\int_{0}^{1}\left(\tilde{T}_r^{\ast}\tilde{T}_r\right)^{-1}drT_1^{\ast}$, so 
\[
\tilde{\mathbf{K}}_1^{-1}=\left({\tilde{T}_1^{-1}}\right)^{\ast}\tilde{C}
\]
where $\tilde{C}$ is defined in Theorem \ref{lem.5.1}, and 
\begin{equation}
\left\langle T^{\ast}_1\tilde{\mathbf{K}}_1^{-1}\tilde{H},\int_{0}^{1}\tilde{T}_s^{-1}d\beta_{s}\right\rangle =\left\langle \tilde{C}\tilde{H},\int_{0}^{1}\tilde{T}_s^{-1}d\beta_{s}\right\rangle=\sum_{\alpha=1}^{d}\left\langle \tilde{C}\tilde{H},e_{\alpha}\right\rangle \int_{0}^{1}\left\langle \left(\tilde{T}_s^{-1}\right)^{\ast}e_{\alpha},d\beta_{s}\right\rangle. \label{eq1}
\end{equation}
The proof is completed by combining Eq.(\ref{equ.8.14}) and (\ref{eq1}).
\end{proof}

We state two supplementary lemmas.
\begin{lemma} \label{lem.5.2}If the curvature tensor is parallel, i.e. $\nabla R\equiv 0$, then
	\begin{align}
	- & \sum_{\alpha=1}^{d}\left\langle X^{Z_{\alpha}}\left(\tilde{C}\tilde{H}\right),e_{\alpha}\right\rangle =divX\circ\Ep-\sum_{\alpha=1}^{d}\left\langle \tilde{C}A_1\left\langle Z_{\alpha}\right\rangle \tilde{H},e_{\alpha}\right\rangle .\label{equ.5.1}
	\end{align}
	where $A_s\left<Z_\alpha\right>=\int_{0}^{s}R_{\tilde{u}_r}\left(Z_\alpha(r),\delta\beta_r\right)$.
\end{lemma}
\begin{proof}
	See Lemma 4.25 in IVP. 
\end{proof}
\begin{lemma} \label{lem.8.8}Fix $s\in\left[0,1\right]$,
consider an one parameter family of paths $\left\{ \sigma_{t}\right\}\subset H_\mathcal{P}\left(M\right) $
and denote by $u_t\left(\cdot\right)$ the horizontal lift of $\sigma_t$. For simplicity, we will denote $u_t\left(s\right)$ by $u_t$, $\sigma_0$ by $\sigma$, the derivative with respect to $t$ by $\cdot$ and the derivative with respect to $s$ by $\prime$. For any $X\in\Gamma\left(TM\right)$, define $f_{X}:\mathcal{\mathcal{O}}\left(M\right)\mapsto\mathbb{R}^{d}\simeq T_{o}M$
by 
\[
f_{X}\left(u\right)=u^{-1}\left(X\circ\pi\right)\left(u\right)
\]
Then: 
\begin{align*}
\frac{d}{dt}|_{0}f_{X}\left(u_t\right)=\left(\frac{d}{dt}|_{0}u_t\right)f_{X}&=u_0^{-1}\nabla_{\dot{\sigma}\left(s\right)}X-\int_{0}^{s}R_{u_0{\left(r\right)}}\left(u_0\left(r\right)^{-1}\sigma^{\prime}\left(r+\right),u_0\left(r\right)^{-1}\dot{\sigma}\left(r\right)\right)drf_X\left(u_0\right).
\end{align*}
\end{lemma}
\begin{proof}
The connection on $\mathcal{O}(M)$ defined in Definition \ref{con} gives the following decomposition:
\[
\dot{u}_0=B_{a}\left(u_0\right)+\tilde{A}\left(u_0\right)
\]
where $a=u_0^{-1}\frac{d}{dt}|_{0}\sigma_{t}\left(s\right)=u_0^{-1}\dot{\sigma}\left(s\right)\in T_{o}M$
and $\tilde{A}\left(u_0\right)=\frac{d}{dt}|_{0}u_0e^{tA}$ for some $A=u_0^{-1}\frac{\triangledown u_t}{dt}\left(0\right)\in\text{\ensuremath{\mathfrak{so}}}(d)$
and $B_{a}\left(u_0\right)=\frac{d}{dt}|_{0}//_{t}\left(\gamma\right)u_0$
where $\gamma$ satisfies $\dot{\gamma}\left(0\right)=u_0a$ and $\gamma\left(0\right)=\sigma\left(s\right)$.
In this example, we can choose $\gamma(\cdot)$ to be $\sigma_{(\cdot)}\left(s\right)$. So
\[
B_{a}\left(u_0\right)f_{X}=\frac{d}{dt}|_{0}u_0^{-1}//_{t}^{-1}\left(\gamma\right)\left(X\circ\pi\right)\left(//_{t}\left(\gamma\right)u_0\right)=u_0^{-1}\nabla_{\dot{\sigma}\left(s\right)}X
\]
and 
\[
\tilde{A}\left(u\right)f_{X}=\frac{d}{dt}|_{0}e^{-tA}u^{-1}\left(X\circ\pi\right)\left(ue^{tA}\right)=-Au_0^{-1}X\left(\sigma\left(1\right)\right)=-Af_x\left(u_0\right)
\]
Following the computation in Theorem 3.3 in \cite{Andersson1999}, we know that 
\[
A=\int_{0}^{s}R_{u_0{\left(r\right)}}\left(u_0\left(r\right)^{-1}\sigma^{\prime}\left(r+\right),u_0\left(r\right)^{-1}\dot{\sigma}\left(r\right)\right)dr.
\]
\end{proof}

\begin{proof}[Proof of Proposition \ref{pro.8.6}] Because of Lemma
\ref{lem.5.2}, it suffices to prove
\begin{align*}
\underset{\left|\mathcal{P}\right|\to 0}{\lim}\mathbb{E}\left[\left\vert   div\tilde{X}_{\mathcal{P}}- \left(divX\circ\Ep-\sum_{\alpha=1}^{d}\left\langle \tilde{C}A_1\left\langle Z_{\alpha}\right\rangle \tilde{H},e_{\alpha}\right\rangle\right)\right\vert^q\right]=0.
\end{align*}
From Definition \ref{def.6.13} we get, for each $\alpha\in \left\{1,\dots, d\right\}$ and $j\in \left\{1,\dots, n\right\}$, that
\[
J_{\mathcal{P}}^{\prime}\left(s_{j-1}+\right)=\mathbf{K}_{\mathcal{P}}^{\prime}\left(s_{j-1}+\right)\mathbf{K}_{\mathcal{P}}\left(1\right)^{-1}H_{\mathcal{P}}=f_{\mathcal{P},j}^{*}\left(1\right)\mathbf{K}_{\mathcal{P}}\left(1\right)^{-1}H_{\mathcal{P}}.
\]
Therefore
\[
X^{h_{\alpha,j}}J_{\mathcal{P}}^{\prime}\left(s_{j-1}+\right)=I_\mathcal{P}\left(\alpha,j\right)+II_\mathcal{P}\left(\alpha,j\right)+III_\mathcal{P}\left(\alpha,j\right),
\]
where 
\begin{align}
I_\mathcal{P}\left(\alpha,j\right) & =\left(X^{h_{\alpha,j}}f_{\mathcal{P},j}^{*}\left(1\right)\right)\mathbf{K}_{\mathcal{P}}\left(1\right)^{-1}H_{\mathcal{P}}\label{equ.8.15}\\
II_\mathcal{P}\left(\alpha,j\right) & =f_{\mathcal{P},j}^{*}\left(1\right)\left(X^{h_{\alpha,j}}\mathbf{K}_{\mathcal{P}}\left(1\right)^{-1}\right)H_{\mathcal{P}}\nonumber \\
III_\mathcal{P}\left(\alpha,j\right) & =f_{\mathcal{P},j}^{*}\left(1\right)\mathbf{K}_{\mathcal{P}}\left(1\right)^{-1}\left(X^{h_{\alpha,j}}H_{\mathcal{P}}\right).\nonumber 
\end{align}
Using Proposition \ref{prop.5.4}, we have
\[div\tilde{X}_\mathcal{P}=\sum_{\alpha=1}^{d}\sum_{j=1}^{n}\left< \left(I_\mathcal{P}+II_\mathcal{P}+III_\mathcal{P}\right)\left(\alpha,j\right), e_\alpha\right>\sqrt{\Delta_j}.\]
Based on the expression above, Proposition \ref{pro.8.6} will be proved as a corollary of Lemma \ref{lem.8.9} to Lemma \ref{lem.8.11}.
In Lemma \ref{lem.8.9} and Lemma \ref{lem.8.12} we show that 
\[\sum_{\alpha=1}^{d}\sum_{j=1}^{n}\left< III_\mathcal{P}\left(\alpha,j\right),e_\alpha\right>\sqrt{\Delta_j}-\left(divX\circ E_1-\sum_{\alpha=1}^{d}\left\langle \tilde{C}A_1\left\langle Z_{\alpha}\right\rangle \tilde{H},e_{\alpha}\right\rangle\right)\to 0\text{ as }\left|\mathcal{P}\right|\to 0.\]
In Lemma \ref{lem.8.10} we show that 
\[\sum_{\alpha=1}^{d}\sum_{j=1}^{n} \left< II_\mathcal{P}\left(\alpha,j\right),e_\alpha\right>\sqrt{\Delta_j}\to 0\text{ as }\left|\mathcal{P}\right|\to 0.\]
In Lemma \ref{lem.8.11} we show that
\[\sum_{\alpha=1}^{d}\sum_{j=1}^{n} \left< I_\mathcal{P}\left(\alpha,j\right),e_\alpha\right>\sqrt{\Delta_j}\to 0\text{ as }\left|\mathcal{P}\right|\to 0.\]

\end{proof}

\begin{lemma} \label{lem.8.9} If $\nabla R\equiv 0$, then for any $q\geq 1$,
\begin{equation}
\underset{\left|\mathcal{P}\right|\to 0}{\lim}\mathbb{E}\left[\left\vert\sum_{\alpha=1}^{d}\sum_{j=1}^{n}\left< III_\mathcal{P}\left(\alpha,j\right),e_\alpha\right>\sqrt{\Delta_j}-\left(divX\circ E_1-\sum_{\alpha=1}^{d}\left\langle \tilde{C}A_1\left\langle Z_{\alpha}\right\rangle \tilde{H},e_{\alpha}\right\rangle\right)\right\vert^q\right]=0.\label{equ.vp}
\end{equation}
\end{lemma}

\begin{proof}Applying Lemma \ref{lem.8.8} to $X^{h_{\alpha,j}}H_{\mathcal{P}}$ gives 
\[
\sum_{\alpha=1}^{d}\sum_{j=1}^{n}\left\langle III_\mathcal{P}\left(\alpha,j\right),e_{\alpha}\right\rangle \sqrt{\Delta_{j}}=IV_\mathcal{P}-V_\mathcal{P},
\]
where 
\[
IV_\mathcal{P}=\sum_{\alpha=1}^{d}\sum_{j=1}^{n}\left\langle f_{\mathcal{P},j}^{*}\left(1\right)\mathbf{K}_{\mathcal{P}}\left(1\right)^{-1}u_\mathcal{P}\left(1\right)^{-1}\nabla_{u_\mathcal{P}\left(1\right)\sqrt{\Delta_{j}}f_{\mathcal{P},j}\left(1\right)e_{\alpha}}X,e_{\alpha}\right\rangle \sqrt{\Delta_{j}}
\]
and 
\[
V_\mathcal{P}=\sum_{\alpha=1}^{d}\sum_{j=1}^{n}\left\langle f_{\mathcal{P},j}^{*}\left(1\right)\mathbf{K}_{\mathcal{P}}\left(1\right)^{-1}\int_{0}^{1}R_{u_{\mathcal{P}}\left(r\right)}\left(\beta_{\mathcal{P}}^{\prime}\left(r+\right),h_{\alpha,j}\left(r\right)\right)drH_{\mathcal{P}},e_{\alpha}\right\rangle \sqrt{\Delta_{j}}.
\]
We first compute $IV_\mathcal{P}$. After viewing $L\left(\cdot\right)=u_\mathcal{P}\left(1\right)^{-1}\nabla_{u_\mathcal{P}\left(1\right)\left(\cdot\right)}X$
as a linear functional on $\mathbb{R}^{d}$ we have
\begin{align}
IV_\mathcal{P} & =\sum_{j=1}^{n}\sum_{\alpha=1}^{d}\left\langle f_{\mathcal{P},j}^{*}\left(1\right)\mathbf{K}_{\mathcal{P}}\left(1\right)^{-1}L\left(f_{\mathcal{P},j}\left(1\right)e_{\alpha}\right),e_{\alpha}\right\rangle \Delta_{j}\nonumber\\
 & =\sum_{j=1}^{n} Tr\left(f_{\mathcal{P},j}^{*}\left(1\right)\mathbf{K}_{\mathcal{P}}\left(1\right)^{-1}Lf_{\mathcal{P},j}\left(1\right)\right)\Delta_{j}\nonumber\\
 & =\sum_{j=1}^{n}Tr\left(\Delta_{j}f_{\mathcal{P},j}\left(1\right)f_{\mathcal{P},j}^{*}\left(1\right)\mathbf{K}_{\mathcal{P}}\left(1\right)^{-1}L\right)\nonumber\\
 & =Tr\left(\sum_{j=1}^{n}\Delta_{j}f_{\mathcal{P},j}\left(1\right)f_{\mathcal{P},j}^{*}\left(1\right)\mathbf{K}_{\mathcal{P}}\left(1\right)^{-1}L\right)\label{tr}\\
 & =Tr\left(L\right)\nonumber\\
 & =divX\circ \Ep,\nonumber
\end{align}
where in Eq. (\ref{tr}) we use identity (\ref{eq:-28}) $\sum_{j=1}^{n}\Delta_{j}f_{\mathcal{P},j}\left(1\right)f_{\mathcal{P},j}^{*}\left(1\right)=\mathbf{K}_{\mathcal{P}}\left(1\right)$ and given $A\in M_{d\times d}$, $Tr\left(A\right):=\sum_{\alpha=1}^{d}\left<Ae_\alpha, e_\alpha\right>$ is the trace of the matrix $A$.

The proof of the lemma will be completed by Lemma \ref{lem.8.12} below which shows term $V_P$ converges to the right side of Eq. (\ref{equ.vp}).
\end{proof}
\begin{lemma}\label{lem.8.12} Let $V_\mathcal{P}$ be defined as in Lemma \ref{lem.8.9} and $\nabla R\equiv 0$, then for any $q\geq 1$,
\begin{equation}
\underset{\left|\mathcal{P}\right|\to 0}{\lim}\mathbb{E}\left[\left\vert V_\mathcal{P}-\sum_{\alpha=1}^{d}\left\langle \tilde{C}A_1\left\langle Z_{\alpha}\right\rangle \tilde{H},e_{\alpha}\right\rangle\right\vert^q\right]=0.\label{e}
\end{equation}
\end{lemma}
\begin{proof}Recall that 
\begin{align*}
V_\mathcal{P}&=\sum_{\alpha=1}^{d}\sum_{j=1}^{n}\left\langle f_{\mathcal{P},j}^{*}\left(1\right)\mathbf{K}_{\mathcal{P}}\left(1\right)^{-1}\int_{0}^{1}R_{u_{\mathcal{P}}\left(r\right)}\left(\beta_{\mathcal{P}}^{\prime}\left(r+\right),h_{\alpha,j}\left(r\right)\right)drH_{\mathcal{P}},e_{\alpha}\right\rangle \sqrt{\Delta_{j}}.
\end{align*}
For each $\alpha\in \left\{1,\dots, d\right\}$ and $j\in \left\{1,\dots, n\right\}$, since $h_{\alpha,j}\left(r\right)=\sqrt{\Delta_j}f_{\mathcal{P},j}\left(r\right)e_\alpha$, we have
\begin{align}
\int_{0}^{1}R_{u_{\mathcal{P}}\left(r\right)}\left(\beta_{\mathcal{P}}^{\prime}\left(r+\right),\frac{1}{\sqrt{\Delta_{j}}}h_{\alpha,j}\left(r\right)\right)dr \nonumber& =\int_{0}^{1}R_{u_{\mathcal{P}}\left(r\right)}\left(\beta_{\mathcal{P}}^{\prime}\left(r+\right),f_{\mathcal{P},j}\left(r\right)e_{\alpha}\right)dr\nonumber\\
 & =\int_{0}^{1}R_{u_{\mathcal{P}}\left(\underline{r}\right)}\left(\beta_{\mathcal{P}}^{\prime}\left(r+\right),f_{\mathcal{P},j}\left(\underline{r}\right)e_{\alpha}\right)dr+e_{0}\label{eq12}
\end{align}
where $e_{0}:=e_{0,1}+e_{0,2}$ 
\[
e_{0,1}=\int_{0}^{1}R_{u_{\mathcal{P}}\left(r\right)}\left(\beta_{\mathcal{P}}^{\prime}\left(r+\right),f_{\mathcal{P},j}\left(r\right)e_{\alpha}\right)dr-\int_{0}^{1}R_{u_{\mathcal{P}}\left(\underline{r}\right)}\left(\beta_{\mathcal{P}}^{\prime}\left(r+\right),f_{\mathcal{P},j}\left(r\right)e_{\alpha}\right)dr
\]
and 
\[
e_{0,2}=\int_{0}^{1}R_{u_{\mathcal{P}}\left(\underline{r}\right)}\left(\beta_{\mathcal{P}}^{\prime}\left(r+\right),f_{\mathcal{P},j}\left(r\right)e_{\alpha}\right)dr-\int_{0}^{1}R_{u_{\mathcal{P}}\left(\underline{r}\right)}\left(\beta_{\mathcal{P}}^{\prime}\left(r+\right),f_{\mathcal{P},j}\left(\underline{r}\right)e_{\alpha}\right)dr.
\]
Since $\nabla R\equiv 0$, $dR_{u_s}=\nabla_{db_s}R\equiv 0$. So $R_{u_s}\equiv R_{u_0}$ is independent
of $u$, therefore $e_{0,1}=0$.

As for $e_{0,2}$, since
\[
\left|e_{0,2}\right|^{q}\leq N\underset{r\in\left[0,1\right]}{\sup}\left|\beta_{\mathcal{P}}^{\prime}\left(r+\right)\right|^{q}\underset{r\in\left[0,1\right],j\in\left\{ 1,\cdots,n\right\} }{\sup}\left|f_{\mathcal{P},j}\left(r\right)-f_{\mathcal{P},j}\left(\underline{r}\right)\right|^{q},
\]
using $\left(\ref{equ.7.3}\right)$ we have 
\[
\left|e_{0,2}\right|^{q}\leq CK_{\gamma}^{3q}\left\vert \mathcal{P}\right\vert ^{q\left(3\gamma-1\right)}e^{2qN\sum_{k=1}^{n}\left\vert \Delta_{k}\beta\right\vert ^{2}},
\]
and from which it follows
\begin{equation}
\mathbb{E}\left[\left|e_{0,2}\right|^{q}\right]\leq C\left\vert \mathcal{P}\right\vert ^{q\left(3\gamma-1\right)}\text{   }\forall n\geq 5qN.\label{eq:-6}
\end{equation}
Picking $\gamma>\frac{1}{3}$ such that $q\left(3\gamma-1\right)>0$ for any $q\geq1$, so $\mathbb{E}\left[\left|e_{0,2}\right|^{q}\right]\to 0$ as $\left\vert \mathcal{P}\right\vert \to 0$.

Next we analyze 
\begin{align*}
\int_{0}^{1}R_{u_{\mathcal{P}}\left(\underline{r}\right)}\left(\beta_{\mathcal{P}}^{\prime}\left(r+\right),f_{\mathcal{P},j}\left(\underline{r}\right)e_{\alpha}\right)dr & =\sum_{k=1}^{n}R_{u_{\mathcal{P}}\left(s_{k-1}\right)}\left(\Delta_{k}\beta,f_{\mathcal{P},j}\left(s_{k-1}\right)e_{\alpha}\right) =\int_{0}^{1}g_{1}\left(s\right)d\beta_{s},
\end{align*}
where 
\[
g_{1}\left(s\right)=\sum_{k=1}^{n}R_{u_{\mathcal{P}}\left(s_{k-1}\right)}\left(\cdot,f_{\mathcal{P},j}\left(s_{k-1}\right)e_{\alpha}\right)1_{[s_{k-1},s_{k})}\left(s\right).
\]
Define 
\begin{align*}
g_{2}\left(s\right) & =\sum_{k=1}^{n}R_{\tilde{u}_{s_{k-1}}}\left(\cdot,f_{\mathcal{P},j}\left(s_{k-1}\right)e_{\alpha}\right)1_{[s_{k-1},s_{k})}\left(s\right)\\
g_{3}\left(s\right) & =\sum_{k=j+1}^{n}R_{\tilde{u}_{s_{k-1}}}\left(\cdot,\tilde{T}_{s_{k-1}}\tilde{T}_{s_{j}}^{-1}e_{\alpha}\right)1_{[s_{k-1},s_{k})}\left(s\right)\\
g_{4}\left(s\right) & =R_{\tilde{u}_{\underline{s}}}\left(\cdot,\tilde{T}_s\tilde{T}_{s_{j}}^{-1}e_{\alpha}\right)1_{\left[s_{j},1\right]}\left(s\right)\\
g_{5}\left(s\right) & =R_{\tilde{u}_s}\left(\cdot,\tilde{T}_s\tilde{T}_{s_{j}}^{-1}e_{\alpha}\right)1_{\left[s_{j},1\right]}\left(s\right).
\end{align*}
For each $i=1,2,3,4$, 
denote $e_{\mathcal{P},i}\left(r\right)=\int_{0}^{r}g_{i}\left(s\right)d\beta_{s}-\int_{0}^{r}g_{i+1}\left(s\right)d\beta_{s}$, then\begin{align}
\int_{0}^{1}R_{u_{\mathcal{P}}\left(\underline{r}\right)}&\left(\beta_{\mathcal{P}}^{\prime}\left(r+\right),f_{\mathcal{P},j}\left(\underline{r}\right)e_{\alpha}\right)dr-\int_{s_{j}}^{1}R_{\tilde{u}_r}\left(d\beta_{r},\tilde{T}_r\tilde{T}_{s_{j}}^{-1}e_{\alpha}\right)=\sum_{i=1}^{4}e_{\mathcal{P},i}\left(1\right).\label{eqe}
\end{align}
Notice that $\left\{g_i\right\}_{i=1}^5$ are all adapted, so based on the same computation as in Lemma \ref{lem.7.12} (mainly Burkholder-Davis-Gundy inequality), we can show for each $i\in \left\{1,2,3,4\right\}$ and for any $q\geq 1$,
\begin{equation}
\lim_{\left|\mathcal{P}\right|\to 0}\mathbb{E}_\nu\left[\left\vert e_{\mathcal{P},i}\left(1\right)\right\vert^q\right]=0\label{error}
\end{equation}
Using Eq.(\ref{eq12}), (\ref{eq:-6}), (\ref{eqe}) and (\ref{error}) we have for any $q\geq 1$,
\begin{equation}
\lim_{\left|\mathcal{P}\right|\to 0}\mathbb{E}_\nu\left[\left\vert V_\mathcal{P}-\tilde{V}_\mathcal{P}\right\vert^q\right]=0\label{e1}
\end{equation}
where
\[\tilde{V}_\mathcal{P}=\sum_{\alpha=1}^{d}\sum_{j=1}^{n}\left\langle \left(\tilde{T}_{s_{j}}^{-1}\right)^{\ast}\tilde{T}_1^{\ast}\tilde{\mathbf{K}}_1^{-1}\int_{s_{j}}^{1}R_{\tilde{u}_r}\left(d\beta_{r},\tilde{T}_r\tilde{T}_{s_{j}}^{-1}e_{\alpha}\right)\tilde{H},e_{\alpha}\right\rangle \Delta_{j}.\]
For each $j\in \left\{1,\dots,n\right\}$, denote by $\mathcal{B}_j$ the bilinear form on $\mathbb{R}^d$:
\[\mathcal{B}_j(u,v)=\left\langle \left(\tilde{T}_{s_{j}}^{-1}\right)^{\ast}\tilde{T}_1^{\ast}\tilde{\mathbf{K}}_1^{-1}\int_{s_{j}}^{1}R_{\tilde{u}_r}\left(d\beta_{r},\tilde{T}_r\tilde{T}_{s_{j}}^{-1}u\right)\tilde{H},v\right\rangle \Delta_{j}\]
Then 
\begin{align}
\tilde{V}_\mathcal{P}=\sum_{j=1}^{n}\sum_{\alpha=1}^{d}\mathcal{B}(e_\alpha,e_\alpha)&=\sum_{j=1}^{n}\sum_{\alpha=1}^{d}\mathcal{B}(\tilde{T}_{s_{j}}e_\alpha,\left(\tilde{T}_{s_{j}}^{-1}\right)^{\ast}e_\alpha)\nonumber\\&=\sum_{\alpha=1}^{d}\int_{0}^{1}\left\langle \tilde{T}_1^{\ast}K_1^{-1}\int_{\underline{s}}^{1}R_{\tilde{u}_r}\left(d\beta_{r},\tilde{T}_re_{\alpha}\right)\tilde{H},\tilde{T}_{\underline{s}}^{-1}\left(\tilde{T}_{\underline{s}}^{-1}\right)^{\ast}e_{\alpha}\right\rangle ds.\label{hi}
\end{align}
Define 
\begin{equation*}
\hat{V}_\mathcal{P}:=\sum_{\alpha=1}^{d}\int_{0}^{1}\left\langle \tilde{T}_1^{\ast}K_1^{-1}\int_{s}^{1}R_{\tilde{u}_r}\left(d\beta_{r},\tilde{T}_re_{\alpha}\right)\tilde{H},\tilde{T}_{s}^{-1}\left(\tilde{T}_{s}^{-1}\right)^{\ast}e_{\alpha}\right\rangle ds
\end{equation*}
Then we are about to show for any $q\geq 1$, 
\begin{equation}
\underset{\left\vert\mathcal{P}\to 0\right\vert}{\lim}\mathbb{E}\left[\left\vert\tilde{V}_\mathcal{P}-\hat{V}_\mathcal{P}\right\vert^q\right]=0.\label{e2}
\end{equation}
Using Eq.(\ref{hi}) we know
\begin{equation*}
\left\vert\tilde{V}_\mathcal{P}-\hat{V}_\mathcal{P}\right\vert\leq \sum_{\alpha=1}^{d}\int_{0}^{1}\left(I_\mathcal{P}\left(s\right)+II_\mathcal{P}\left(s\right)\right)ds,
\end{equation*}
where 
\[
I_\mathcal{P}\left(s\right)=\left\langle \tilde{T}_1^{\ast}K_1^{-1}\int_{\underline{s}}^{s}R_{\tilde{u}_r}\left(d\beta_{r},\tilde{T}_re_{\alpha}\right)\tilde{H},\tilde{T}_s^{-1}\left(\tilde{T}_s^{-1}\right)^{\ast}e_{\alpha}\right\rangle 
\]
and 
\begin{align}
& II_\mathcal{P}\left(s\right)= \left\langle \tilde{T}_1^{\ast}K_1^{-1}\int_{s}^{1}R_{\tilde{u}_r}\left(d\beta_{r},\tilde{T}_re_{\alpha}\right)\tilde{H},\left({\tilde{T}_{\underline{s}}}^{-1}\left({\tilde{T}_{\underline{s}}}^{-1}\right)^{\ast}-\tilde{T}_s^{-1}\left(\tilde{T}_s^{-1}\right)^{\ast}\right)e_{\alpha}\right\rangle. 
\end{align}
Since
\[
\left\vert I_\mathcal{P}\left(s\right)\right\vert ^{q}\leq C\left\vert \int_{\underline{s}}^{s}R_{\tilde{u}_r}\left(d\beta_{r},\tilde{T}_re_{\alpha}\right)\right\vert ^{q},
\]
\[\left\vert II_\mathcal{P}\left(s\right)\right\vert ^{q}\leq C\left\vert \mathcal{P}\right\vert ^{q}\left\vert \int_{\underline{s}}^{s}R_{\tilde{u}_r}\left(d\beta_{r},\tilde{T}_re_{\alpha}\right)\right\vert ^{q}\]
and by Burkholder-Davies-Gundy inequality, $\mathbb{E}\left[\left\vert \int_{\underline{s}}^{s}R_{\tilde{u}_r}\left(d\beta_{r},\tilde{T}_re_{\alpha}\right)\right\vert ^{q}\right]\leq C\left\vert \mathcal{P}\right\vert ^{\frac{q}{2}}$, we obtain
\begin{align*}
\mathbb{E}\left[\left\vert \tilde{V}_\mathcal{P}-\hat{V}_\mathcal{P}\right\vert ^{q}\right]& \leq C\sum_{\alpha=1}^{d}\sum_{j=1}^{n}\int_{s_{j-1}}^{s_{j}}\mathbb{E}\left[\left\vert I_\mathcal{P}\left(s\right)\right\vert ^{q}+\left\vert II_\mathcal{P}\left(s\right)\right\vert ^{q}\right]\\
 & \leq C\sum_{\alpha=1}^{d}\sum_{j=1}^{n}\int_{s_{j-1}}^{s_{j}}\left(C+\left\vert \mathcal{P}\right\vert ^{q}\right)\mathbb{E}\left[\left\vert \int_{\underline{s}}^{s}R_{\tilde{u}_r}\left(d\beta_{r},\tilde{T}_re_{\alpha}\right)\right\vert ^{q}\right]\\
 & =C\left\vert \mathcal{P}\right\vert ^{\frac{q}{2}}
\end{align*}
and from which Eq. (\ref{e2}) follows.

Then we show a change of integration order, i.e. \begin{equation}
\hat{V}_\mathcal{P}=\sum_{\alpha=1}^{d}\mathcal{B}_1(e_\alpha,e_\alpha)\label{equ.8.16}
\end{equation}, where $\mathcal{B}_1(u,v):=\left\langle \tilde{T}_1^{\ast}K_1^{-1}\int_{0}^{1}R_{\tilde{u}_r}\left(d\beta_{r},\tilde{T}_ru\right)\tilde{H},\int_{0}^{r}\tilde{T}_s^{-1}\left(\tilde{T}_s^{-1}\right)^{\ast}vds\right\rangle.$ Define
\[
f\left(t\right)=\sum_{\alpha=1}^{d}\int_{0}^{t}\left\langle \tilde{T}_1^{\ast}K_1^{-1}\int_{s}^{t}R_{\tilde{u}_r}\left(d\beta_{r},\tilde{T}_re_{\alpha}\right)\tilde{H},\tilde{T}_s^{-1}\left(\tilde{T}_s^{-1}\right)^{\ast}e_{\alpha}\right\rangle ds
\]
and
\[
g\left(t\right)=\sum_{\alpha=1}^{d}\int_{0}^{r}\left\langle \tilde{T}_1^{\ast}K_1^{-1}\int_{0}^{t}R_{\tilde{u}_r}\left(d\beta_{r},\tilde{T}_re_{\alpha}\right)\tilde{H},\int_{0}^{r}\tilde{T}_s^{-1}\left(\tilde{T}_s^{-1}\right)^{\ast}dse_{\alpha}\right\rangle. 
\]
Then 
\[
df=\sum_{\alpha=1}^{d}\left\langle \tilde{T}_1^{\ast}K_1^{-1}R_{\tilde{u}_t}\left(d\beta_{t},\tilde{T}_te_{\alpha}\right)\tilde{H},\int_{0}^{t}\tilde{T}_s^{-1}\left(\tilde{T}_s^{-1}\right)^{\ast}dse_{\alpha}\right\rangle. 
\]
Since 
\[
dg=\sum_{\alpha=1}^{d}\left\langle \tilde{T}_1^{\ast}K_1^{-1}R_{\tilde{u}_t}\left(d\beta_{t},\tilde{T}_te_{\alpha}\right)\tilde{H},\int_{0}^{t}\tilde{T}_s^{-1}\left(\tilde{T}_s^{-1}\right)^{\ast}dse_{\alpha}\right\rangle =df
\]
and $g\left(0\right)=0=f(0)$,  Eq. (\ref{equ.8.16}) is proved by observing that $\hat{V}_\mathcal{P}=f_1=g_1=\sum_{\alpha=1}^{d}\mathcal{B}_1(e_\alpha,e_\alpha)$.

Lastly, notice that 
\begin{align}
\sum_{\alpha=1}^{d}\mathcal{B}_1(e_\alpha,e_\alpha)=Tr\left(\mathcal{B}_1\right)=\sum_{\alpha=1}^{d}\mathcal{B}_1\left(\int_{0}^{r}\tilde{T}_s^{-1}\left(\tilde{T}_s^{-1}\right)^{\ast}dse_{\alpha},\left[\left(\int_{0}^{r}\tilde{T}_s^{-1}\left(\tilde{T}_s^{-1}\right)^{\ast}ds\right)^{-1}\right]^{\ast}e_{\alpha}\right)\nonumber
\end{align}
and $\tilde{T}_r\int_{0}^{r}\tilde{T}_s^{-1}\left(\tilde{T}_s^{-1}\right)^{\ast}dse_{\alpha}=Z_{\alpha}\left(r\right)$,
we combine Eq. (\ref{e1}), (\ref{e2}) and (\ref{equ.8.16}) to prove Eq.(\ref{e}). 
\end{proof}

\begin{lemma} \label{lem.8.10}If $\nabla R\equiv 0$, then for any $q\geq 1$, 
	\[\underset{\left\vert\mathcal{P}\right\vert\to 0}{\lim}\mathbb{E}\left[\left\vert\sum_{\alpha=1}^{d}\sum_{j=1}^{n} \left< I_\mathcal{P}\left(\alpha,j\right),e_\alpha\right>\sqrt{\Delta_j}\right\vert^q\right]=0.\]
\end{lemma}

\begin{proof} Define $\tilde{g}_{j}\left(s\right):=X^{h_{\alpha,j}}f_{\mathcal{P},j}\left(s\right)$
and $g_{j}\left(s\right):=\tilde{g}_{j}\left(s\right)-\tilde{g}_{j}\left(\underline{s}\right)$.
Then we know that $g_{j}\left(s\right)$ satisfies the following ODE: $\text{ for }k=j,\cdots,n$
\[
\begin{cases}
g_{j}^{\prime\prime}\left(s\right)=A_{\mathcal{P},k}\left(s\right)g_{j}\left(s\right)+\dot{A}_{\mathcal{P},k}\left(s\right)\left(f_{\mathcal{P},j}\left(s\right)-f_{\mathcal{P},j}\left(\underline{s}\right)\right) & s\in\left[s_{k-1},s_{k}\right]\\
g_{j}\left(\underline{s}\right)=0\\
g_{j}^{\prime}\left(\underline{s}\right)=0 
\end{cases}
\]
where $\dot{A}_{\mathcal{P},k}\left(s\right)=\frac{d}{dt}|_{0}\left(R_{u_{\mathcal{P}}\left(t,s\right)}\left(\beta_{\mathcal{P}}^{\prime}\left(t,s\right),\cdot\right)\beta_{\mathcal{P}}^{\prime}\left(t,s\right)\right)$, $\beta\left(t,\cdot\right):=\tilde{\phi}^{-1}\left(E\left(tX^{h_{\alpha,j}}\right)\right)$ is the stochastic anti-rolling of the approximate flow $E\left(tX^{h_{\alpha,j}}\right)$ $(\text{See Corollary 4.6 in \cite{Driver1999}})$,  $\beta_{\mathcal{P}}\left(t,\cdot\right)$ is the piecewise linear approximation of $\beta\left(t,\cdot\right)$ and $u_\mathcal{P}\left(t,\cdot\right)=\eta\circ \beta_\mathcal{P}\left(t,\cdot\right)$ is the horizontal lift of $\beta_\mathcal{P}\left(t,\cdot\right)$.

Since $\nabla R\equiv 0$, $A_{\mathcal{P},k}(s)$ is a constant operator for $s\in\left[s_{k-1},s_{k}\right]$, therefore by DuHammel's principle, 
\[
g_{j}\left(s\right)=\int_{s_{k-1}}^{s}S_{\mathcal{P},k}\left(s-r\right)\dot{A}_{\mathcal{P},k}\left(f_{\mathcal{P},j}\left(r\right)-f_{\mathcal{P},j}\left(s_{k-1}\right)\right)dr.
\]
Using Eq. \ref{lem.6.3} and \ref{equ.7.3} we obtain the following estimate,
\begin{align}
\left\vert g_j\left(s\right)\right\vert&\leq \int_{s_{k-1}}^{s}\left\vert S_{\mathcal{P},k}\left(s-r\right)\right\vert \left\vert \dot{A}_{\mathcal{P},k}\right\vert \left\vert f_{\mathcal{P},j}\left(r\right)-f_{\mathcal{P},j}\left(s_{k-1}\right)\right\vert dr\nonumber\\&\leq C\underset{1\leq k\leq n}{\sup}\left\vert \dot{A}_{\mathcal{P},k}\right\vert \left\vert \mathcal{P}\right\vert^{2\gamma+2}K_{\gamma}^{2q}e^{3N\sum_{k=1}^{n}\left\vert\Delta_k\beta\right\vert^2}.\label{eq3} 
\end{align}
Therefore
\begin{align}
\left\vert \tilde{g}_{j}\left(1\right)\right\vert  \leq\sum_{k=j}^{n}\left\vert g_{j}\left(s_{k}\right)\right\vert\leq C\underset{1\leq k\leq n}{\sup}\left\vert \dot{A}_{\mathcal{P},k}\right\vert \left\vert \mathcal{P}\right\vert^{2\gamma}K_{\gamma}^{2q}e^{3N\sum_{k=1}^{n}\left\vert\Delta_k\beta\right\vert^2}.\label{equ.8.17}
\end{align}
Then we analyze $\underset{k\in\left\{ 1,\dots,n\right\}}{\sup}\left\vert \dot{A}_{\mathcal{P},k}\right\vert$. Since $\nabla R\equiv 0$,
\begin{align*}
\dot{A}_{\mathcal{P},k}& =2R_{u_{\mathcal{P}}\left(s\right)}\left(\frac{d}{dt}|_{0}\beta_{\mathcal{P}}^{\prime}\left(t,s\right),\cdot\right)\beta_{\mathcal{P}}^{\prime}\left(s\right).
\end{align*}
Notice that $\beta_{\mathcal{P}}^{\prime}\left(t,s\right)=u_{\mathcal{P}}^{-1}\left(s\right)\sigma_{\mathcal{P}}^{\prime}\left(t,s\right)$, where $\sigma_{\mathcal{P}}(t,\cdot)=\phi\circ \beta_{\mathcal{P}}(t,\cdot)$ is the rolling of $\beta_{\mathcal{P}}(t,\cdot)$, so we can use Lemma \ref{lem.8.8} to get
\begin{equation}
X^{h_{\alpha,j}}\beta_{\mathcal{P}}^{\prime}\left(s_{k-1}+\right)=\frac{\delta_{k}^{j}e_{\alpha}}{\sqrt{\Delta_{j}}}-\int_{0}^{s_{k-1}}R_{u_{\mathcal{P}}\left(\tau\right)}\left(\beta_{\mathcal{P}}^{\prime}\left(\tau+\right),h_{\alpha,j}\left(\tau\right)\right)d\tau \beta_{\mathcal{P}}^{\prime}\left(s_{k-1}+\right).\label{equ.8.18}
\end{equation}
Therefore
\begin{align*}
\left\vert \dot{A_{\mathcal{P}}^{k}}\left(r\right)\right\vert  & \leq N\left\vert X^{h_{\alpha,j}}\beta_{\mathcal{P}}^{\prime}\left(s_{k-1}+\right)\right\vert \left\vert \beta_{\mathcal{P}}^{\prime}\left(s_{k-1}\right)\right\vert \\
 & \leq N\left(\frac{1}{\sqrt{\left\vert \mathcal{P}\right\vert }}+N\underset{j,s}{\sup}\left\vert h_{\alpha,j}\left(s\right)\right\vert \underset{s\in\left[0,1\right]}{\sup}\left\vert \beta_{\mathcal{P}}^{\prime}\left(s\right)\right\vert ^{2}\right)\left\vert \beta_{\mathcal{P}}^{\prime}\left(s_{k-1}\right)\right\vert \\
 & \leq N\left(\frac{1}{\sqrt{\left\vert \mathcal{P}\right\vert }}+Nf\left(K_{\gamma}\right)\sqrt{\left\vert \mathcal{P}\right\vert }\left\vert \mathcal{P}\right\vert ^{2\left(\gamma-1\right)}\right)K_{\gamma}\left\vert \mathcal{P}\right\vert ^{\gamma-1}\\
 & \leq f\left(K_{\gamma}\right)\left\vert \mathcal{P}\right\vert ^{3\gamma-\frac{5}{2}}
\end{align*}
where $f\left(K_{\gamma}\right)$ is some random variable in $L^{1}\left(W_o\left(M\right)\right)$,
so
\begin{equation}
\left\vert \tilde{g}_{j}\left(1\right)\right\vert \leq Cf\left(K_{\gamma}\right)\left\vert \mathcal{P}\right\vert ^{5\gamma-\frac{3}{2}}.\label{equ.8.19}
\end{equation}
From above one can see 
\begin{align*}
\sum_{\alpha,j=1,1}^{d,n}\left\langle I,e_{\alpha}\right\rangle \sqrt{\Delta_{j}} & =\sum_{\alpha,j=1,1}^{d,n}\left\langle \left(X^{h_{\alpha,j}}T_{j}^{\ast}\right)\mathbf{K}_{\mathcal{P}}^{-1}\left(1\right)H_{\mathcal{P}},e_{\alpha}\right\rangle \sqrt{\Delta_{j}}\\
 & =\sum_{\alpha=1}^{d}\left\langle \sum_{j=1}^{n}\left(\tilde{g}_{j}^{\ast}\left(1\right)\sqrt{\left\vert \mathcal{P}\right\vert }\right)\mathbf{K}_{\mathcal{P}}\left(1\right)^{-1}H_{\mathcal{P}},e_{\alpha}\right\rangle. 
\end{align*}
From $\left(\ref{equ.8.19}\right)$ we know that $\sum_{j=1}^{n}\left(\tilde{g}_{j}^{\ast}\left(1\right)\sqrt{\left\vert \mathcal{P}\right\vert }\right)\rightarrow0$
$\text{ in }L^{\infty-}\left(W\right)$, also notice that for any $q\geq 1$,
\[\lim_{\left|\mathcal{P}\right|\to0}\mathbb{E}_\nu\left[\left\vert \mathbf{K}_{\mathcal{P}}\left(1\right)^{-1}H_{\mathcal{P}}-\mathbf{K}\left(1\right)^{-1}\tilde{H}\right\vert^q\right]=0.\] 
So
\[\lim_{\left|\mathcal{P}\right|\to0}\mathbb{E}_\nu\left[\left\vert \sum_{\alpha=1}^{d}\left\langle \sum_{j=1}^{n}\left(\tilde{g}_{j}^{\ast}\left(1\right)\sqrt{\left\vert \mathcal{P}\right\vert }\right)\mathbf{K}_{\mathcal{P}}\left(1\right)^{-1}H_{\mathcal{P}},e_{\alpha}\right\rangle\right\vert^q\right]=0\text{ }\forall q\geq 1.\] 
\end{proof}

\begin{lemma} \label{lem.8.11}If $\nabla R\equiv 0$,
then for any $q\geq 1$,
	\[\underset{\left\vert\mathcal{P}\right\vert\to 0}{\lim}\mathbb{E}\left[\left\vert\sum_{\alpha=1}^{d}\sum_{j=1}^{n} \left< II_\mathcal{P}\left(\alpha,j\right),e_\alpha\right>\sqrt{\Delta_j}\to 0\text{ as }\left\vert \mathcal{P}\right\vert\right\vert^q\right]=0.\]
\end{lemma}
\begin{proof}Since
\[
X^{h_{\alpha,j}}\left(\mathbf{K}_{\mathcal{P}}\left(1\right)^{-1}\right)=-\mathbf{K}_{\mathcal{P}}\left(1\right)^{-1}X^{h_{\alpha,j}}\left(\mathbf{K}_{\mathcal{P}}\left(1\right)\right)\mathbf{K}_{\mathcal{P}}\left(1\right)^{-1},
\]
so 
\[
\left\vert X^{h_{\alpha,j}}\left(\mathbf{K}_{\mathcal{P}}\left(1\right)^{-1}\right)\right\vert \leq\left\vert X^{h_{\alpha,j}}\left(\mathbf{K}_{\mathcal{P}}\left(1\right)\right)\right\vert. 
\]
Then let $\tilde{g}_{j}\left(s\right):=X^{h_{\alpha,j}}\left(\mathbf{K}_{\mathcal{P}}\left(s\right)\right)$
and this lemma follows from a Lemma \ref{lem.8.10}-type argument.
\end{proof}

\section{Proof of Theorem \ref{thm2}\label{cha.7}}
First we collect a list of supplementary results.
\begin{lemma}\label{inf}
	For any $f\in \mathcal{FC}^1_b$, $\tilde{X}^{tr,\nu}f\in L^{\infty-}\left(W_o(M),\nu\right)$.
\end{lemma}
\begin{proof}
	See Lemma 4.24 in IVP.
\end{proof}

\begin{lemma} \label{lem7.4}For any $f\in\mathcal{FC}_{b}^{1}$ and $q\geq 1$, there exists a constant $M=M(q)$ such that for all partition $\mathcal{P}$ with $\left\vert\mathcal{P}\right\vert<\frac{1}{M}$,  $\tilde{X}_{\mathcal{P}}^{tr,\nu_{\mathcal{P}}^{1}}f\in L^{q}\left(H_{\mathcal{P}}\left(M\right),\nu_{\mathcal{P}}^{1}\right).$

\end{lemma}

\begin{proof}
From Theorem \ref{the.8.4} we know that 
\[\underset{\left\vert\mathcal{P}\right\vert\to 0}{\lim}\mathbb{E}\left[\left\vert \tilde{X}_{\mathcal{P}}^{tr,\nu_{\mathcal{P}}^{1}}f\left(\phi\left(\beta_\mathcal{P}\right)\right)-\tilde{X}^{tr,\nu}f\right\vert^q\right]=0.\]
%where $\tilde{f}\left(\sigma\right)=f\left(\tilde{u}\right)\in \mathcal{FC}^1_b$.
By Lemma \ref{inf}, $\tilde{X}^{tr,\nu}f\in L^{\infty-}\left(W_o\left(M\right)\right)$. Therefore Lemma \ref{lem7.4} follows from triangle inequality and the fact that the law of $\phi\left(\beta_\mathcal{P}\right)$ under $\nu $ is $\nu_{\mathcal{P}}^1$.
\end{proof}
\begin{notation}
Denote by $g$ any one of $\left\{g_i\right\}^d_{i=0}$ as in Theorem \ref{thm.1.11} and $\left\{g^{\left(m\right)}\right\}_m\subset C^{\infty}_0\left(M\right)$ be the approximate sequence in $L^{\frac{d}{d-1}}\left(M\right)$ as defined in Remark \ref{rem3.1.2}. 
\end{notation}
\begin{lemma}\label{lem7.6}
Define $\tilde{g}\left(\sigma\right)=g\left(\sigma\left(1\right)\right)$ and $\tilde{g}^{\left(m\right)}\left(\sigma\right)=g^{\left(m\right)}\left(\sigma\left(1\right)\right)$, then for any $f\in\mathcal{FC}^{1}_b$,
 \begin{equation}
  \int_{W_o\left(M\right)}\left|\tilde{g}\cdot\left(\tilde{X}^{tr,\nu}f\right)\right|\left(\sigma\right)d\nu\left(\sigma\right)<\infty\label{Eq:0}
 \end{equation}
 and 
 \begin{equation}
 \lim_{m\to\infty}\int_{W_o\left(M\right)}\tilde{g}^{\left(m\right)}\left(\sigma\right)\left(\tilde{X}^{tr,\nu}f\right)\left(\sigma\right)d\nu\left(\sigma\right)=\int_{W_o\left(M\right)}\tilde{g}\left(\sigma\right)\left(\tilde{X}^{tr,\nu}f\right)\left(\sigma\right)d\nu\left(\sigma\right).
 \label{Eq:1}
 \end{equation} 
\end{lemma}
\begin{proof}
Since $\nu\left\{\sigma:\sigma\left(1\right)=x\right\}=0$, so $\tilde{g}$ is $\nu-a.s.$ well-defined. In particular, for any $p>0$, 
\begin{align}
 \int_{W_o\left(M\right)}\left|\tilde{g}\left(\sigma\right)\right|^pd\nu\left(\sigma\right)=\int_M\left|g\left(x\right)\right|^pp_1\left(0,x\right)d\lambda\left(x\right),\label{6}
\end{align}
where $\lambda$ is the volume measure on $M$.

Since $g$ has compact support and $p_1\left(0,\cdot\right)\in C^\infty\left(M\right)$, 
\begin{align}
\int_M\left|g\left(x\right)\right|^pp_1\left(0,x\right)d\lambda\left(x\right)\leq C\left\Vert g\right\Vert^p_{L^p\left(M\right)}.\label{5}
\end{align}
Combining Eq.(\ref{6}) and (\ref{5}) and letting $p=\frac{d}{d-1}$ we get
\begin{equation}
\tilde{g}\in  L^{1+\frac{1}{d-1}}\left(W_o\left(M\right)\right).\label{4}
\end{equation}
Since $\tilde{X}^{tr,\nu}f\in L^{\infty-}\left(W_o\left(M\right)\right)$ by Lemma \ref{inf}, using Holder's inequality we prove Eq.(\ref{Eq:0}).

Since $\cup_m supp\left(g^{(m)}\right)$ is contained in a compact set, Eq.(\ref{Eq:1}) can be proved similarly with $g$ replaced by $g^{(m)}-g$.
\end{proof}
\begin{lemma} \label{lem7.7} Define $\tilde{g}:H_\mathcal{P}\left(M\right)\to \mathbb{R}$ to be $\tilde{g}\left(\sigma\right)=g\left(\sigma\left(1\right)\right)$, then $\tilde{g}\in L^{\frac{d}{d-1}}\left(H_{\mathcal{P}}\left(M\right),\nu_{\mathcal{P}}^{1}\right)$. \end{lemma}

\begin{proof}
	Set 
	\begin{equation*}
	V_\mathcal{P}:=\left\{\sigma\in H_\mathcal{P}\left(M\right):\Ep^{\mathcal{P}}\left(\sigma\right)=x \right\}.
	\end{equation*}
	Applying the co-area formula Eq.(\ref{co}) to $f(y)=1_{\left\{y=x\right\}}$, we have
	\begin{equation*}
	\nu_{\mathcal{P}}^1\left(V_\mathcal{P}\right)=\int_{H_{\mathcal{P}}\left(M\right)}f\left(\sigma\left(1\right)\right)d\nu_{\mathcal{P}}^{1}\left(\sigma\right)=\int_{M}f\left(y\right)h_{\mathcal{P}}\left(y\right)dy=0.
	\end{equation*}
	So $\tilde{g}$ is $\nu_{\mathcal{P}}^1-a.s.$ well-defined. Then applying the co-area formula Eq.(\ref{co}) again to $\left|\tilde{g}\right|^{\frac{d}{d-1}},$
we have 
\begin{equation}
\int_{H_{\mathcal{P}}\left(M\right)}\left|\tilde{g}\left(\sigma\right)\right|^{\frac{d}{d-1}}d\nu_{\mathcal{P}}^{1}\left(\sigma\right)=\int_{M}\left|g\left(x\right)\right|^{\frac{d}{d-1}}h_{\mathcal{P}}\left(x\right)dx,\label{eq4}
\end{equation}
where $h_{\mathcal{P}}\left(x\right)\in C\left(M\right)$
is defined in Theorem \ref{thm3.2.1} with $f\equiv 1$. Since $g$ has compact support,
\begin{equation*}
\int_{M}\left|g\left(x\right)\right|^{\frac{d}{d-1}}h_{\mathcal{P}}\left(x\right)dx\leq C\int_{M}\left|g\left(x\right)\right|^{\frac{d}{d-1}}dx.
\end{equation*}
Therefore 
$\tilde{g}\in L^{\frac{d}{d-1}}\left(H_{\mathcal{P}}\left(M\right),\nu_{\mathcal{P}}^{1}\right)$.
\end{proof}
\begin{lemma}\label{the.1.3.5} 
Define $\tilde{g}\left(\sigma\right)=g\left(\sigma\left(1\right)\right)$ and $\tilde{g}^{\left(m\right)}\left(\sigma\right)=g^{\left(m\right)}\left(\sigma\left(1\right)\right)$, then there exists a constant $M$ such that for any $f\in\mathcal{FC}_{b}^{1}$ and $\mathcal{P}$ with $\left\vert\mathcal{P}\right\vert<\frac{1}{M}$,
 \begin{equation}
  \int_{H_\mathcal{P}\left(M\right)}\left|\tilde{g}\cdot\left(\tilde{X}^{tr,\nu_{\mathcal{P}}^1}f\right)\right|\left(\sigma\right)d\nu_{\mathcal{P}}^1\left(\sigma\right)<\infty\label{Eq:3}
 \end{equation}
 and 
 \begin{equation}
  \lim_{m\to\infty}\int_{H_\mathcal{P}\left(M\right)}\tilde{g}^{\left(m\right)}\left(\sigma\right)\left(\tilde{X}^{tr,\nu_{\mathcal{P}}^1}f\right)\left(\sigma\right)d\nu_{\mathcal{P}}^1\left(\sigma\right)=\int_{H_\mathcal{P}\left(M\right)}\tilde{g}\left(\sigma\right)\left(\tilde{X}^{tr,\nu_{\mathcal{P}}^1}f\right)\left(\sigma\right)d\nu_{\mathcal{P}}^1\left(\sigma\right).
 \label{Eq:2} 
 \end{equation}
\end{lemma}
\begin{proof}
Using Lemma \ref{lem7.4}, Lemma \ref{lem7.7}
and Holder's inequality, we can easily see Eq.(\ref{Eq:3}). Then applying the co-area formula (\ref{co}) with 
\[\left(H,M,p,g,f\right)=\left(H_\mathcal{P}\left(M\right),M,\Ep^{\mathcal{P}},\frac{1}{Z_\mathcal{P}^1}e^{-\frac{E}{2}},\left|\left(\tilde{g}^{\left(m\right)}-\tilde{g}\right)\left(\sigma\right)\right|^{\frac{d}{d-1}}\right),\]we have: 
\[
\int_{H_{\mathcal{P}}\left(M\right)}\left|\left(\tilde{g}^{\left(m\right)}-\tilde{g}\right)\left(\sigma\right)\right|^{\frac{d}{d-1}}d\nu_{\mathcal{P}}^{1}\left(\sigma\right)=\int_{M}\left|\left(g^{m}-g\right)\left(x\right)\right|^{\frac{d}{d-1}}h_{\mathcal{P}}\left(x\right)dx.
\]
Since $\cup_m supp\left(g^{(m)}\right)$ is contained in a compact set, Eq.(\ref{Eq:2}) can be proved using exactly the same argument as in Lemma \ref{the.1.3.5} with $g$ replaced by $g^{(m)}-g$ and letting $m\to \infty$.\end{proof}
\begin{lemma}\label{lem7.10}
For any $p\leq \frac{d}{d-1}$, $\sup_{\mathcal{P}}\mathbb{E}\left[\left|\tilde{g}\left(\phi\circ \beta_\mathcal{P}\right)\right|^p\right]<\infty$.
\end{lemma}
\begin{proof}
Since the law of $\phi\circ \beta_\mathcal{P}$ under $\nu$ is $\nu_\mathcal{P}^1$, we have
\begin{equation}
\mathbb{E}\left[\left|\tilde{g}\left(\phi\circ \beta_\mathcal{P}\right)\right|^p\right]=\int_{H_\mathcal{P}\left(M\right)}\left|\tilde{g}\right|^p\left(\sigma\right)d\nu_\mathcal{P}^1\left(\sigma\right).
\end{equation}
Then applying co-area formula (\ref{co}) exactly as Eq. (\ref{eq4}) we get
\begin{equation}
\int_{H_\mathcal{P}\left(M\right)}\left|\tilde{g}\right|^p\left(\sigma\right)d\nu_\mathcal{P}^1\left(\sigma\right)=\int_M\left|g\left(x\right)\right|^p\nu_{\mathcal{P},x}^1\left(H_{\mathcal{P},x}(M)\right)dx
\end{equation}
Using Proposition \ref{pro3.31}, note that $g$ has compact support, we have
\begin{align}
\sup_{\mathcal{P}}\int_M\left|g\left(x\right)\right|^p\nu_{\mathcal{P},x}^1\left(H_{\mathcal{P},x}(M)\right)dx\nonumber&\leq \left\Vert g\right\Vert_{L^{\frac{d}{d-1}}(M)} \left\Vert \sup_{\mathcal{P}}\nu_{\mathcal{P},x}^1\left(H_{\mathcal{P},x}(M)\right)\cdot 1_{supp(g)}(x)\right\Vert_{L^{\frac{d}{d-p(d-1)}}(M)} \\&\leq C\left\Vert g\right\Vert_{L^{\frac{d}{d-1}}(M)}.
\end{align}\end{proof}

\begin{theorem} \label{the.1.3.7} For any $f\in\mathcal{FC}_{b}^{1}$,
\[
\lim_{\left|\mathcal{P}\right|\to0}\int_{H_{\mathcal{P}}\left(M\right)}\tilde{g}\left(\sigma\right)\tilde{X}_{\mathcal{P}}^{tr,\nu_{\mathcal{P}}^{1}}f\left(\sigma\right)d\nu_{\mathcal{P}}^{1}\left(\sigma\right)=\int_{W_o\left(M\right)}\tilde{g}\left(\sigma\right)\tilde{X}^{tr,\nu}f\left(\sigma\right)d\nu\left(\sigma\right).
\]

\end{theorem}

\begin{proof} 
Since
\begin{equation*}
\int_{H_{\mathcal{P}}\left(M\right)}\tilde{g}\left(\sigma\right)\left(\tilde{X}_{\mathcal{P}}^{tr,\nu_{\mathcal{P}}^{1}}f\right)\left(\sigma\right)d\nu_{\mathcal{P}}^{1}\left(\sigma\right)=\mathbb{E}_{\nu}\left[\tilde{g}\cdot \left(\tilde{X}_{\mathcal{P}}^{tr,\nu_{\mathcal{P}}^{1}}f\right)\left(\phi\circ \beta_\mathcal{P}\right)\right]
\end{equation*}
and
\begin{equation*}
\int_{W_o\left(M\right)}\tilde{g}\left(\sigma\right)\left(\tilde{X}^{tr,\nu}f\right)\left(\sigma\right)d\nu\left(\sigma\right)=\mathbb{E}_{\nu}\left[\tilde{g}\cdot \tilde{X}^{tr,\nu}f\right],
\end{equation*}
we have
\begin{align*}
&\left|\int_{H_{\mathcal{P}}\left(M\right)}\tilde{g}\left(\sigma\right)\tilde{X}_{\mathcal{P}}^{tr,\nu_{\mathcal{P}}^{1}}f\left(\sigma\right)d\nu_{\mathcal{P}}^{1}\left(\sigma\right) - \int_{W_o\left(M\right)}\tilde{g}\left(\sigma\right)\tilde{X}^{tr,\nu}f\left(\sigma\right)d\nu\left(\sigma\right)\right|\\&\leq \mathbb{E}\left[\left|\tilde{g}\cdot \tilde{X}_{\mathcal{P}}^{tr,\nu_{\mathcal{P}}^{1}}f\left(\phi\circ \beta_\mathcal{P}\right)-\tilde{g}\cdot \tilde{X}^{tr,\nu}f\right|\right]\\&\leq \mathbb{E}\left[\left|\tilde{g}\left(\phi\circ \beta_\mathcal{P}\right)\right|\cdot\left|\tilde{X}_{\mathcal{P}}^{tr,\nu_{\mathcal{P}}^{1}}f\left(\phi\circ \beta_\mathcal{P}\right)-\tilde{X}^{tr,\nu}f\right|\right]+\mathbb{E}\left[\left|\tilde{g}\left(\phi\circ \beta_\mathcal{P}\right)-\tilde{g}\right|\cdot\left|\tilde{X}^{tr,\nu}f\right|\right].
\end{align*}
Choosing $p<\frac{d}{d-1}$ and using Holder's inequality, we have
\begin{align*}
 \mathbb{E}&\left[\left|\tilde{g}\left(\phi\circ \beta_\mathcal{P}\right)\right|\cdot\left|\tilde{X}_{\mathcal{P}}^{tr,\nu_{\mathcal{P}}^{1}}f\left(\phi\circ \beta_\mathcal{P}\right)-\tilde{X}^{tr,\nu}f\right|\right]\\&\leq \left\Vert\tilde{g}\left(\phi\circ \beta_\mathcal{P}\right)\right\Vert_{L^p(W_o(M))}\cdot \left\Vert\tilde{X}_{\mathcal{P}}^{tr,\nu_{\mathcal{P}}^{1}}f\left(\phi\circ \beta_\mathcal{P}\right)-\tilde{X}^{tr,\nu}f\right\Vert_{L^q(W_o(M))}\\&\leq \sup_{\mathcal{P}}\left\Vert\tilde{g}\left(\phi\circ \beta_\mathcal{P}\right)\right\Vert_{L^p(W_o(M))}\cdot \left\Vert\tilde{X}_{\mathcal{P}}^{tr,\nu_{\mathcal{P}}^{1}}f\left(\phi\circ \beta_\mathcal{P}\right)-\tilde{X}^{tr,\nu}f\right\Vert_{L^q(W_o(M))}.
\end{align*}
Then using Theorem \ref{the.8.4} and Lemma \ref{inf} we have 
\[\lim_{\left|\mathcal{P}\right|\to 0}\mathbb{E}\left[\left|\tilde{g}\left(\phi\circ \beta_\mathcal{P}\right)\right|\cdot\left|\tilde{X}_{\mathcal{P}}^{tr,\nu_{\mathcal{P}}^{1}}f\left(\phi\circ \beta_\mathcal{P}\right)-\tilde{X}^{tr,\nu}f\right|\right]=0.\]
Then because of Lemma \ref{inf}, it suffices to  find a $p\leq \frac{d}{d-1}$ such that 
\begin{equation}
\lim_{\left|\mathcal{P}\right|\to 0}\mathbb{E}_\nu\left[\left|\tilde{g}\left(\phi\circ \beta_\mathcal{P}\right)-\tilde{g}\right|^p\right]=0.\label{Eq:-1a}
\end{equation}
Since for any $\epsilon>0$, there exists a constant $C_{p,\epsilon}$ such that
 \[\left|\tilde{g}\left(\phi\circ \beta_\mathcal{P}\right)-\tilde{g}\right|^{p\left(1+\epsilon\right)}\leq C_{p,\epsilon}\left(\left|\tilde{g}\left(\phi\circ \beta_\mathcal{P}\right)\right|^{p\left(1+\epsilon\right)}+\left|\tilde{g}\right|^{p\left(1+\epsilon\right)}\right).\]
 We choose $p$ and $\epsilon$ such that $p\left(1+\epsilon\right)<\frac{d}{d-1}$. From Eq. (\ref{4}) we know $\mathbb{E}\left[\left|\tilde{g}\right|^{p\left(1+\epsilon\right)}\right]<\infty$. Then using Lemma \ref{lem7.10} we have
 \[\sup_\mathcal{P}\mathbb{E}_\nu\left[\left|\tilde{g}\left(\phi\circ \beta_\mathcal{P}\right)\right|^{p\left(1+\epsilon\right)}\right]<\infty.\]
 So $\sup_\mathcal{P}\mathbb{E}_\nu\left[\left|\tilde{g}\left(\phi\circ \beta_\mathcal{P}\right)-\tilde{g}\right|^{p\left(1+\epsilon\right)}\right]<\infty$ and thus
\[\left\{\left|\tilde{g}\left(\phi\circ \beta_\mathcal{P}\right)-\tilde{g}\right|^{p}\right\}\text{ is uniformly integrable under }\nu.\]
Then we want to show $\left|\tilde{g}\left(\phi\circ \beta_\mathcal{P}\right)-\tilde{g}\right|^{p}\to 0$ in probability.

Let $U_\mathcal{P}:=\left\{\sigma\in W_o\left(M\right): \Ep\circ{\Phi^{-1}}\circ \beta_\mathcal{P}\left(\sigma\right)=x\right\}$, since the law of ${\Phi^{-1}}\circ \beta_\mathcal{P}$ under $\nu$ is $\nu_{\mathcal{P}}^1$, recall from Lemma \ref{lem7.7} that 
$V_\mathcal{P}:=\left\{\sigma\in H_\mathcal{P}\left(M\right):\Ep^{\mathcal{P}}\left(\sigma\right)=x \right\}$ and $\nu_{\mathcal{P}}^1\left(V_\mathcal{P}\right)=0$, so 
\[\nu\left(U_\mathcal{P}\right)=\nu_{\mathcal{P}}^1\left(V_\mathcal{P}\right)=0.\]
From there we can construct a $\nu-$Null set;
\[N:=\cup_{\mathcal{P}}U_\mathcal{P}\cup\left\{\sigma\in W_o\left(M\right):\Ep\left(\sigma\right)=x\right\}.\]
Recall from Theorem 10 in \cite{Elworthy82} that 
\[\left|u_\mathcal{P}\left(1\right)-\tilde{u}\left(1\right)\right|\to 0\text{ in probability.}\]
Notice that $g\in C^{\infty}\left(M/\left\{x\right\}\right)$ and $\pi:\mathcal{O}\left(M\right)\to M$ is smooth, so excluding $N$, we have
\begin{equation}
\left|\tilde{g}\left(\phi\circ \beta_\mathcal{P}\right)-\tilde{g}\right|=\left|g\circ \pi(u_\mathcal{P}\left(1\right))-g\circ \pi\left(\tilde{u}\left(1\right)\right)\right|\to 0 \text{ in probability.}\label{Eq:-2}
\end{equation} 
Combining this with uniformly integrability we proved Eq. (\ref{Eq:-1a}). 
\end{proof}
%\begin{notation}\label{not7.11}
%Denote by $\mathcal{FC}^{m}_{1-}$ the subspace of $\mathcal{FC}^m$consisting of functions that are $\mathcal{F}_{s}$--measurable for some $s<1$. 
%\end{notation}
\begin{proposition}\label{prop7.12}
Let $f\in \mathcal{FC}_b^1$, then
\[
\lim_{m\rightarrow\infty}\int_{W_{o}\left(  M\right)  }\delta_{x}^{\left(
m\right)  }\left(  \Sigma_{1}\right)  fd\nu=\int_{W_{o}\left(  M\right)
}fd\nu_{x}.
\]
where
$\Sigma_{r}\left(  \sigma\right)  =\sigma\left(  r\right)  $ is the canonical
process on $W_{o}\left(  M\right)$.
\end{proposition}
\begin{proof}
Since $f=F\left(\Sigma_{s_1},\dots,\Sigma_{s_n}\right)$, we have by Markov property,
\[\int_{W_{o}\left(  M\right)  }\delta_{x}^{\left(
m\right)  }\left(  \Sigma_{1}\right)  fd\nu=\int_{M^n}\delta_{x}^{\left(
m\right)  }\left(x_n\right)F\left(x_1,\dots, x_n\right)\Pi_{j=1}^np_{\frac{1}{n}}\left(x_{j-1},x_j\right)dx_1\cdots dx_n.\]
Viewing $\int_{M^{n-1}}F\left(x_1,\dots, x_n\right)\Pi_{j=1}^np_{\frac{1}{n}}\left(x_{j-1},x_j\right)dx_1\cdots dx_{n-1}$ as a function of $x_n$, observe that it is uniformly integrable with respect to $x_n$, therefore it is a continuous function of $x_n$. Thus
\begin{align*}
\lim_{m\rightarrow\infty}\int_{W_{o}\left(  M\right)  }\delta_{x}^{\left(
m\right)  }\left(  \Sigma_{1}\right)  fd\nu&=\lim_{m\rightarrow\infty}\int_M\delta_{x}^{\left(
m\right)  }\left(x_n\right)\int_{M^{n-1}}F\left(x_1,\dots, x_n\right)\Pi_{j=1}^np_{\frac{1}{n}}\left(x_{j-1},x_j\right)dx_1\cdots dx_{n-1}\\&=\int_{M^{n-1}}F\left(x_1,\dots, x_{n-1},x\right)\Pi_{j=1}^{n-1}p_{\frac{1}{n}}\left(x_{j-1},x_j\right)\cdot p_{\frac{1}{n}}\left(x_{n-1},x\right)dx_1\cdots dx_{n-1}\\&=\int_{W_o\left(M\right)}fd\nu_x.
\end{align*}
\end{proof}
%\begin{notation}
%Given a set $X$ and $Q$ is a set of functions from $X$ to $\mathbb{R}$. Then $Q$ is said to be closed under bounded convergence iff for any sequence $\left\{f_n\right\}\subset Q$ such that 
%\begin{itemize}
%\item there exists a constant $N$ such that $\left|f_n\left(x\right)\right|\leq N$ for all $x\in X$ and $n$
%and
%\item for each $x\in X$, $f_n\left(x\right)\to f\left(x\right)$ for some $f$,
%\end{itemize}
%then $f\in Q$.
%\end{notation}

%\begin{lemma}\label{lem. la}
%Denote by $\bar{\mathcal{RFC}_b^1}$ the closure of $\mathcal{RFC}_b^1$ under %bounded convergence. Then $\bar{\mathcal{RFC}_b^1}$ contains all bounded %continuous functions on $W_o\left(M\right)$.
%\end{lemma}
%\begin{proof}
%It is easy to see that $\mathcal{RFC}_b^1$ is a multiplicative system, i.e. it is closed under pointwise multiplication. Also $\mathcal{RFC}_b^1$ contains constant function and generates $\mathcal{F}$, therefore by Dynkin's multiplicative system theorem, $\bar{\mathcal{RFC}_b^1}$ contains all bounded measurable functions on %$W_o\left(M\right)$.
%\end{proof}

\begin{proof}[Proof of Theorem \ref{thm2}] 
Recall from Remark \ref{rem3.1.2} that we can construct an approximate sequence of the delta mass $\delta_x$ on $M$:
\[
\delta_{x}^{(m)}:=g_{0}^{\left(m\right)}+\sum_{j=1}^{d}X_{j}g_{j}^{\left(m\right)}\in C_{0}^{\infty}\left(M\right)
\]
where $\left\{g_j^{\left(m\right)}:0 \leq j\leq d, m\geq 1\right\}\subset C_0^{\infty}\left(M\right)$ and $\left\{X_j:1\leq j\leq d\right\}\subset \Gamma\left(TM\right)$ with compact supports.
Consider their orthogonal lift on $H_{\mathcal{P}}(M)$ (referring to Theorem \ref{the.6.11}) as follows
\[
\tilde{\delta_{x}}^{(m)}:=\tilde{g_{0}}^{\left(m\right)}+\sum_{j=1}^{d}X_{\mathcal{P},j}\tilde{g_{j}}^{\left(m\right)}\in C^{\infty}\left(H_{\mathcal{P}}(M)\right)\label{Eq:-app}
\]
where $\tilde{g}\left(\sigma\right)=g\circ\Ep^{\mathcal{P}}\left(\sigma\right)$ for any $g\in C\left(M\right)$ and $X_{\mathcal{P},i}$ is the orthogonal lift of $X_i$ into $\Gamma\left(TH_\mathcal{P}\left(M\right)\right)$. 

For any $0\leq j\leq d$ (with the convention that $X_{\mathcal{P},0}=I$), using integration by parts, we get:
\begin{equation}
\int_{H_\mathcal{P}\left(M\right)}\left(\tilde{g}_{0}^{\left(m\right)}+\sum_{j=1}^{d}X_{\mathcal{P},j}\tilde{g}_{j}^{\left(m\right)}\right)fd\nu_\mathcal{P}^1=\int_{H_\mathcal{P}\left(M\right)}\left(\tilde{g}_{0}^{\left(m\right)}\cdot f+\sum_{j=1}^{d}X_{\mathcal{P},j}^{tr,\nu_\mathcal{P}^1}f\cdot\tilde{g}_{j}^{\left(m\right)}\right)d\nu_\mathcal{P}^1.\label{Eq:4}
\end{equation}
Now let $m\to \infty$, from Corollary \ref{Col3.2.2} we have:
\[\text{the left--hand side of } (\ref{Eq:4})=\int_{H_{\mathcal{P},x}\left(M\right)}fd\nu_{\mathcal{P},x}^1.\]
Applying Lemma \ref{the.1.3.5} to each pair $\left(\tilde{g}_{j}^{\left(m\right)}, X_{\mathcal{P},j}\right)$, we have:
\begin{equation}
\text{the right-hand side of } (\ref{Eq:4})=\int_{H_\mathcal{P}\left(M\right)}\left(\tilde{g_{0}}\cdot f+\sum_{j=1}^{d}X_{\mathcal{P},j}^{tr,\nu_\mathcal{P}^1}f\cdot\tilde{g_{j}}\right)d\nu_\mathcal{P}^1.
\end{equation}
Then let $\left|\mathcal{P}\right|\to 0$, from Theorem \ref{the.1.3.7}  we have:
\begin{equation}
\lim_{\left|\mathcal{P}\right|\to 0}\int_{H_{\mathcal{P},x}\left(M\right)}fd\nu_{\mathcal{P},x}^1=\int_{W_o\left(M\right)}\left(\tilde{g_{0}}\cdot f+\sum_{j=1}^{d}\tilde{X_{j}}^{tr,\nu}f\cdot\tilde{g_{j}}\right)d\nu.\label{m1}
\end{equation}
According to Lemma \ref{lem7.6}, 
\begin{align}
\int_{W_o\left(M\right)}&\left(\tilde{g_{0}}\cdot f+\sum_{j=1}^{d}\tilde{X_{j}}^{tr,\nu}f\cdot\tilde{g_{j}}\right)d\nu=\lim_{m\to \infty}\int_{W_o\left(M\right)}\left(\tilde{g}_{0}^{\left(m\right)}\cdot f+\sum_{j=1}^{d}\tilde{X_{j}}^{tr,\nu}f\cdot\tilde{g}_{j}^{\left(m\right)}\right)d\nu.
\end{align}
Then using integration by parts formula developed in Theorem $\ref{lem.5.1}$ we have:
\begin{align}
\int_{W_o\left(M\right)}\left(\tilde{g}_{0}^{\left(m\right)}\cdot f+\sum_{j=1}^{d}\tilde{X_{j}}^{tr,\nu}f\cdot\tilde{g}_{j}^{\left(m\right)}\right)d\nu&=\int_{W_o\left(M\right)}\left(\tilde{g}_{0}^{\left(m\right)}+\sum_{j=1}^{d}\tilde{X_{j}}\tilde{g}_{j}^{\left(m\right)}\right)\cdot f d\nu\nonumber\\&=\int_{W_o\left(M\right)}\tilde{\delta}_x^{\left(m\right)}fd\nu.\label{m2}
\end{align}
Lastly, using Proposition \ref{prop7.12} we have
\begin{equation}
\lim_{m\to \infty}\int_{W_o\left(M\right)}\tilde{\delta_x}^{\left(m\right)}fd\nu=\int_{W_o\left(M\right)}fd\nu_x.\label{m3}
\end{equation}
Combining Eq.(\ref{m1}) (\ref{m2}) and (\ref{m3}) we have
\begin{equation*}
\lim_{\left|\mathcal{P}\right|\to 0}\int_{H_{\mathcal{P},x}\left(M\right)}fd\nu_{\mathcal{P},x}^1=\int_{W_o\left(M\right)}fd\nu_x.
\end{equation*}
\end{proof}

%% APPENDIX
\appendix

\section{ODE estimates}\label{App.B}
\begin{lemma}
If X is a normal random variable with mean 0 and variance $t$, then 
\[
\mathbb{E}\left[e^{k\left|X\right|^2}\right]=\begin{cases}
\infty & \text{ if } k\geq \frac{1}{2t}\\
\left(1-2kt\right)^{-\frac{1}{2}}& \text{ if } k< \frac{1}{2t}\\
\end{cases}
\]
\end{lemma}
\begin{proof}
The result follows from the direct computation below.
\[\mathbb{E}\left[e^{k\left|X\right|^2}\right]=\int_{-\infty}^{\infty}e^{kx^2}\frac{1}{\sqrt{2\pi t}}e^{-\frac{x^2}{2t}}dx=\frac{1}{\sqrt{2\pi t}}\int_{-\infty}^{\infty}e^{\left(k-\frac{1}{2t}\right)x^2}dx.\]
If $k\geq \frac{1}{2t}$, then 
\[\frac{1}{\sqrt{2\pi t}}\int_{-\infty}^{\infty}e^{\left(k-\frac{1}{2t}\right)x^2}dx\geq \frac{1}{\sqrt{2\pi t}}\int_{-\infty}^{\infty}dx=\infty.\]
If $k< \frac{1}{2t}$, then
\[ \frac{1}{\sqrt{2\pi t}}\int_{-\infty}^{\infty}e^{\left(k-\frac{1}{2t}\right)x^2}dx=\frac{1}{\sqrt{2\pi t\left(\frac{1}{2t}-k\right)}}\int_{-\infty}^{\infty}e^{y^2}dy=\left(1-2kt\right)^{-\frac{1}{2}}.
\]
\end{proof}
\begin{lemma}\label{cBM}Let
	$\beta$ be a standard Brownian motion on $\mathbb{R}^d$, $\left\{s_i=\frac{i}{n}\right\}_{i=0}^n$ be an equally spaced partition of $\left[0,1\right]$ and $\Delta_i\beta$ be $\beta_{s_{i}}-\beta_{s_{i-1}}$, then for any $q>0$, we have
	\begin{equation}
	\underset{n:n>2q}{\sup}\mathbb{E}\left[e^{q\sum_{j=1}^{n}\left\vert \Delta_{j}\beta\right\vert ^{2}}\right]<\infty.\label{Gs}
	\end{equation}
\end{lemma}
\begin{proof}
	Since for each $j$, $\left\vert \Delta_{j}\beta\right\vert ^{2}=\sum_{l=1}^{d}\left\vert \left(\Delta_{j}\beta\right)_l\right\vert ^{2}$, where $\left\{\left(\Delta_{j}\beta\right)_l\right\}_{l=1}^d$ are coordinates of $\Delta_{j}\beta$, i.e. $\Delta_{j}\beta=\left(\left(\Delta_{j}\beta\right)_1,\dots,\left(\Delta_{j}\beta\right)_d\right)$. Since $\beta$ is a Brownian motion on $\mathbb{R}^d$, $\left\{\left(\Delta_{j}\beta\right)_l\right\}_{l=1}^d$ are i.i.d with Gaussian distribution of mean 0 and variance $\frac{1}{n}$. Using Lemma \ref{cBM} we have
	\[\mathbb{E}\left[e^{q\sum_{j=1}^{n}\left\vert \Delta_{j}\beta\right\vert ^{2}}\right]=\Pi_{j=1}^n\Pi_{l=1}^{d}\mathbb{E}\left[e^{q\left\vert \left(\Delta_{j}\beta\right)_l\right\vert ^{2}}\right]=\left(1-\frac{2q}{n}\right)^{-\frac{nd}{2}}.\]
	Then Eq.$\left(\ref{Gs}\right)$ follows since  $\left(1-\frac{2q}{n}\right)^{-\frac{nd}{2}}$ converges as $n\to \infty$.
	
\end{proof}
\begin{proposition} \label{prop A-1}

Consider an ODE: 
\[
Y^{\prime\prime}\left(s\right)=A\left(s\right)Y\left(s\right)
\]
where $Y\left(s\right),A\left(s\right)\in M_{n\times n}\left(\mathbb{R}\right)$ are real $n\times n$ matrices 
and $A\left(s\right)$ is positive semi-definite.

Denote by $\left\{ C\left(s\right),S\left(s\right)\right\} $the solutions
to this ODE with initial values: 
\[
C\left(0\right)=I,C^{\prime}\left(0\right)=0\text{ and }S\left(0\right)=0,S^{\prime}\left(0\right)=I
\]
Recall that in this paper we use $eig\left(X\right)$ to denote the set of eigenvalues of matrix $X$.
Then 
\begin{itemize}
\item $\text{If }\lambda\in eig\left(C\left(s\right)\right),\text{ then }\left|\lambda\right|\geq1$. 
\item $\text{If }\lambda\in eig\left(S\left(s\right)\right),\text{ then }\left|\lambda\right|\geq s$. 
\end{itemize}
\end{proposition}

\begin{proof} For all $v\in\mathbb{C}^{d}$, define $v\left(s\right):=C\left(s\right)v$,
then: 
\[
\left\langle v^{\prime\prime}\left(s\right),v\left(s\right)\right\rangle =\left\langle A\left(s\right)v\left(s\right),v\left(s\right)\right\rangle \geq0.
\]
Therefore, 
\[
\frac{d}{ds}\left\langle v^{\prime}\left(s\right),v\left(s\right)\right\rangle =\left\langle v^{\prime\prime}\left(s\right),v\left(s\right)\right\rangle +\left\Vert v^{\prime}\left(s\right)\right\Vert ^{2}\geq0.
\]
Since $\left\langle v^{\prime}\left(0\right),v\left(0\right)\right\rangle =0$,
so $\left\langle v^{\prime}\left(s\right),v\left(s\right)\right\rangle \geq0$.
Therefore 
\[
\frac{d}{ds}\left\Vert v\left(s\right)\right\Vert ^{2}=2Re\left\langle v^{\prime}\left(s\right),v\left(s\right)\right\rangle \geq0.
\]
Notice that $\left\Vert v\left(0\right)\right\Vert ^{2}=\left\Vert v\right\Vert ^{2}$,
so 
\[
\left\Vert v\left(s\right)\right\Vert ^{2}\geq\left\Vert v\right\Vert ^{2}.\]
Therefore if $\lambda\in eig\left(C\left(s\right)\right)$, choose $v\in \mathbb{C}^d$ to be an eigenvector associated to $\lambda$, then 
\[ 
\left\Vert \lambda v\right\Vert ^{2}=\left\Vert C\left(s\right)v\right\Vert ^{2}\geq \left\Vert v\right\Vert^{2}.
\]
So 
\[\left|\lambda\right|\geq 1.\]
Therefore $C\left(s\right)$ is invertible and 
\[
\left\Vert C\left(s\right)\right\Vert =\max_{\lambda\in eig\left(C\left(s\right)\right)}\left|\lambda\right|\geq1.
\]
A lower bound result for $\left\Vert S(s)v\right\Vert $ can be found
in \cite[Appendix E]{Laetsch2013}: 
\[
\left\Vert S(s)v\right\Vert \geq s\left\Vert v\right\Vert .
\]
From there it follows 
\[
\text{If }\lambda\in eig\left(S\left(s\right)\right),\text{ then }\left|\lambda\right|\geq s
\]
and $S\left(s\right)$ is invertible with 
\[
\left\Vert S\left(s\right)\right\Vert =\max_{\lambda\in eig\left(S\left(s\right)\right)}\left|\lambda\right|\geq s.
\]

\end{proof}

\begin{definition} Fix $\xi\in \mathbb{R}^d$, $\sigma\in H(M)$, denote $R_{u\left(s\right)}\left(\xi,\cdot\right)\xi$
by $A_{\xi}\left(s\right)$, and let $C_{\xi}\left(s\right),S_{\xi}\left(s\right)\in M_{d\times d}$ be the solutions to
$V^{\prime\prime}\left(s\right)= A_{\xi_{x}}(s)V\left(s\right)$ with initial values $C_\xi(0)=I, C^\prime_\xi(0)=0$ and $S_\xi(0)=0, S^\prime_\xi(0)=I$.

\end{definition}

\begin{proposition} \label{prop A-2}If $R$ is bounded by a constant
$N$, i.e. $\left|R\left(\xi,\cdot\right)\xi\right|\leq N\left|\xi\right|^{2},$
then 
\begin{equation}
\left|C_{\xi}\left(s\right)\right|\leq\cosh\left(\sqrt{N}\left\vert \xi\right\vert s\right)\leq e^{\frac{1}{2}N\left|\xi\right|^{2}s^{2}},\label{eq:-10}
\end{equation}
\begin{align}
\left\vert S_{\xi}(s)\right\vert  & \leq\sqrt{N}\left\vert \xi\right\vert s\frac{\sinh\left(\sqrt{N}\left\vert \xi\right\vert s\right)}{\sqrt{N}\left\vert \xi\right\vert s}\leq\cosh\left(\sqrt{N}\left\vert \xi\right\vert s\right)\sqrt{N}\left\vert \xi\right\vert s\leq\sqrt{N}\left\vert \xi\right\vert se^{\frac{1}{2}N\left\vert \xi\right\vert ^{2}s^{2}},\label{eq:-14}
\end{align}
\begin{equation}
\left\vert S_{\xi}\left(s\right)-sI\right\vert \leq\frac{N\left\vert \xi\right\vert ^{2}s^{3}}{6}e^{\frac{1}{2}N\left\vert \xi\right\vert ^{2}s^{2}},\label{eq:-15}
\end{equation}
and 
\begin{equation}
\left\vert C_{\xi}\left(s\right)-I\right\vert \leq\frac{N\left\vert \xi\right\vert ^{2}s^{2}}{2}e^{\frac{1}{2}N\left\vert \xi\right\vert ^{2}s^{2}}.\label{eq:-16}
\end{equation}

\end{proposition}

\begin{proof} \ref{eq:-10} and \ref{eq:-14} are quite elementary,
so here we only present the proof of \ref{eq:-15} and \ref{eq:-16}.

By Taylor's expansion, 
\[
S_{\xi}\left(s\right)=sI+\int_{0}^{s}R_{u_r}\left(\xi,S_{\xi}\left(r\right)\right)\xi\left(s-r\right)dr.
\]
\begin{align*}
\left\vert S_{\xi}\left(s\right)-sI\right\vert  & \leq N\left\vert \xi\right\vert ^{2}\int_{0}^{s}\left\vert S_{\xi}\left(r\right)\right\vert \left(s-r\right)dr\leq N\left\vert \xi\right\vert ^{2}\int_{0}^{s}\left[\left\vert S_{\xi}\left(r\right)-rI\right\vert +r\right]\left(s-r\right)dr
\end{align*}
Define $f\left(s\right):=\left\vert S_{\xi}\left(s\right)-sI\right\vert ,$
then we have: 
\[
f\left(s\right)\leq\int_{0}^{s}N\left\vert \xi\right\vert ^{2}\left(s-r\right)f\left(r\right)dr+N\left\vert \xi\right\vert ^{2}\frac{s^{3}}{6}
\]
By Gronwall's inequality: 
\[
f\left(s\right)\leq N\left\vert \text{\ensuremath{\xi}}\right\vert ^{2}\frac{s{}^{3}}{6}e^{\frac{1}{2}N\left\vert \xi\right\vert ^{2}s^{2}}
\]
Then we consider $C_{\xi}\left(s\right):$ 
\[
C_{\xi}\left(s\right)=I+\int_{0}^{s}R_{\tilde{u}_r}\left(\xi,C_{\xi}\left(r\right)\right)\xi\left(s-r\right)dr.
\]
So 
\begin{align*}
\left\vert C_{\xi}\left(s\right)-I\right\vert  & \leq N\left\vert \xi\right\vert ^{2}\int_{0}^{s}\left\vert C_{\xi}\left(r\right)\right\vert \left(s-r\right)dr\leq N\left\vert \xi\right\vert ^{2}\int_{0}^{s}\left[\left\vert C_{\xi}\left(r\right)-I\right\vert +1\right]\left(s-r\right)dr.
\end{align*}
Define $f\left(s\right):=\left\vert C_{\xi}\left(s\right)-I\right\vert ,$
then we have: 
\[
f\left(s\right)\leq\int_{0}^{s}N\left\vert \xi\right\vert ^{2}\left(s-r\right)f\left(r\right)dr+N\left\vert \xi\right\vert ^{2}\frac{s^{2}}{2}.
\]
By Gronwall's inequality: 
\[
f\left(s\right)\leq N\left\vert \xi\right\vert ^{2}\frac{s^{2}}{2}e^{\frac{1}{2}N\left\vert \xi\right\vert ^{2}s^{2}}.
\]
\end{proof}
\section{Matrix Analysis\label{app.C}}
\iffalse
Consider 
\[
a:=\left[\begin{array}{c}
a_{1}\\
a_{2}\\
\vdots\\
a_{n}
\end{array}\right]\in\mathbb{R}^{n}\text{ and }S=\left[\begin{array}{c}
I_{n\times n}\\
a^{\operatorname{tr}}
\end{array}\right]
\]
so that 
\[
S^{\operatorname{tr}}=\left[\begin{array}{cc}
I_{n\times n} & a\end{array}\right].
\]
Notice that $S$ is a $\left(n+1\right)\times n$ and $S^{\operatorname{tr}}$
is $n\times\left(n+1\right)$ matrix. For $x\in\mathbb{R}^{n}\ $and
$u\in\mathbb{R}$ we have 
\begin{align*}
S^{\operatorname{tr}}\left[\begin{array}{c}
x\\
u
\end{array}\right] & =x+ua\text{ and }Sx=\left[\begin{array}{c}
x\\
a\cdot x
\end{array}\right]\\
S^{\operatorname{tr}}Sx & =x+\left(a\cdot x\right)a=x+a~a^{\operatorname{tr}}x=\left(I+aa^{\operatorname{tr}}\right)x.
\end{align*}
Thus choosing an orthonormal basis $\left\{ u_{i}\right\} _{i=1}^{n}$
for $\mathbb{R}^{n}$ such that $u_{1}=\hat{a}$ we learn that 
\[
S^{\operatorname{tr}}Su_{1}=\left(1+\left\Vert a\right\Vert ^{2}\right)u_{1}\text{ and }S^{\operatorname{tr}}Su_{i}=u_{i}\text{ for }i>1.
\]
Thus it follows that $\det\left(S^{\operatorname{tr}}S\right)=1+\left\Vert a\right\Vert ^{2}.$
We record the higher dimensional generalization of the result above. It is used in computing some determinants in the dissertation.
\fi
\begin{theorem}\label{Conje1}
Suppose that $V$ is a finite
dimensional inner product space, $A:V^{n}\rightarrow V$ is a linear
map, and 
\[
S:=\left[\begin{array}{c}
I_{V^{n}\times V^{n}}\\
A
\end{array}\right]:V^{n}\rightarrow V^{n+1}.
\]
Then 
\[
\det\left[S^{\operatorname{tr}}S\right]=\det\left[I_{V}+AA^{\operatorname{tr}}\right].
\]
\end{theorem}
\begin{proof}
First observe that 
\[
S^{\operatorname{tr}}S=\left[\begin{array}{cc}
I & A^{\operatorname{tr}}\end{array}\right]\left[\begin{array}{c}
I\\
A
\end{array}\right]=I+A^{\operatorname{tr}}A.
\]
We denote $\dim(V)=d$ and let $\left\{ u_{j}\right\} _{j=1}^{d}\subset V$ be an orthonormal
basis of eigenvectors for $AA^{\operatorname{tr}}:V\rightarrow V$
so that $AA^{\operatorname{tr}}u_{j}=\lambda_{j}u_{j}$ and then let
$v_{j}:=A^{\operatorname{tr}}u_{j}.$ Then it follows that 
\[
A^{\operatorname{tr}}Av_{j}=A^{\operatorname{tr}}AA^{\operatorname{tr}}u_{j}=A^{\operatorname{tr}}\lambda_{j}u_{j}=\lambda_{j}A^{\operatorname{tr}}u_{j}=\lambda_{j}v_{j}.
\]
Now extend $\left\{ v_{j}\right\} _{j=1}^{d}$ to a basis for all
$V^{n}.$ From this we will find that $S^{\operatorname{tr}}S$ has
eigenvalues $\left\{ 1\right\} \cup\left\{ 1+\lambda_{j}\right\} _{j=1}^{d}$
and therefore 
\[
\det\left(S^{\operatorname{tr}}S\right)=\prod_{j=1}^{d}\left(1+\lambda_{j}\right)=\det\left(I+AA^{\operatorname{tr}}\right).
\]

\end{proof}

%% END MATTER
%\printindex %% Uncomment to display the index
%\nocite{}  %% Put any references that you want to include in the bib
%but haven't cited in the braces.
%% This is just my personal favorite style.
%There are many others.
%\setlength{\bibleftmargin}{0.25in}  % indent each item
%\setlength{\bibindent}{-\bibleftmargin}  % unindent the first line
%\def\baselinestretch{1.0}  % force single spacing
%\setlength{\bibitemsep}{0.16in}  % add extra space between items
\bibliographystyle{amsplain}
\bibliography{references}
%% This looks for the bibliography in template.bib
%which should be formatted as a bibtex file.
%and needs to be separately compiled into a bbl file.

\end{document}